\documentclass[oneside,english]{amsart}
\usepackage[T1]{fontenc}
\usepackage[latin9]{inputenc}
\usepackage{mathtools}
\usepackage{enumitem}
\usepackage{amstext}
\usepackage{amsthm}
\usepackage{amssymb}
\usepackage{xcolor}
\usepackage[normalem]{ulem}

\makeatletter
\numberwithin{equation}{section}
\theoremstyle{plain}
\newtheorem{thm}{\protect\theoremname}
\theoremstyle{definition}
\newtheorem{defn}{\protect\definitionname}[section]
\theoremstyle{plain}

\theoremstyle{plain}
\theoremstyle{remark}
\newtheorem*{rem*}{\protect\remarkname}
\theoremstyle{remark}
\newtheorem{rem}[defn]{\protect\remarkname}
\theoremstyle{definition}
\newtheorem{example}[defn]{\protect\examplename}
\theoremstyle{plain}
\newtheorem{lem}[defn]{\protect\lemmaname}
\theoremstyle{plain}
\newtheorem{cor}[defn]{Corollary}
\theoremstyle{plain}
\newtheorem{property}[defn]{Property}
\theoremstyle{plain}
\newtheorem*{lem*}{\protect\lemmaname}
\theoremstyle{remark}
\newtheorem*{claim*}{\protect\claimname}
\theoremstyle{plain}
\newtheorem{prop}[defn]{\protect\propositionname}
\newtheorem*{property*}{\protect\propertyname}
\newlist{casenv}{enumerate}{4}
\setlist[casenv]{leftmargin=*,align=left,widest={iiii}}
\setlist[casenv,1]{label={{\itshape\ \casename} \arabic*.},ref=\arabic*}
\setlist[casenv,2]{label={{\itshape\ \casename} \roman*.},ref=\roman*}
\setlist[casenv,3]{label={{\itshape\ \casename\ \alph*.}},ref=\alph*}
\setlist[casenv,4]{label={{\itshape\ \casename} \arabic*.},ref=\arabic*}

\usepackage{babel}
\providecommand{\lemmaname}{Lemma}
\providecommand{\propositionname}{Proposition}
\providecommand{\theoremname}{Theorem}

\makeatother

\usepackage{babel}
\providecommand{\casename}{Case}
\providecommand{\claimname}{Claim}
\providecommand{\definitionname}{Definition}
\providecommand{\examplename}{Example}
\providecommand{\factname}{Fact}
\providecommand{\lemmaname}{Lemma}
\providecommand{\propositionname}{Proposition}
\providecommand{\remarkname}{Remark}
\providecommand{\theoremname}{Theorem}

\begin{document}
\title{Loosely Bernoulli Odometer-Based Systems Whose Corresponding Circular
Systems Are Not Loosely Bernoulli}
\author{Marlies Gerber$^1$}   
\thanks{$^1$ Indiana University, Department of Mathematics, Bloomington, IN 47405, USA}
\author{Philipp Kunde$^2$} 
\thanks{$^2$ Pennsylvania State University, Department of Mathematics, State College, PA 16802, USA.
	P.K. acknowledges financial support from a DFG Forschungsstipendium under Grant No. 405305501.}
\maketitle

\begin{abstract}
    M.\ Foreman and B.\ Weiss \cite{FW3} obtained an anti-classification result for smooth ergodic diffeomorphisms, up to measure isomorphism,  by using  a functor $\mathcal{F}$ (see \cite{FW2}) mapping odometer-based systems, $\mathcal{OB}$, to circular systems, $\mathcal{CB}$. This functor transfers the classification problem from $\mathcal{OB}$ to $\mathcal{CB}$, and it preserves weakly mixing extensions, compact
extensions, factor maps, the rank-one property, and certain types
of isomorphisms \cite{FW2}. Thus it is natural to ask whether  
$\mathcal{F}$ preserves other dynamical properties. We show that $\mathcal{F}$ does 
\emph{not} preserve 
    the loosely Bernoulli property 
    by providing positive and zero entropy examples of loosely Bernoulli odometer-based systems whose corresponding circular systems are not loosely Bernoulli.
    We also construct a loosely Bernoulli circular system whose corresponding odometer-based system has zero entropy and is not loosely Bernoulli.
\end{abstract}

\insert\footins{\footnotesize - 
}

\section{Introduction}

An important development in ergodic theory that began in the late
1990's is the emergence of \emph{anti-classification} results for
measure-preserving transformations (MPT's) up to isomorphism. Here,
an MPT is a measure-preserving automorphism of a standard non-atomic
probability space, and two such MPT's, $T$ and $S,$ are said to
be isomorphic if there is a measure-preserving isomorphism between
the underlying probability spaces that intertwines the actions of
$T$ and $S.$ We denote by $X$ the set of MPT's on a fixed standard
non-atomic probability space $(\Omega,\mathcal{M},\mu)$, and let
the equivalence relation $\mathcal{R}\subset X\times X$ be defined
by $\mathcal{R:}=\{(T,S):T\text{\ and\ }S\text{\ are\ isomorphic}\}.$
We endow $X$ with the weak topology. (Recall that $T_{n}\rightarrow T$
in the weak topology if and only if $\mu\left(T_{n}(A)\triangle T(A)\right)\to0$
for every $A\in\mathcal{M}$.) The first anti-classification theorem
in ergodic theory is due to F.\ Beleznay and M.\ Foreman \cite{BF}, who
showed that a certain natural class of measure distal transformations
is not a Borel set in $X$. In the present context of isomorphism
of MPT's, the first result is due to G.\ Hjorth \cite{H}, who proved
that $\mathcal{R}$ is not a Borel subset of $X\times X.$ However,
this left open the question of what happens if we replace $X$ by
the subset $\widetilde{X}$ consisting of ergodic MPT's, with the
relative topology. Foreman, D.\ Rudolph, and B.\ Weiss \cite{FRW} proved
that the equivalence relation $\widetilde{R}:=\mathcal{R}\cap(\widetilde{X}\times\widetilde{X})$
is also not a Borel set. These results show that the problem of classifying
MPT's (or ergodic MPT's) up to isomorphism, which goes back to J.
von Neumann's 1932 paper \cite{Ne}, is inaccessible to countable methods
that use countable amounts of information. (See \cite{FW1}, \cite{FW2}, \cite{FW3}
for further discussion of this interpretation of these anti-classification
results.)

The most important positive results consist of P. Halmos and von Neumann's
classification of ergodic MPT's with pure point spectrum \cite{HN}
and the classification of Bernoulli shifts by their (metric) entropy
due to A. Kolmogorov \cite[Section 4.3]{KH95}, Ya. Sinai \cite{S}, and D. Ornstein \cite{Or}.
Yet many open questions remain. For example, the rank-one transformations,
        which have been studied extensively, form a dense $G_{\delta}$ subset of $X$,
and the restriction of the equivalence relation $\mathcal{R}$ to
rank-one transformations is Borel \cite{FRW}. However, there is still
no known classification of rank-one transformations up to isomorphism. 

In view of the anti-classification results mentioned above and, in
general, the difficulty of classifying ergodic MPT's, other versions
of the classification problem have been considered. One possibility
is to restrict the attempted classification to smooth ergodic diffeomorphisms
of a compact manifold $M$ with respect to a smooth measure $\mu.$
Except in dimensions one and two, there are no known obstructions
to realizing an arbitrary ergodic MPT as a diffeomorphism of a compact
manifold except the requirement, proved by A. Kushnirenko \cite{Ku},
that the ergodic MPT have finite entropy. Thus, it is not clear that
this restricted classification problem is any easier. Indeed, in this
context there is also an anti-classification result due to Foreman
and Weiss \cite{FW3}. They showed that if $X$ is replaced by the collection
$\widehat{X}$ of smooth Lebesgue-measure preserving diffeomorphisms
of the two-dimensional torus and $\widehat{X}$ is given the $C^{\infty}$
topology, then the equivalence relation $\widehat{R}$ consisting
of pairs of isomorphic elements of $\widehat{X}$ still fails to be
a Borel set. 

Another modification of the classification problem is to consider
Kakutani equivalence instead of isomorphism. Two ergodic MPT's are
said to be Kakutani equivalent if they are isomorphic to measurable
cross-sections of the same ergodic flow. It follows from Abramov's
formula, that two Kakutani-equivalent MPT's have the same
entropy type: zero entropy, finite entropy, or infinite entropy. Until
the work of A. Katok \cite{K75,K77} in the case of zero entropy, and J. Feldman
\cite{Fe} in the general case, no other restrictions were known for
achieving Kakutani equivalence. Ornstein, Rudolph, and Weiss \cite{ORW} showed, by building on the work of Feldman, that there are uncountably
many non-Kakutani equivalent ergodic MPT's of each entropy type. Thus
there is a rich variety of Kakutani equivalence classes, and classification
of ergodic MPT's up to Kakutani equivalence also remains an open problem.
It is not known whether anti-classification results analogous to those
in \cite{FRW} and \cite{FW3} can be obtained for Kakutani equivalence,
neither in the original setting of ergodic MPT's nor in the setting
of smooth diffeomorphisms preserving a smooth measure. 

In transferring the results of \cite{FRW} to the smooth setting, Foreman
and Weiss \cite{FW2} introduced a continuous functor $\mathcal{F}$ that maps
odometer-based systems to circular systems. (See Section \ref{sec:Odometer-Based-and-Circular} for definitions
of these terms.) According to an announcement in \cite{FW2}, any finite
entropy system that has an odometer factor can be represented as an
odometer-based system. It is a difficult open question whether any
transformation with a non-trivial odometer factor can be realized
as a smooth diffeomorphism on a compact manifold. Foreman and Weiss
\cite{FW2} were able to circumvent this difficulty by using the functor
$\mathcal{F}$ to transfer the classification problem for odometer-based
systems to circular systems. Under mild growth conditions on the parameters,
circular systems can be realized as smooth diffeomorphisms of the
two-dimensional torus using the Anosov-Katok method \cite{AK}. The
functor $\mathcal{F}$ preserves weakly mixing extensions, compact
extensions, factor maps, the rank-one property, and certain types
of isomorphisms, as well as numerous other properties (see \cite{FW2}).
While all circular systems have zero entropy, there exist positive-entropy
odometer-based systems, and thus $\mathcal{F}$ does not preserve
the entropy type. In connection with a possible Kakutani-equivalence
version of the results in \cite{FW3}, a natural question is whether
$\mathcal{F}$ preserves Kakutani equivalence (at least for zero-entropy
odometer-based systems). In particular, J.-P.\ Thouvenot asked whether
$\mathcal{F}$ maps loosely Bernoulli automorphisms (those Kakutani-equivalent
to an irrational rotation of the circle in case of zero entropy, or
those Kakutani-equivalent to a Bernoulli shift in case of positive
entropy) to loosely Bernoulli automorphisms. We provide examples to
show that the answer to both of these questions is `no'. We also obtain
an example which shows that $\mathcal{F}^{-1}$ fails to preserve
the loosely Bernoulli property. Our examples suggest that a different
approach may be needed for Kakutani-equivalence versions of anti-classification
results in the diffeomorphism setting. 

In Sections \ref{sec:positive entropy 1} and \ref{sec:positive entropy 2}, we give an example of a positive-entropy odometer-based
system $\mathbb{E}$ that is loosely Bernoulli, but the circular system
$\mathcal{F}(\mathbb{E})$ is not loosely Bernoulli. In this example,
$(n+1)$-blocks of the odometer-based system are constructed mostly
by independent concatenation of $n$-blocks. Because of this, $\mathbb{E}$
satisfies the positive-entropy version of the loosely Bernoulli property.
However, using techniques of A. Rothstein \cite{R}, we show that this
independent concatenation, when transferred to $\mathcal{F}(\mathbb{E}),$
causes the zero-entropy version of the loosely Bernoulli property
to fail.

The zero-entropy odometer-based system $\mathbb{K}$ constructed in
Sections \ref{sec:Feldman-patterns}--\ref{sec:Proof-of-Theorem}, is of greater interest in connection with \cite{FW2}, \cite{FW3},
because the anti-classification results of Foreman and Weiss can be
obtained by considering only zero-entropy odometer-based systems.
Our zero-entropy example is more difficult to construct than our positive-entropy
example, and it uses some delicate refinements of the methods in \cite{ORW}.
However, the heuristics of the construction can be described fairly
easily, as illustrated in Figure 2. This example is loosely Bernoulli,
but its image under $\mathcal{F}$ is not loosely Bernoulli. There
is also a simple example (see Example \ref{exa:lb}) of a zero-entropy odometer-based
system that is loosely Bernoulli and whose image under $\mathcal{F}$
is again loosely Bernoulli. This example, together with our example
$\mathbb{K},$ shows that $\mathcal{F}$ does not preserve Kakutani
equivalence. 

Finally, in Sections \ref{sec:OdomFeld}--\ref{sec:ConverseProof}, we give an example $\mathbb{M}$ of
a zero-entropy non-loosely Bernoulli odometer-based system whose corresponding
circular system is loosely Bernoulli. Figure \ref{fig:fig3} shows the idea for
this construction. Our Sections \ref{sec:Feldman-patterns}--\ref{sec:ConverseProof} with the zero-entropy odometer-based
systems can be read independently of Sections \ref{sec:positive entropy 1} and \ref{sec:positive entropy 2}.

Our results may also be of interest as another way that non-loosely
Bernoulli transformations arise naturally from loosely Bernoulli transformations.
Previous examples in this spirit include non-loosely Bernoulli Cartesian
products in which the factors are loosely Bernoulli \cite{ORW}, \cite{R1}, \cite{R2}, \cite{KR}, \cite{KW}. The functor $\mathcal{F}$ in \cite{FW2} changes the way $(n+1)$-blocks 
are built out of $n$-blocks according to a scheme that seems, upon
first consideration, likely to preserve the loosely Bernoulli property.
In this sense, our zero-entropy examples $\mathbb{K}$ and $\mathbb{M}$
were unexpected. 

\section{The $\overline{f}$ metric and the loosely Bernoulli property}
Feldman \cite{Fe} introduced a notion of distance, now called $\overline{f}$, between strings of symbols. He replaced the Hamming metric in Ornstein's very weak Bernoulli property \cite{Orbook} to define \emph{loosely Bernoulli} transformations (see Definitions \ref{def:generalLB} and \ref{def:zeroLB} below). A zero entropy version of this property was introduced independently by Katok \cite{K77}.
\begin{defn}
\label{def:fbar}A \emph{match} between two strings of symbols $a_{1}a_{2}\dots a_{n}$ and $b_{1}b_{2}\dots b_{m}$
from a given alphabet $\Sigma,$ is a collection $I$ of pairs of indices $(i_s,j_s)$, $s=1,\dots,r$ such that $1\le i_1 <i_2<\cdots <i_r\le n$, $1\le j_1 <j_2 <\cdots <j_r\le m$ and $a_{i_s}=b_{i_s}$ for $s=1,2,\dots,r.$ Then 
\begin{equation}\label{eq:cl}
\begin{array}{ll}
\overline{f}(a_{1}a_{2}\dots a_{n},b_{1}b_{2}\dots b_{m}) =\hfill \\ \displaystyle{1-\frac{2\sup\{|I|:I\text{\ is\ a\ match\ between\ }a_{1}a_{2}\cdots a_{n}\text{\ and\ }b_{1}b_{2}\cdots b_{m}\}}{n+m}.}
\end{array}
\end{equation} 
\end{defn}
We will refer to $\overline{f}(a_{1}a_{2}\cdots a_{n},b_{1}b_{2}\cdots b_{m})$
as the ``$\overline{f}$-distance'' between $a_{1}a_{2}\cdots a_{n}$
and $b_{1}b_{2}\cdots b_{m},$ even though $\overline{f}$ does not
satisfy the triangle inequality unless the strings are all of the
same length. A match $I$ is called a \emph{best possible match} if it realizes the supremum in the definition of $\overline{f}$. 

Suppose $(T,\mathcal{P})$ is a process, that is, $T$ is a measurable automorphism of a measurable space $(\Omega,\mathcal{M})$ and $\mathcal{P}=\{P_{\sigma}:\sigma\in \Sigma\}$ is a finite  measurable partition of $\Omega.$ For $x,y\in\Omega$ with $T^i(x)\in P_{a_i}$ and $T^i(y)\in P_{b_i}$ for $i=1,\dots,K$, we define $\overline{f}_{K}(x,y):=
\overline{f}(a_{1}a_{2}\dots a_{K},b_{1}b_{2}\dots b_{K}).$  
If $\nu$ and $\omega$ are probability measures on $(\Omega,\mathcal{M})$ then we say
$\overline{f}_{K}(\nu,\omega)<\epsilon$ if there is a
measure-preserving invertible map $\phi:(\Omega,\mathcal{M},\nu)\to(\Omega,\mathcal{M},\omega)$
such that there exists a set $G\subset\Omega$ with $\nu(G)>1-\varepsilon$
and $\overline{f}_{K}(x,\phi(x))<\varepsilon$ for all $x\in G.$

We now define loosely Bernoulli
in the general case (no assumptions on the entropy of the process).
\begin{defn}[Loosely Bernoulli in the general case]
\label{def:generalLB} A measure-preserving process
$(T,\mathcal{P},\nu)$ is \emph{loosely Bernoulli} if for every $\varepsilon>0,$
there exists a positive integer $K=K(\varepsilon)$ such that for
every positive integer $M$ the following holds: there exists a collection
$\mathcal{G}$ of ``good'' atoms of $\lor_{-M}^{0}T^{-i}\mathcal{P}$
whose union has measure greater than $1-\varepsilon,$ so that for
each pair $A,B$ of atoms in $\mathcal{G}$ of positive $\nu$-measure,
the measures $\nu_{A}$ and $\nu_{B}$ satisfy $\overline{f}_{K}(\nu_{A},\nu_{B})<\varepsilon.$
Here $\nu_{A}$ and $\nu_{B}$ denote the conditional measures on
$\Omega$ defined by $\nu_{A}(C)=\nu(C|A)=\nu(C\cap A)/\nu(A),$
and similarly for $\nu_{B}.$ 
\end{defn}

A measure-preserving transformation $(T,\nu)$ is \emph{loosely Bernoulli}
if $(T,\mathcal{P},\nu)$ is a loosely Bernoulli process for every
partition $\mathcal{P}.$ In fact, it suffices for $(T,\mathcal{P},\nu)$ to
be loosely Bernoulli for a generating partition $\mathcal{P}.$ 

As was pointed out in \cite{Fe}, Definition \ref{def:generalLB} is
equivalent to the definition obtained by replacing ``there exists
a positive integer $K=K(\varepsilon)"$ by ``for any sufficiently
large positive integer $K$ (how large depends on $\varepsilon)".$
Moreover, according to \cite[Corollary 2]{Fe}, ``every positive integer
$M$'' can be replaced by ``for every sufficiently large positive
integer $M$''. In fact, our Lemma \ref{lem:finer} implies \cite[Corollary 2]{Fe}, and the proof is similar. 

In the case of zero entropy,
no conditioning on the past is needed, and there
is a simpler definition of loosely Bernoulli. That is, the definition
reduces to the following.
\begin{defn}[Loosely Bernoulli in the case of zero entropy]
\label{def:zeroLB}
A measure-preserving process $(T,\mathcal{P},\nu)$ is \emph{zero-entropy loosely Bernoulli
}if for every $\varepsilon>0,$ there exists a positive integer $K=K(\varepsilon)$
and a collection $\mathcal{G}$ of ``good'' atoms of $\lor_{1}^{K}T^{-i}\mathcal{P}$
with total measure greater than $1-\varepsilon$ such that for each
pair $A,B$ of atoms in $\mathcal{G},$ $\overline{f}_{K}(x,y)<\varepsilon$
for $x\in A,$ $y\in B.$ 
\end{defn}

If this condition is satisfied, then routine estimates show that the
$(T,\mathcal{P},\nu)$ process indeed has zero entropy.

The following simple properties of $\overline{f}$, which were already
used in \cite{Fe} and \cite{ORW}, will appear frequently in our arguments.
These properties can be proved easily by considering the \emph{fit},
$1-\overline{f}(a,b),$ between two strings $a$ and $b$.

\begin{property}\label{property:omit_symbols}
Suppose $a$ and $b$ are strings of symbols
of length $n$ and $m,$ respectively, from an alphabet $\Sigma$.
If $\tilde{a}$ and $\tilde{b}$ are strings of symbols obtained by
deleting at most $\lfloor\gamma(n+m)\rfloor$ terms from $a$ and
$b$ altogether, 
where $0<\gamma<1$, then 
\begin{equation}
\overline{f}(a,b)\ge\overline{f}(\tilde{a},\tilde{b})-2\gamma.\label{eq:omit_symbols}
\end{equation}
Moreover, if there exists a best possible match between $a$ and $b$ such that no term that is deleted from $a$ and $b$ to form $\tilde{a}$ and $\tilde{b}$ is matched with a non-deleted term, then
\begin{equation}
\overline{f}(a,b)\ge\overline{f}(\tilde{a},\tilde{b})-\gamma.\label{eq:omit_symbols2}
\end{equation}
Likewise, if $\tilde{a}$ and $\tilde{b}$ are obtained by adding at most $\lfloor\gamma(n+m)\rfloor$ symbols to $a$ and $b$, then (\ref{eq:omit_symbols2}) holds. 
\end{property}

\begin{property}
\label{property:substring_matching}Suppose $x=x_{1}x_{2}\cdots x_{n}$ and
$y=y_{1}y_{2}\cdots y_{n}$ are decompositions of the strings of symbols
$x$ and $y$ into substrings such that there exists a best possible match between $x$ and $y$ where terms in $x_i$ are only matched with terms in $y_i$ (if they are matched with any term in $y$). Then 
\[
\overline{f}(x,y)=\sum_{i=1}^{n}\overline{f}(x_{i},y_{i})v_{i},
\]
where 
\begin{equation}
v_{i}=\frac{|x_{i}|+|y_{i}|}{|x|+|y|}.\label{eq:substring_matching}
\end{equation}
\end{property}

\begin{property}
\label{property:string_length}If $x$ and $y$ are strings of symbols
such that $\overline{f}(x,y)\le\gamma,$ for some $0\le\gamma<1,$
then 
\begin{equation}
\left(\frac{1-\gamma}{1+\gamma}\right)|x|\leq|y|\le\left(\frac{1+\gamma}{1-\gamma}\right)|x|.\label{eq:string_length}
\end{equation}
 We often use this property with $\gamma=1/7,$ in which case the conclusion
can be formulated as 
\begin{equation}
\frac{3|x|}{4}\le|y|\le\frac{4|x|}{3}.\label{eq:string_length_special_case}
\end{equation}
\end{property}

\section{\label{sec:Odometer-Based-and-Circular}Odometer-Based and Circular
Symbolic Systems}

In this section we review the notation and definitions for odometer-based
and circular symbolic systems. We also present
the functor $\mathcal{F}$ of \cite{FW2} between these two systems.

\subsection{Symbolic Systems}
An \emph{alphabet} is a countable or finite collection of symbols.
In the following, let $\Sigma$ be a finite alphabet endowed with
the discrete topology. Then $\Sigma^{\mathbb{Z}}$ with the product
topology is a separable, totally disconnected and compact space. The
shift 
\[
sh:\Sigma^{\mathbb{Z}}\to\Sigma^{\mathbb{Z}},\;sh(f)(n)=f(n+1)
\]
is a homeomorphism. If $\mu$ is a shift-invariant Borel measure,
then the measure-preserving dynamical system $\left(\Sigma^{\mathbb{Z}},\mathcal{B},\mu,sh\right)$
is called a \emph{symbolic system}. The closed support of $\mu$ is
a shift-invariant subset of $\Sigma^{\mathbb{Z}}$ called a \emph{symbolic
shift} or \emph{sub-shift}. 

Symbolic shifts are often described by giving a collection of words
that constitute a basis for the support of an invariant measure. A
word $w$ in $\Sigma$ is a finite sequence of elements of $\Sigma$,
and we denote its length by $\lvert w\rvert$. A \emph{language} (over
$\Sigma$) is a subset of the set of all words.
\begin{defn}
\label{def:ConstrSeq} A sequence of collection of words $\left(W_{n}\right)_{n\in\mathbb{N}}$,
where $\mathbb{N}=\left\{ 0,1,2,\dots\right\} $, satisfying the following
properties is called a \emph{construction sequence}:
\begin{enumerate}
\item for every $n\in\mathbb{N}$ all words in $W_{n}$ have the same length
$h_{n}$,
\item each $w\in W_{n}$ occurs at least once as a subword of each $w^{\prime}\in W_{n+1}$,
\item there is a summable sequence $\left(\varepsilon_{n}\right)_{n\in\mathbb{N}}$
of positive numbers such that for every $n\in\mathbb{N}$, every word
$w\in W_{n+1}$ can be uniquely parsed into segments $u_{0}w_{1}u_{1}w_{1}\dots w_{l}u_{l+1}$
such that each $w_{i}\in W_{n}$, each $u_{i}$ (called spacer or
boundary) is a word in $\Sigma$ of finite length and for this parsing
\[
\frac{\sum_{i=0}^{l+1}\lvert u_{i}\rvert}{h_{n+1}}<\varepsilon_{n+1}.
\]
\end{enumerate}
\end{defn}

We will often call words in $W_{n}$\emph{ $n$-words} or \emph{$n$-blocks},
while a general concatenation of symbols from $\Sigma$ is called
a \emph{string}. We also associate a symbolic shift with a construction
sequence: Let $\mathbb{K}$ be the collection of $x\in\Sigma^{\mathbb{Z}}$
such that every finite contiguous substring of $x$ occurs inside
some $w\in W_{n}$. Then $\mathbb{K}$ is a closed shift-invariant
subset of $\Sigma^{\mathbb{Z}}$ that is compact if $\Sigma$ is finite.
In order to be able to unambiguously parse elements of $\mathbb{K}$
we will use construction sequences consisting of uniquely readable
words.
\begin{defn}
Let $\Sigma$ be a language and $W$ be a collection of finite words
in $\Sigma$. Then $W$ is \emph{uniquely readable} iff whenever $u,v,w\in W$
and $uv=pws$ with $p$ and $s$ strings of symbols in $\Sigma$,
then either $p$ or $s$ is the empty word.
\end{defn}

Moreover, our $(n+1)$-words will be uniform in the $n$-words as
defined below. 
\begin{defn}
We call a construction sequence $\left(W_{n}\right)_{n\in\mathbb{N}}$
\emph{uniform} if for each $n\in\mathbb{N}$ there is a constant $c>0$
such that for all words $w^{\prime}\in W_{n+1}$ and $w\in W_{n}$
the number of occurrences of $w$ in $w^{\prime}$ is equal to $c$. 
\end{defn}

\begin{rem*}
In \cite{FW2} such construction sequences are called strongly uniform.
Since we will only deal with this strong notion of uniformity in this
paper, we abbreviate that terminology.
\end{rem*}
To check the loosely Bernoulli property for our symbolic systems we will use the following criterion from \cite{R} (stated there for processes constructed inductively by cutting and stacking such that all columns
of the $n$th tower have the same height, or equivalently, the associated
$n$-blocks all have same length).
\begin{lem}
\label{lem:Rothstein} Suppose $\mathbb{K}$ is a zero entropy symbolic system with uniform and uniquely readable construction sequence. Then $\mathbb{K}$ is loosely
Bernoulli if and only if for every $\varepsilon>0$ there exists $N$ such
that for $n\ge N,$ there is a set of $n$-blocks $\mathcal{G}_{n}$ with cardinality $\lvert \mathcal{G}_n \rvert > (1-\varepsilon)\lvert W_n \rvert$
such that for $A,B\in\mathcal{G}_{n},$
$\overline{f}(A,B)<\varepsilon.$ 
\end{lem}

\subsection{Odometer-Based Systems}

Let $\left(k_{n}\right)_{n\in\mathbb{N}}$ be a sequence of natural
numbers $k_{n}\geq2$ and 
\[
O=\prod_{n\in\mathbb{N}}\left(\mathbb{Z}/k_{n}\mathbb{Z}\right)
\]
be the $\left(k_{n}\right)_{n\in\mathbb{N}}$-adic integers. Then
$O$ has a compact abelian group structure and hence carries a Haar
measure $\lambda$. We define a transformation $T:O\to O$ to be addition
by $1$ in the $\left(k_{n}\right)_{n\in\mathbb{N}}$-adic integers
(i.e. the map that adds one in $\mathbb{Z}/k_{0}\mathbb{Z}$ and carries
right). Then $T$ is a $\lambda$-preserving invertible transformation
called \emph{odometer transformation} which is ergodic and has discrete
spectrum.

We now define the collection of symbolic systems that have odometer
systems as their timing mechanism to parse typical elements of the
system.
\begin{defn}
Let $\left(\mathtt{W}_{n}\right)_{n\in\mathbb{N}}$ be a uniquely
readable construction sequence with $\mathtt{W}_{0}=\Sigma$ and $\mathtt{W}_{n+1}\subseteq\left(\mathtt{W}_{n}\right)^{k_{n}}$
for every $n\in\mathbb{N}$. The associated symbolic shift will be
called an \emph{odometer-based system}.
\end{defn}

Thus, odometer-based systems are those built from construction sequences
$\left(\mathtt{W}_{n}\right)_{n\in\mathbb{N}}$ such that the words
in $\mathtt{W}_{n+1}$ are concatenations of a fixed number $k_{n}$
of words in $\mathtt{W}_{n}$. Hence, the words in $\mathtt{W}_{n}$
have length $\mathtt{h}_{n}$, where
\[
\mathtt{h}_{n}=\prod_{i=0}^{n-1}k_{i}
\]
 if $n>0$, and $\mathtt{h}_{0}=1$. Moreover, the spacers in part
3 of Definition \ref{def:ConstrSeq} are all the empty words (i.e.
an odometer-based transformation can be built by a cut-and-stack construction
using no spacers).

\subsection{\label{subsec:Circular-Systems}Circular Systems}

A \emph{circular coefficient sequence }is a sequence of pairs of integers
$\left(k_{n},l_{n}\right)_{n\in\mathbb{N}}$ such that $k_{n}\geq2$
and $\sum_{n\in\mathbb{N}}\frac{1}{l_{n}}<\infty.$ From these numbers
we inductively define numbers 
\[
q_{n+1}=k_{n}l_{n}q_{n}^{2}
\]
and 
\[
p_{n+1}=p_{n}k_{n}l_{n}q_{n}+1,
\]
where we set $p_{0}=0$ and $q_{0}=1$. Obviously, $p_{n+1}$ and
$q_{n+1}$ are relatively prime. Moreover, let $\Sigma$ be a non-empty
finite alphabet and $b,e$ be two additional symbols (called \emph{spacers}).
Then given a circular coefficient sequence $\left(k_{n},l_{n}\right)_{n\in\mathbb{N}}$
we build collections of words $\mathcal{W}_{n}$ in the alphabet $\Sigma\cup\{b,e\}$
by induction as follows:
\begin{itemize}
\item Set $\mathcal{W}_{0}=\Sigma$.
\item Having built $\mathcal{W}_{n}$ we choose a set $P_{n+1}\subseteq\left(\mathcal{W}_{n}\right)^{k_{n}}$
of so-called \emph{prewords }and form $\mathcal{W}_{n+1}$ by taking
all words of the form 
\[
\mathcal{C}_{n}\left(w_{0},w_{1},\dots,w_{k_{n}-1}\right)=\prod_{i=0}^{q_{n}-1}\prod_{j=0}^{k_{n}-1}\left(b^{q_{n}-j_{i}}w_{j}^{l_{n}-1}e^{j_{i}}\right)
\]
 with $w_{0}\dots w_{k_{n}-1}\in P_{n+1}$. If $n=0$ we take $j_0=0$, and for $n>0$ we let $j_i\in \{0,\dots,q_n-1\}$ be such that  
\[
j_{i}\equiv\left(p_{n}\right)^{-1}i\;\mod q_{n}.
\]
 We note that each word in $\mathcal{W}_{n+1}$ has length $k_{n}l_{n}q_{n}^{2}=q_{n+1}$.
\end{itemize}
\begin{defn}
A construction sequence $\left(\mathcal{W}_{n}\right)_{n\in\mathbb{N}}$
will be called \emph{circular }if it is built in this manner using
the $\mathcal{C}$-operators, a circular coefficient sequence and
each $P_{n+1}$ is uniquely readable in the alphabet with the words
from $\mathcal{W}_{n}$ as letters (this last property is called the
\emph{strong readability assumption}).
\end{defn}

\begin{rem}
By \cite[Lemma 45]{FW2} each $\mathcal{W}_{n}$ in a circular construction
sequence is uniquely readable even if the prewords are not uniquely
readable. However, the definition of a circular construction sequence
requires this stronger readability assumption.
\end{rem}

\begin{defn}
A symbolic shift $\mathbb{K}$ built from a circular construction
sequence is called a \emph{circular system}. For emphasis we will
often denote it by $\mathbb{K}^{c}$.
\end{defn}

For a word $w\in\mathcal{W}_{n+1}$ we introduce the following subscales
as in \cite[subsection 3.3]{FW2}: 
\begin{itemize}
\item Subscale $0$ is the scale of the individual powers of $w_{j}\in\mathcal{W}_{n}$
of the form $w_{j}^{l-1}$ and each such occurrence of a $w_{j}^{l-1}$
is called a \emph{$0$-subsection}.
\item Subscale $1$ is the scale of each term in the product $\prod_{j=0}^{k_{n}-1}\left(b^{q_{n}-j_{i}}w_{j}^{l_{n}-1}e^{j_{i}}\right)$
that has the form $\left(b^{q_{n}-j_{i}}w_{j}^{l_{n}-1}e^{j_{i}}\right)$
and these terms are called\emph{ $1$-subsections}.
\item Subscale $2$ is the scale of each term of $\prod_{i=0}^{q_{n}-1}\prod_{j=0}^{k_{n}-1}\left(b^{q_{n}-j_{i}}w_{j}^{l_{n}-1}e^{j_{i}}\right)$
that has the form $\prod_{j=0}^{k_{n}-1}\left(b^{q_{n}-j_{i}}w_{j}^{l_{n}-1}e^{j_{i}}\right)$
and these terms are called \emph{$2$-subsections}.
\end{itemize}

\subsection{The Functor $\mathcal{F}$}

For a fixed circular coefficient sequence $\left(k_{n},l_{n}\right)_{n\in\mathbb{N}}$
we consider two categories $\mathcal{OB}$ and $\mathcal{CB}$ whose
objects are odometer-based and circular systems respectively. The
morphisms in these categories are (synchronous and anti-synchronous)
graph joinings. In \cite{FW2} Foreman and Weiss define a functor taking
odometer-based systems to circular system that preserves the factor
and conjugacy structure. In this subsection we review the definition
of the functor from the odometer-based symbolic systems to the circular
symbolic systems.

For this purpose, we fix a circular coefficient sequence $\left(k_{n},l_{n}\right)_{n\in\mathbb{N}}$.
Let $\Sigma$ be an alphabet and $\left(\mathtt{W}_{n}\right)_{n\in\mathbb{N}}$
be a construction sequence for an odometer-based system with coefficients
$\left(k_{n}\right)_{n\in\mathbb{N}}$. Then we define a circular
construction sequence $\left(\mathcal{W}_{n}\right)_{n\in\mathbb{N}}$
and bijections $c_{n}:\mathtt{W}_{n}\to\mathcal{W}_{n}$ by induction:
\begin{itemize}
\item Let $\mathcal{W}_{0}=\Sigma$ and $c_{0}$ be the identity map.
\item Suppose that $\mathtt{W}_{n}$, $\mathcal{W}_{n}$ and $c_{n}$ have
already been defined. Then we define 
\[
\mathcal{W}_{n+1}=\left\{ \mathcal{C}_{n}\left(c_{n}\left(\mathtt{w}_{0}\right),c_{n}\left(\mathtt{w}_{1}\right),\dots,c_{n}\left(\mathtt{w}_{k_{n}-1}\right)\right)\::\:\mathtt{w}_{0}\mathtt{w}_{1}\dots\mathtt{w}_{k_{n}-1}\in\mathtt{W}_{n+1}\right\} 
\]
 and the map $c_{n+1}$ by setting 
\[
c_{n+1}\left(\mathtt{w}_{0}\mathtt{w}_{1}\dots\mathtt{w}_{k_{n}-1}\right)=\mathcal{C}_{n}\left(c_{n}\left(\mathtt{w}_{0}\right),c_{n}\left(\mathtt{w}_{1}\right),\dots,c_{n}\left(\mathtt{w}_{k_{n}-1}\right)\right).
\]
 In particular, the prewords are 
\[
P_{n+1}=\left\{ c_{n}\left(\mathtt{w}_{0}\right)c_{n}\left(\mathtt{w}_{1}\right)\dots c_{n}\left(\mathtt{w}_{k_{n}-1}\right)\::\:\mathtt{w}_{0}\mathtt{w}_{1}\dots\mathtt{w}_{k_{n}-1}\in\mathtt{W}_{n+1}\right\} .
\]
 
\end{itemize}
\begin{defn}
Suppose that $\mathbb{K}$ is built from a construction sequence $\left(\mathtt{W}_{n}\right)_{n\in\mathbb{N}}$
and $\mathbb{K}^{c}$ has the circular construction sequence $\left(\mathcal{W}_{n}\right)_{n\in\mathbb{N}}$
as constructed above. Then we define a map $\mathcal{F}$ from the
set of odometer-based systems (viewed as subshifts) to circular systems
(viewed as subshifts) by 
\[
\mathcal{F}\left(\mathbb{K}\right)=\mathbb{K}^{c}.
\]
 
\end{defn}

\begin{rem}
The map $\mathcal{F}$ is a bijection between odometer-based symbolic
systems with coefficients $\left(k_{n}\right)_{n\in\mathbb{N}}$ and
circular symbolic systems with coefficients $\left(k_{n},l_{n}\right)_{n\in\mathbb{N}}$
that preserves uniformity. Since the construction sequences for our
odometer-based systems will be uniquely readable, the corresponding
circular construction sequences will automatically satisfy the strong
readability assumption.
\end{rem}

In the following we will denote blocks in the odometer-based system
by letters in typewriter font (e.g. $\mathtt{A}$). For the corresponding
block in the circular system we will use calligraphic letters (e.g.
$\mathcal{A}$). As already noted the length of a $n$-block $\mathtt{w}$
in the odometer-based system is $\mathtt{h}_{n}=\prod_{i=0}^{n-1}k_{i}$,
if $n>0$, and $\mathtt{h}_{0}=1$, while the length of a $n$-block
in the circular system is $q_{n}$, i.e. $\lvert c_{n}\left(\mathtt{w}\right)\rvert=q_{n}$.
Moreover, we will use the following map from substrings of the underlying
odometer-based system to the circular system: 
\[
\mathcal{C}_{n,i}\left(\mathtt{w}_{s}\mathtt{w_{s+1}}\dots\mathtt{w}_{t}\right)=\prod_{j=s}^{t}\left(b^{q_{n}-j_{i}}\left(c_{n}\left(\mathtt{w}_{j}\right)\right)^{l_{n}-1}e^{j_{i}}\right)
\]
for any $0\leq i\leq q_{n}-1$ and $0\leq s\leq t\leq k_{n}-1$. 
\begin{example}
\label{exa:lb}We give an example of a loosely Bernoulli odometer-based
system $\mathbb{K}$ of zero measure-theoretic entropy with uniform
and uniquely readable construction sequence such that $\mathcal{F}\left(\mathbb{K}\right)$
is also loosely Bernoulli. For this purpose, let $\Sigma$ be an alphabet
with $2$ symbols and $\varepsilon_{n}\searrow0$. Assume that we
have two $n$-blocks $\mathtt{w}_{0}$ and $\mathtt{w}_{1}$ in the
odometer-based system. Then we define two $(n+1)$-blocks by the following
rule 
\begin{align*}
\mathtt{B}_{0}^{(n+1)}= & \mathtt{w}_{1}\mathtt{w}_{1}\underbrace{\mathtt{w}_{0}\mathtt{w}_{1}\mathtt{w}_{0}\mathtt{w}_{1}\dots\mathtt{w}_{0}\mathtt{w}_{1}}_{2s_{n} \text{ blocks}}\mathtt{w}_{0}\mathtt{w}_{0}\\
\mathtt{B}_{1}^{(n+1)}= & \mathtt{w}_{1}\mathtt{w}_{1}\mathtt{w}_{1}\underbrace{\mathtt{w}_{0}\mathtt{w}_{1}\dots\mathtt{w}_{0}\mathtt{w}_{1}}_{2s_{n}-2 \text{ blocks}}\mathtt{w}_{0}\mathtt{w}_{0}\mathtt{w}_{0}
\end{align*}
where we choose the integer $s_{n}$ sufficiently large to guarantee
that $\mathtt{B}_{0}^{(n+1)}$ and $\mathtt{B}_{1}^{(n+1)}$ are $\varepsilon_{n}$-close
to each other in $\overline{f}$. Clearly, the construction sequence defined
like this is uniform and uniquely readable. Moreover, the corresponding
$(n+1)$-blocks $\mathcal{B}_{0}^{(n+1)}$ and $\mathcal{B}_{1}^{(n+1)}$
in the circular system are also $\varepsilon_{n}$-close to each other
in $\overline{f}$. Hence, $\mathbb{K}$ and $\mathcal{F}\left(\mathbb{K}\right)$
are loosely Bernoulli by Lemma \ref{lem:Rothstein}.
\end{example}

\part{Positive entropy example}

\section{Positive Entropy Loosely Bernoulli Odometer-Based System} \label{sec:positive entropy 1}

In this section we construct a uniquely readable uniform odometer-based
system $\mathbb{E}$ of positive entropy that is loosely Bernoulli.
In the next section we will prove that $\mathcal{F}(\mathbb{E})$
is not loosely Bernoulli. The main idea in the construction of $\mathbb{E}$
is to concatenate $n$-blocks independently in long initial segments
of $(n+1)$-blocks, and use a relatively small final segment of the
$(n+1)$-block to achieve uniformity and unique readability. If we
used only the independent concatenation, then $\mathbb{E}$ would
be Bernoulli (and hence loosely Bernoulli), and $\mathcal{F}(\mathbb{E})$
would still be non-loosely Bernoulli. In this case the proof given
in the next section that $\mathcal{F}(\mathbb{E})$ is not loosely
Bernoulli could be simplified, but we want to achieve uniformity and
unique readability to make our example fit the framework of the odometer-based
constructions in \cite{FW2} and \cite{FW3}.

Our approach takes advantage of the fact that in the case of positive
entropy, in particular for the system $\mathbb{E}$, there will be
many $n$-blocks that are bounded apart in $\overline{f}$-distance
and the loosely Bernoulli property can still be satisfied. However,
all circular systems, as described in Section 2, and in particular
$\mathcal{F}(\mathbb{E}),$ have entropy zero. In this case the loosely
Bernoulli property fails to hold if most of the $n$-blocks are bounded
apart in the $\overline{f}$ metric. In the next section, we will
use the approach of A. Rothstein \cite{R} to prove that the independent
concatenation of $n$-blocks that are mostly bounded apart in $\overline{f}$
distance leads to $(n+1)$-blocks that are also mostly bounded apart
in $\overline{f}$ distance, and the lower bound on the $\overline{f}$
distance decreases only slightly in going from $n$-blocks to $(n+1)$-blocks. 

We begin by describing some of the conditions on the parameters involved
in the construction of $\mathbb{E}.$ Further requirements on the
lower bound on the $k_{n}$ will be imposed in the next section and
after the first lemma in the present section. First we choose positive
rational numbers $\varepsilon_{n}$ such that $\varepsilon_{n}<2^{-(n+12)}.$
Then we choose $k_{n}$ (depending on $\ell_{n},N(n),\varepsilon_{n})$
so that $\sum_{n=1}^{\infty}N(n)^{2}/(\varepsilon_{n}^{2}k_{n})<1/8.$
Furthermore, we require that $\varepsilon_{n}N(n)>2$, $\varepsilon_{n}k_{n}$
is an integer, and $k_{n}$ is a multiple of $N(n).$ Let $k_{n}'=(1-\varepsilon_{n})k_{n}.$

We now describe the construction sequence $(\mathtt{W}_{n})_{n\in\mathbb{N}}$
for $\mathbb{E}.$ Recall that $\mathtt{W}_{0}=\Sigma.$ Suppose there are
$N(n)$ distinct $n$-blocks of length $\mathtt{h}_{n}$ in $\mathtt{W}_{n},$
say $\mathtt{W}_{n}=\{\mathtt{y}_{1},\dots,\mathtt{y}_{N(n)}\}.$
Then $\mathtt{W}_{n+1}$ consists of all words of the form $\mathtt{w}_{1}\mathtt{w}_{2}\cdots\mathtt{w}_{k_{n}},$
where each $\mathtt{w}_{j}=\mathtt{y}_{i(j)}$ for some $i(j)\in\{1,\dots,N_{n}(n)\},$
subject to the following conditions:
\begin{enumerate}
\item For each $i\in\{1,2,\dots,N(n)-1\},$ card$\{j\in\{1,2,\dots,k_{n}'\}:\mathtt{w}_{j}=\mathtt{y}_{i}\}\le\frac{k_{n}}{N(n)},$
and $\mathtt{w}_{j}\ne\mathtt{y}_{N(n)}$ for $j\in\{1,2,\dots,k_{n}'$\}. 
\item For $j\in\{k_{n}-(k_{n}/N(n)),k_{n}-(k_{n}/N(n))+1,\dots,k_{n}-1\},$
$\mathtt{w}_{j}=\mathtt{y}_{N(n)}.$
\item The substring $\mathtt{w}_{(1-\varepsilon_{n})k_{n}}\cdots\mathtt{w}_{k_{n}-(k_{n}/N(n))-1}$
consists of a finite string of $\mathtt{y}_{1}$'s, followed by a
finite string of $\mathtt{y}_{2}$'s, etc. ending with a finite string
of $\mathtt{y}_{N(n)-1}$'s such that card$\{j\in\{1,2,\dots k_{n}\}:\mathtt{w}_{j}=\mathtt{y}_{i}\}=k_{n}/{N(n)}$
for every $i\in\{1,2,\dots,N(n)\}.$
\end{enumerate}
Whenever condition (1) on the first $k_{n}'$ $n$-blocks is satisfied,
then there is a unique way of completing the $(n+1)$-block so that
conditions (2) and (3) are satisfied. Condition (2) implies unique
readability, and condition (3)  implies uniformity. 
Our block construction and Lemma \ref{lem:Chebychev} below
are essentially a special case of the techniques in the Substitution Lemma in Section 8 of \cite{FRW}. 
\begin{lem}[Chebychev Application]
\label{lem:Chebychev} Suppose $k_{n}'=(1-\varepsilon_n)k_n$ symbols
are chosen independently from $\{1,2,\dots,,N(n)-1\},$ where each
symbol is equally likely to be chosen. Then the probability $\tau_{n}$
that there exists a symbol that is chosen more than $k_{n}/{N(n)}$
times satisfies $\tau_{n}<4N(n)^{2}/(\varepsilon_{n}^{2}k_{n}).$ 
\end{lem}

\begin{proof}
For a fixed $i_{0}\in\{1,\dots,N(n)-1\},$ let $S=S_{k_{n}'}$ be
the number of times $i_{0}$ is chosen in $k_{n}'$ Bernoulli trials,
where the probability of $i_{0}$ being chosen in any one trial is
$1/({N(n)-1}).$ Then the expected value of $S$ is $E(S)=k_{n}'/(N(n)-1)$
and the standard deviation of $S$ is $\sigma=\sqrt{k_{n}'(N(n)-2)}/({N(n)-1}).$
Note that the condition $\varepsilon_{n}N(n)>2$ implies that $1/{N(n)}>(1-(\varepsilon_{n}/2))/({N(n)-1}).$
We have the following estimates on the probabilities, where we apply
Chebychev's inequality in the last step:
\[
\begin{array}{cl}
\text{Pr}\Big(S>k_{n}/{N(n)}\Big) & \le\text{Pr}\left(|S-E(S)|>\frac{k_{n}}{N(n)}-\frac{(1-\varepsilon_{n})k_{n}}{N(n)-1}\right)\\
 & \le\text{Pr}\left(|S-E(S)|>\frac{(1-(\varepsilon_{n}/2))k_{n}}{N(n)-1}-\frac{(1-\varepsilon_{n})k_{n}}{N(n)-1}\right)\\
 & =\text{Pr}\left(|S-E(S)|>\frac{(\varepsilon_{n}/2)k_{n}}{N(n)-1}\right)\\
 & =\text{Pr}\left(|S-E(S)|>\alpha\sigma\right)\\
 & <1/\alpha^{2},
\end{array}
\]
where $\alpha=\varepsilon_{n}k_{n}/\big(2{\sqrt{k_{n}'(N(n)-2)}}\big).$
Thus $\text{Pr}\big(S>k_n/{N(n)}\big)<4N(n)/(\varepsilon_{n}^{2}k_{n}).$
Since $i_{0}$ was only one of $N(n)-1$ possible symbols, the upper
bound in the statement of the lemma is obtained by multiplying
the upper bound on $\text{Pr}\big(S>k_{n}/N(n)\big)$ by $N(n).$
\end{proof}
We require the $k_{n}$'s to be sufficiently large so that $\sum_{n=1}^{\infty}\tau_{n}^{(2^{-n})}<\infty.$
From Lemma \ref{lem:Chebychev}, this is possible because $k_{n}$
is chosen after $\varepsilon_{n}$ and $N(n)$ are determined.

We apply Lemma \ref{lem:Chebychev} to the independent choice of $k_{n}'$
$n$-blocks from the first $N(n)-1$ many $n$-blocks. Suppose $E$
is a string of $k_{n}'$ $n$-blocks chosen from the first $N(n)-1$
many $n$-blocks. If this string of $k_{n}'$ $n$-blocks satisfies
condition (1) above, then $E$ is a possible initial string of $k_{n}'$
$n$-blocks in an $(n+1)$-block. If $\text{Pr}_{n}(E)$ is the probability
of $E$ in the process $\mathbb{E},$ and $\widetilde{\text{Pr}}_{n}(E)$
is the probability of $E$ in the process that consists of concatenating $k_{n}'$  $n$-blocks chosen independently from 
the first $N(n)-1$ many $n$-blocks, then 
\begin{equation}
\text{Pr}_{n}(E)=\left(\frac{1}{1-\tau_{n}}\right)\widetilde{\text{Pr}}_{n}(E)\label{eq:tau_n estimate}
\end{equation}
 if $E$ is a possible initial string, and $\text{Pr}{}_{n}(E)=0$
if $E$ is not a possible initial string. Note that if we compare
the probability distributions $\text{Pr}_{n}$ and $\widetilde{\text{Pr}}_{n}$
on the collection $\mathcal{A}$ of all strings of $k_{n}'$ $n$-blocks,
and the collection $\mathcal{A}'$ of possible initial strings of
$k_{n}'$ $n$-blocks chosen from the first $N(n)-1$ many $n$-blocks, we obtain

\begin{equation}
\begin{split}
    \sum_{E\in\mathcal{A}}|\text{Pr}_{n}(E)-\widetilde{\text{Pr}}_{n}(E)|& =\tau_{n}+\sum_{E\text{\ensuremath{\in\mathcal{A}'}}}|\text{Pr}_{n}(E)-\widetilde{\text{Pr}}_{n}(E)| \\
    & =\tau_{n}+\left(\frac{1}{1-\tau_{n}}-1\right)\left(1-\tau_{n}\right)=2\tau_{n}.
\end{split}\label{eq:compare_distributions}
\end{equation}

For $i=1,2,\dots,|\Sigma|,$ let $P_{i}$ be the set of points in
$\Sigma^{\mathbb{Z}}$ with the symbol $i$ in position $0.$ Then
$\mathcal{P}:=\{P_{1},P_{2},\dots,P_{|\Sigma|}\}$ is a generating
partition for the odometer system. For integers $a$ and $b$ with
$a\le b,$ let $\mathcal{P}_{a}^{b}$ denote the partition of $\Sigma^{\mathbb{Z}}$ 
into sets with the same $\mathcal{P}$-name from time $a$ to time
$b.$ As before, we let $\mu$ denote the invariant measure on $\Sigma^{\mathbb{Z}}$
corresponding to the process $\mathbb{E}.$ We let $H=H$(sh,$\mathcal{P})$
denote the measure entropy of the left shift on $\Sigma^{\mathbb{Z}}$
with respect to $\mathcal{P},$ and we let $H_{{\rm top}}({\rm sh},\mathcal{P})$
denote the topological entropy. Note that uniformity and unique readability
imply that sh$:\Sigma^{\mathbb{Z}}\to\Sigma^{\mathbb{\mathbb{Z}}}$
is uniquely ergodic. Therefore, by the variational principle \cite[Theorem 4.5.3]{KH95}, $H=H_{{\rm top}}({\rm sh},\mathcal{P}).$ 
\begin{lem}\label{lem:positive entropy}
If $\mathbb{E}$ is the odometer-based system constructed above, then
the entropy of $\mathbb{E}$ is positive.
\end{lem}
\begin{proof}
The number of elements in $\mathcal{P}_{1}^{\mathtt{h}_{n}}$ is at least $N(n),$
while the number of elements in $\mathcal{P}_{1}^{k\mathtt{h}_{n}}$ is at
most $\mathtt{h}_{n}N(n)^{k+1}$. These estimates show that 
\[
H=\lim_{n\to\infty}\frac{\log N(n)}{\mathtt{h}_{n}}.
\]

Here $\mathtt{h}_{n}=\prod_{i=0}^{n-1}k_{i}$ and $N(n+1)=N(n)^{k_{n}'}(1-\tau_{n})>N(n)^{k_{n}'-1}.$
Thus 
\[
\frac{\log N(n+1)}{\mathtt{h}_{n+1}}\ge\frac{(k_{n}'-1)\log N(n)}{k_{n}\mathtt{h}_{n}}\ge\frac{(1-2\varepsilon_{n})\log N(n)}{\mathtt{h}_{n}}.
\]
 Therefore 
\[
\frac{\log N(n)}{\mathtt{h}_{n}}\ge(\log|\Sigma|)\prod_{i=0}^{n-1}(1-2\varepsilon_{i}).
\]
 Since $\prod_{i=0}^{\infty}(1-2\varepsilon_{i})$ converges to a
positive value, it follows that $H>0.$
\end{proof}
\begin{lem}[Conditioning Lemma]
\label{lem:conditioning} Suppose $\text{\rm{Pr}}$
and $\text{\rm{Pr}}'$ are two probability distributions defined on the
join $\mathcal{Q}\lor\mathcal{R}$ of two partitions $\mathcal{Q}$
and $\mathcal{R}$ of the same space. Suppose that for some $0<\varepsilon<1,$ 

\begin{equation}
\sum_{Q\in\mathcal{Q},R\in\mathcal{R}}|\text{\rm{Pr}}(Q\cap R)-\text{\rm{Pr}}'(Q\cap R)|<\varepsilon.\label{eq:close_dist}
\end{equation}
Also assume that $\text{\rm{Pr}}'(Q)>0$ whenever $\text{\rm{Pr}}(Q)>0$.
Then for all but $\text{\rm{Pr}}$ at most $\sqrt{\varepsilon}$ of the
$Q$'s in $\mathcal{Q},$ the conditional probabilities corresponding
to $\text{\rm{Pr}}$ and $\text{\rm{Pr}}'$ satisfy:

\[
\sum_{R\in\mathcal{R}}|\text{\rm{Pr}}(R|Q)-\text{\rm{Pr}}'(R|Q)|<2\sqrt{\varepsilon}.
\]
\end{lem}

\begin{proof}
It follows from (\ref{eq:close_dist}) that for all but  $\text{Pr }$ at most
$\sqrt{\varepsilon}$ of the $Q$'s in $\mathcal{Q},$
\begin{equation}
\sum_{R\in\mathcal{R}}|\text{Pr}(Q\cap R)-\text{Pr}'(Q\cap R)|<\sqrt{\varepsilon}\text{Pr}(Q).\label{eq:good_A}
\end{equation}
 Suppose $Q\in\mathcal{Q}$ is chosen so that (\ref{eq:good_A}) holds.
Then

\[
|\text{Pr}(Q)-\text{Pr}'(Q)|=\left|\sum_{R\in\mathcal{R}}[\text{Pr}(Q\cap R)-\text{Pr}'(Q\cap R)]\right\rvert <\sqrt{\varepsilon}\text{Pr}(Q).
\]
If we divide (\ref{eq:good_A}) by $\text{Pr}(Q)$, we obtain 
\begin{equation}
\sum_{R\in\mathcal{R}}\left|\frac{\text{Pr}(Q\cap R)}{\text{Pr}(Q)}-\frac{\text{Pr}'(Q\cap R)}{\text{Pr}(Q)}\right|<\sqrt{\varepsilon}.\label{eq:conditional1}
\end{equation}
We also have 
\begin{equation}
\begin{array}{cccccc}
\displaystyle{\sum_{R\in\mathcal{R}}\left|\frac{\text{Pr}'(Q\cap R)}{\text{Pr}(Q)}-\frac{\text{Pr}'(Q\cap R)}{\text{Pr}'(Q)}\right|} & = & \displaystyle{\sum_{R\in\mathcal{R}}\text{Pr}'(Q\cap R)\left|\frac{1}{\text{Pr}(Q)}-\frac{1}{\text{Pr}'(Q)}\right|}\\
 & = & \displaystyle{\text{Pr}'(Q)\frac{|\text{Pr}(Q)-\text{Pr}'(Q)|}{\text{Pr}(Q)\text{Pr}'(Q)}<\sqrt{\varepsilon}.}
\end{array}\label{eq:conditional2}
\end{equation}
The lemma now follows from (\ref{eq:conditional1}) and (\ref{eq:conditional2}).
\end{proof}

\begin{rem}\label{rem:less than 1} The above proof also holds in case $\text{Pr}'$ is a probability distribution on a larger space that contains
$\cup_{Q\in\mathcal{Q},R\in\mathcal{R}}$. That is, $\sum_{Q\in\mathcal{Q},R\in\mathcal{R}}{\text {Pr}'(Q\cap R)}$ can be less than $1$.  
\end{rem}
\begin{lem}[Finer Partitioning Lemma]
\label{lem:finer} Let $0<\varepsilon<1.$
Suppose $\mathcal{Q}$ is a refinement of $\mathcal{P}_{-M}^{0}$
and there is a probability measure $\omega$ on $\Sigma^{\mathbb{Z}}$
such that there is a collection $\tilde{\mathcal{G}}$ of ``good
atoms'' in $\mathcal{Q}$ with total $\mu$-measure greater than
$1-\varepsilon^{2}/16$ such that for $\tilde{Q}\in\tilde{\mathcal{G}}$
we have 
\[
\overline{f}_{K}\left(\mu(\cdot|\tilde{Q}),\mathcal{\omega}\right)<\varepsilon/4.
\]
 Then there is a collection $\mathcal{G}$ of ``good atoms'' in
$\mathcal{P}_{-M}^{0}$ with total $\mu$-measure greater than $1-\varepsilon/4$
such that for $Q\in\mathcal{G}$ we have 

\[
\overline{f}_{K}\left(\mu(\cdot|Q),\omega\right)<\varepsilon/2.
\]
 Consequently, 
\[
\overline{f}_{K}\left(\mu(\cdot|Q),\mu(\cdot|R)\right)<\varepsilon,
\]
 for $Q,R \in\mathcal{G}.$
\end{lem}

\begin{proof}
We let $\mathcal{G}$ consist of those atoms $Q$ in $\mathcal{P}_{-M}^{0}$
such that a subset of $Q$ of measure greater than $(1-\varepsilon/4)\mu(Q)$
is a union of atoms in $\tilde{\mathcal{G}}.$ 
\end{proof}
The following definition is due to D. Ornstein \cite{Or}.
\begin{defn}
\label{def:epsilon_partition}A partition $\mathcal{R}$ is said to
be $\varepsilon$-independent of a partition $\mathcal{Q}$ (with
respect to a given measure $\nu$) if for a collection of atoms $Q\in\mathcal{Q}$
of total $\nu$-measure at least $1-\varepsilon,$ 
\[
\sum_{R\in\mathcal{R}}|\nu(R|Q)-\nu(R)|\le\varepsilon.
\]
 In this case we write $\mathcal{R}\perp_{\nu}^{\varepsilon}\mathcal{Q}.$
If the measure $\nu$ is understood, we may omit the subscript $\nu.$ 
\end{defn}

\begin{rem}\label{rem:refinement}
If $\mathcal{Q},\mathcal{R},$ and $\mathcal{S}$ are partitions such that $\mathcal{R}$ refines $\mathcal{S}$ and $\mathcal{R}\perp_{\nu}^{\varepsilon}\mathcal{Q}$, then by the triangle inequality, $\mathcal{S}\perp_{\nu}^{\varepsilon}\mathcal{Q}.$ 
\end{rem}
\begin{rem}\label{rem:antisymmetry}
The definition of $\varepsilon$-independence is not symmetric in
$\mathcal{Q}$ and $\mathcal{R},$ but $\mathcal{R}\perp_{\nu}^{\varepsilon}\mathcal{Q}$
implies $\mathcal{Q}\perp_{\nu}^{\sqrt{3\varepsilon}}\mathcal{R}.$
(See p. 23 of \cite{Sm}.)
\end{rem}
\begin{lem}[Epsilon Independence Lemma]
\label{lem:epsilon} Let $\tau_{n}$
be as in the Chebychev Application, and let Pr$_{n}$ be the probability
distribution on possible initial strings of $k_{n}'$ many $n$-blocks
within $(n+1)$-blocks. Let $\mathcal{Q}$ and $\mathcal{R}$ be partitions
of the union of all $(n+1)$-blocks such that $\mathcal{Q}$ is the
partition into sets that have the same $a_{n}$ initial $n$-blocks
and $\mathcal{R}$ is the partition into sets that have the same $b_{n}$
many $n$-blocks appearing as the $(a_{n}+1)$st through $(a_{n}+b_{n})$th
$n$-blocks of an $(n+1)$-block. Assume that $a_{n}+b_{n}\le k_{n}'.$
Then $\mathcal{R}\perp_{\text{Pr}_{n}}^{3\sqrt{\tau_{n}}}\mathcal{Q}.$
\end{lem}

\begin{proof}
Let $\widetilde{\text{Pr}}_{n}$ be the probability distribution for
$k_{n}'$ $n$-blocks chosen independently from the first $N(n)-1$
many $n$-blocks, with each of these $n$-blocks equally likely. Then
by equation (\ref{eq:compare_distributions}), 

\[
\sum_{Q\cap R\in\mathcal{Q}\lor\mathcal{R}}|\text{Pr}_{n}(Q\cap R)-\widetilde{\text{Pr}}_{n}(Q\cap R)|\le\tau_{n}.
\]
 Note that we are only summing over those $Q\cap R$ that actually
occur as initial strings of some $(n+1)$-block(s). Therefore by Lemma
\ref{lem:conditioning} and Remark \ref{rem:less than 1}, for a collection $\mathcal{G}$ of $Q\in\mathcal{Q}$
of total $\text{Pr}_{n}$ measure at least $1-\sqrt{\tau_{n}},$ 
\[
\sum_{R\in\mathcal{R}}|\text{Pr}_{n}(R|Q)-\widetilde{\text{Pr}}_{n}(R|Q)|\le2\sqrt{\tau_{n}}.
\]
 But $\widetilde{\text{Pr}}_{n}(R|Q)=\widetilde{\text{Pr}}_{n}(R),$
and $\sum_{R\in\mathcal{R}}|\text{Pr}_{n}(R)-\widetilde{\text{Pr}}_{n}(R)|\le\tau_{n}.$
Therefore for $Q\in\mathcal{G},$ 
\[
\sum_{R\in\mathcal{R}}|\text{Pr}_{n}(R|Q)-\text{Pr}_{n}(R)|\le2\sqrt{\tau_{n}}+\tau_{n}<3\sqrt{\tau_{n}}.
\]
 
\end{proof}
\begin{rem*}
If $b_{n}=1,$ then $\text{Pr}_{n}(R)=\widetilde{\text{Pr}}_{n}(R)$
and $\mathcal{R}\perp_{\text{Pr}_{n}}^{2\sqrt{\tau_{n}}}\mathcal{Q}.$
\end{rem*}
The following lemma will be used in the inductive step of the proof
that the odometer-based system $\mathbb{E}$ is loosely Bernoulli. 
\begin{lem}[Inductive Step]
\label{lem:inductive_step} Suppose $\nu$ is a probability
measure, and $\mathcal{R}_{1}$, $\mathcal{R}_{2}$, $\mathcal{Q}_{1}$, $\mathcal{Q}_{2}$
are measurable partitions. If $\mathcal{R}_{1}\perp_{\nu}^{\varepsilon_{1}}\mathcal{Q}_{1}$
and $\mathcal{R}_{1}=\mathcal{Q}_{2}\lor\mathcal{R}_{2},$ where $\mathcal{R}_{2}\perp_{\nu}^{\varepsilon_{2}}\mathcal{Q}_{2}$ and $0<\varepsilon_1,\varepsilon_2<1$,
then $\mathcal{R}_{2}\perp_{\nu}^{2\sqrt{\varepsilon_{1}}+2\sqrt{\varepsilon_{2}}}\mathcal{Q}_{1}\lor\mathcal{Q}_{2}.$
\end{lem}

\begin{proof}
Since $\mathcal{R}_{2}\perp_{\nu}^{\varepsilon_{2}}\mathcal{Q}_{2},$
for a collection $\mathcal{G}_{2}$ of atoms $Q_{2}$ of $\mathcal{Q}_{2}$
with $\nu(\cup_{Q_{2}\in\mathcal{G}_{2}}Q_{2})\ge1-\varepsilon_{2},$
\begin{equation}
\sum_{R_{2}\in\mathcal{R}_{2}}|\nu(R_{2}|Q_{2})-\nu(R_{2})|\le\varepsilon_{2}.\label{eq:Inductive1}
\end{equation}
 For at least total $\nu$-measure $1-\sqrt{\varepsilon_{2}}$ of
the $Q_{1}$ in $\mathcal{Q}_{1},$ 
\begin{equation}
\nu\left(Q_{1}\cap\left(\cup_{Q_{2}\in\mathcal{G}_{2}}Q_{2}\right)\right)\ge\left(1-\sqrt{\varepsilon_{2}}\right)\nu(Q_{1}).\label{eq:Inductive2}
\end{equation}
 That is, $\mathcal{G}_{2}$ has total $\nu(\cdot|Q_{1})$ measure
at least $1-\sqrt{\varepsilon_{2}}.$ Since $\mathcal{R}_{1}=\mathcal{Q}_{2}\lor\mathcal{R}_{2}$
and $\mathcal{R}_{1}\perp_{\nu}^{\varepsilon_{1}}\mathcal{Q}_{1}$, for
a collection of atoms $Q_{1}\in\mathcal{Q}_{1}$ of total $\nu$-measure
at least $1-\varepsilon_{1},$ 
\begin{equation}
\sum_{Q_{2}\cap R_{2}\in\mathcal{Q}_{2}\lor\mathcal{R}_{2}}|\nu(Q_{2}\cap R_{2}|Q_{1})-\nu(Q_{2}\cap R_{2})|\le\varepsilon_{1}.\label{eq:Inductive3}
\end{equation}
 Let $\mathcal{G}_{1}$ be the collection of $Q_{1}\in\mathcal{Q}_{1}$
such that (\ref{eq:Inductive2}) and (\ref{eq:Inductive3}) hold.
Then we have $\nu\left(\cup_{Q_{1}\in\mathcal{G}_{1}}Q_{1}\right)\ge1-\varepsilon_{1}-\sqrt{\varepsilon_{2}}.$
Fix a choice of $Q_{1}\in\mathcal{G}_{1}.$ By Lemma \ref{lem:conditioning}
applied to $\text{Pr}=\nu(\cdot|Q_{1})$ and $\text{Pr}'=\nu,$ it
follows from (\ref{eq:Inductive3}) that for a collection $\mathcal{G}_{3}=\mathcal{G}_{3}(Q_{1})$
of atoms $Q_{2}\in\mathcal{Q}_{2}$ with total $\nu(\cdot|Q_{1})$
measure at least $1-\sqrt{\varepsilon_{1}},$ 
\begin{equation}
\sum_{R_{2}\in\mathcal{R}_{2}}|\nu(R_{2}|Q_{1}\cap Q_{2})-\nu(R_{2}|Q_{2})|\le2\sqrt{\varepsilon_{1}}.\label{eq:Inductive 4}
\end{equation}
 Then $\nu\left(\cup_{Q_{2}\in\mathcal{G}_{2}\cap\mathcal{G}_{3}}Q_{2}|Q_{1}\right)\ge1-\sqrt{\varepsilon_{1}}-\sqrt{\varepsilon_{2}},$
and for $Q_{2}\in\mathcal{G}_{2}\cap\mathcal{G}_{3},$ it follows
from (\ref{eq:Inductive1}) and (\ref{eq:Inductive 4}) that 
\begin{equation}
\sum_{R_{2}\in\mathcal{R}_{2}}|\nu(R_{2}|Q_{1}\cap Q_{2})-\nu(R_{2})|\le2\sqrt{\varepsilon_{1}}+\varepsilon_{2}.\label{eq:Inductive5}
\end{equation}
 Since $\nu\left(\cup_{Q_{1}\in\mathcal{G}_{1}}Q_{1}\right)\ge1-\varepsilon_{1}-\sqrt{\varepsilon_{2}}$
and for any $Q_{1}\in\mathcal{G}_{1},$ $\nu\left(\cup_{Q_{2}\in\mathcal{G}_{2}\cap\mathcal{G}_{3}}Q_{2}|Q_{1}\right)\ge1-\sqrt{\varepsilon_{1}}-\sqrt{\varepsilon_{2}},$
the collection $\mathcal{G}$ of $Q_{1}\cap Q_{2}\in\mathcal{Q}_{1}\lor\mathcal{Q}_{2}$
such that (\ref{eq:Inductive5}) holds has total $\nu$-measure at
least $(1-\varepsilon_{1}-\sqrt{\varepsilon_{2}})(1-\sqrt{\varepsilon_{1}}-\sqrt{\varepsilon_{2}})>1-2\sqrt{\varepsilon_{1}}-2\sqrt{\varepsilon_{2}}.$
Therefore $\mathcal{R}_{2}\perp_{\nu}^{2\sqrt{\varepsilon_{1}}+2\sqrt{\varepsilon_{2}}}\mathcal{Q}_{1}\lor\mathcal{Q}_{2}.$
\end{proof}
\begin{thm}
\label{thm:positive_entropy_LB}The odometer-based system $\mathbb{E}$
constructed in this section is loosely Bernoulli.
\end{thm}

\begin{proof}
Let $0<\varepsilon<1.$ Fix a choice of $n\ge2$ sufficiently large
so that $\sum_{j\ge n}\varepsilon_{j}<\varepsilon^{2}/100,$ and $\sum_{j\ge n}\tau_j^{(2^{-j})}<\varepsilon^{2}/100,$
and $(\varepsilon_{n}k_{n})^{-1}<\varepsilon^{2}/16.$ Let $K=(\varepsilon_{n}k_{n}+1)h_{n}$
and $M\ge h_{n+1}.$ We will show that the conclusion of Definition
\ref{def:generalLB} holds for these choices of $K$ and $M.$ Let
$\omega$ be any probability measure on $\Sigma^{\mathbb{Z}}$ such
that
\[
\omega\{(x_{k}):x_{1}x_{2}\cdots x_{\varepsilon_{n}k_{n}h_{n}}=c_{1}c_{2}\cdots c_{\varepsilon_{n}k_{n}h_{n}}\}=\text{Pr}_{n}(c_{1}c_{2}\cdots c_{\varepsilon_{n}k_{n}h_{n}}),
\]
which is the probability that the string $c_{1}c_{2}\cdots c_{\varepsilon_{n}k_{n}h_{n}}$
comprises the $\varepsilon_{n}k_{n}$ initial $n$-blocks in an $(n+1)$-block.
By Lemma \ref{lem:finer}, it suffices to show that for some refinement
$\mathcal{Q}$ of $\mathcal{P}_{-M}^{0},$ there is a collection $\mathcal{G}$
of good atoms of $\mathcal{Q}$ of total measure at least $1-\varepsilon^{2}/16$
such that for $Q\in\mathcal{G},$
\begin{equation}
\overline{f}_{K}(\mu(\cdot|Q),\omega)<\varepsilon/4.\label{eq:LBgoal}
\end{equation}
Fix a choice of $m$ such that $h_{n+m-1}\ge M.$ Let $\mathcal{S}$
be the partition of the space $\Sigma^{\mathbb{Z}}$ into sets that
have the same $(n+m)$-block structure, that is, time $0$ is in the
same position within the $(n+m)$-block.

We now describe the collection $\tilde{\mathcal{G}}$ of good atoms
in $\mathcal{S}.$ First we eliminate those atoms in $\mathcal{S}$
such that for any $j=2,\dots,m$, the deterministic part of the $(n+j)$-block
containing time $0$ overlaps with the time interval $[1,K].$ (The
deterministic part of an $(n+j)$-block consists of the last $\varepsilon_{n+j-1}k_{n+j-1}$ many
$(n+j-1)$-blocks within the $(n+j)$-block.) We also eliminate those
atoms in $\mathcal{S}$ such that time $0$ lies in the first $(n+m-1)$-block
within the $(n+m)$-block containing time $0.$ Moreover, we eliminate
those atoms in $\mathcal{S}$ such that time $0$ occurs in any of
the last $2\varepsilon_{n}k_{n}$ $n$-blocks within an $(n+1)$-block.
The total measure of the sets eliminated is less than $2(\varepsilon_{n}+\cdots+\varepsilon_{n+m-1})<\varepsilon^{2}/50.$
Let $\tilde{\mathcal{G}}$ be the collection of atoms in $\mathcal{S}$
that remain, and fix a choice $S\in\tilde{\mathcal{G}}.$ For this
$S$ and $j=0,\dots,m-1,$ let $a_{n+j}$ be the number of $(n+j)$-blocks
preceding the $(n+j)$-block containing time $0$ within the $(n+j+1)$-block
containing time $0.$ Since $a_{n+m-1}\ge1$ and $h_{n+m-1}\ge M,$
the beginning of the $(n+m)$-block containing time $0$ occurs at
or before time $-M.$ Let $\nu=\nu_{S}$ be the normalized restriction
of $\mu$ to $S.$

Let $\mathcal{Q}_{n}$ be the partition of $S$ into sets with the
same collection of $n$-blocks appearing in positions $1$ to $a_{n}+1$
at the beginning of the $(n+1)$-block containing time $0,$ and let
$\mathcal{R}_{n}$ be the partition of $S$ into sets with the same
collection of $n$-blocks appearing in the $\varepsilon_{n}k_{n}$
$n$-blocks that follow the $n$-block containing time $0.$ For $k=1,2,\dots,m-1,$
let $\mathcal{Q}_{n+m-k}$ be the partition of $S$ into sets with
the same collection of $(n+m-k)$-blocks comprising the $a_{n+m-k}$
$(n+m-k)$-blocks that precede the $(n+m-k)$-block that contains
time $0,$ and let $\mathcal{R}_{n+m-k}$ be the partition of $S$
into sets with $(n+m-k)$-blocks containing time $0$ agreeing up to the position of the last symbol in the $\varepsilon_nk_n$ $n$-blocks that follow the $n$-block containing time $0$  (see Figure \ref{fig:fig1}).  

\begin{figure}
    \centering
    \includegraphics[width=\textwidth]{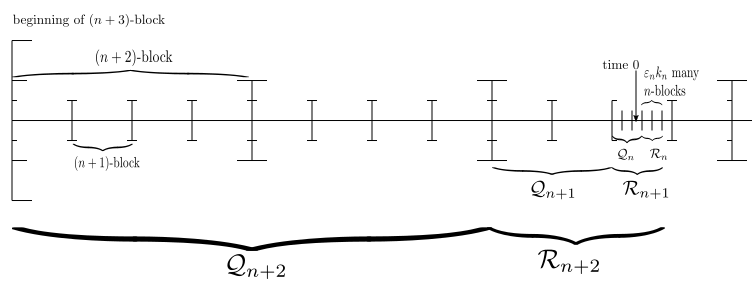}
    \caption{For $j=n,n+1,n+2$, the $\mathcal{Q}_j$'s and $\mathcal{R}_j$'s are partitions into sets according to the $\mathcal{P}$-names that appear in the indicated $j$-blocks. The actual numbers of $j$-blocks are much larger than can be depicted in the figure.}
    \label{fig:fig1}
\end{figure}
\begin{claim*}
Let $\eta_{n+m-k}=4(\tau_{n+m-1}^{2^{-(k+1)}}+\tau_{n+m-2}^{2^{-k}}+\cdots+\tau_{n+m-(k-1)}^{2^{-3}}+\tau_{n+m-k}^{2^{-2}})$
for $k=1,2,\dots,m.$ Then $\mathcal{R}_{n+m-k}\perp_{\nu}^{\eta_{n+m-k}}(\mathcal{Q}_{n+m-1}\lor\mathcal{Q}_{n+m-2}\lor\cdots\lor\mathcal{Q}_{n+m-k}),$
for $k=1,2,\dots,m.$ 
\end{claim*}
\begin{proof}
According to Lemma \ref{lem:epsilon} and Remark \ref{rem:refinement}, $\mathcal{R}_{n+m-1}\perp_{\nu}^{3\sqrt{\tau_{n+m-1}}}\mathcal{Q}_{n+m-1}.$
Since $3\sqrt{\tau_{n+m-1}}<4\tau_{n+m-1}^{2^{-2}}=\eta_{n+m-1},$
the claim holds for $k=1.$ Now suppose the claim holds for some $k=1,2,\dots,m-1,$
that is, $\mathcal{R}_{n+m-k}\perp_{\nu}^{\eta_{n+m-k}}(\mathcal{Q}_{n+m-1}\lor\mathcal{Q}_{n+m-2}\lor\cdots\lor\mathcal{Q}_{n+m-k}).$
We have $\mathcal{R}_{n+m-k}=\mathcal{Q}_{n+m-(k+1)}\lor\mathcal{R}_{n+m-(k+1)},$
and by Lemma \ref{lem:epsilon} and Remark \ref{rem:refinement}, $\mathcal{R}_{n+m-(k+1)}\perp_{\nu}^{3\sqrt{\tau_{n+m-(k+1)}}}\mathcal{Q}_{n+m-(k+1)}.$
Therefore, by Lemma \ref{lem:inductive_step}, $\mathcal{R}_{n+m-(k+1)}\perp_{\nu}^{\eta}(\mathcal{Q}_{n+m-1}\lor\mathcal{Q}_{n+m-2}\lor\cdots\lor\mathcal{Q}_{n+m-(k+1)}),$ where
\begin{align*}
\eta & =2\sqrt{\eta_{n+m-k}}+2\sqrt{3}\tau_{n+m-(k+1)}^{2^{-2}}\\
 & =4[\tau_{n+m-1}^{2^{-(k+1)}}+\tau_{n+m-2}^{2^{-k}}+\cdots+\tau_{n+m-(k+1)}^{2^{-3}}+\tau_{n+m-k}^{2^{-2}}]^{1/2}+2\sqrt{3}\tau_{n+m-(k+1)}^{2^{-2}}\\
 & \le4[\tau_{n+m-1}^{2^{-(k+2)}}+\tau_{n+m-2}^{2^{-(k+1)}}+\cdots+\tau_{n+m-(k+1)}^{2^{-4}}+\tau_{n+m-k}^{2^{-3}}+\tau_{n+m-(k+1)}^{2^{-2}}]\\
 & =\eta_{n+m-(k+1).}
\end{align*}
 This completes the inductive step. Therefore the claim holds. 
\end{proof}
Applying the claim with $k=m,$ we obtain $\mathcal{R}_{n}\perp_{\nu}^{\eta_{n}}(\mathcal{Q}_{n+m-1}\lor\cdots\lor\mathcal{Q}_{n}).$
Note that $\mathcal{Q}_{n+m-1}\lor\cdots\lor\mathcal{Q}_{n}$ is a
refinement of $S\cap\mathcal{P}_{-M}^{0}$ and $\mathcal{R}_{n}$
is the partition into the possible $\varepsilon_{n}k_{n}$ $n$-blocks
comprising a fraction $\varepsilon_{n}k_{n}h_{n}/K=\varepsilon_{n}k_{n}/(\varepsilon_{n}k_{n}+1)>1-\varepsilon^{2}/16$
of the $\mathcal{P}_{1}^{K}$ names. Since $\mathcal{R}_{n}\perp_{\nu}^{\eta_{n}}(\mathcal{Q}_{n+m-1}\lor\cdots\lor\mathcal{Q}_{n}),$
for a set $\mathcal{G}_{S}$ of atoms $Q\in\mathcal{Q}_{n+m-1}\lor\cdots\lor\mathcal{Q}_{n}$
of $\nu$-measure at least $1-\eta_{n},$ there is a measure-preserving
map $\phi_{S}:(S,\nu)\to(S\cap Q,\nu(\cdot|Q))$ such that on a set
of $\nu$-measure at least $1-\eta_{n},$ $\phi_{S}(x)$ is contained
in the intersection with $Q$ of that atom of $\mathcal{R}_{n}$ that
contains $x,$ and the $\mathcal{R}_{n}$ part of the $\mathcal{P}_{1}^{K}$
name of $x$ is the same as that of $\phi_{S}(x).$ Thus, $\overline{f}_{K}(\nu(\cdot|Q),\omega)<\varepsilon/4$
for $Q\in\mathcal{G}_{S}.$ Finally, we let $\mathcal{G}=\cup_{S\in\mathcal{\tilde{G}}}\mathcal{G}_{S}.$
Then the total $\mu$-measure of atoms in $\mathcal{G}$ is greater
than $1-\varepsilon^{2}/50-\eta_{n}>1-\varepsilon^{2}/16.$ Note that
for $Q\in\mathcal{G}_{S},$ $\nu_{S}(\cdot|Q)$ is the same as $\mu(\cdot|Q).$
Therefore, for $Q\in\mathcal{G}_{S},$ $\overline{f}_{K}(\mu(\cdot|Q),\omega)<\varepsilon/4.$
Then Lemma \ref{lem:finer} implies that Definition \ref{def:generalLB}
is satisfied. 
\end{proof}

\section{Non-Loosely Bernoulli Circular System Arising From Positive Entropy
Loosely Bernoulli Odometer-Based System} \label{sec:positive entropy 2}
\begin{thm}
\label{non-LB_Rothstein} If $\mathbb{E}$ is the positive entropy
loosely Bernoulli odometer-based system constructed in the previous
section, and $\mathcal{F}$ is the map from odometer-based systems
to circular systems defined in Section \ref{sec:Odometer-Based-and-Circular}, then $\mathcal{F}(\mathbb{E})$
is non-loosely Bernoulli. 
\end{thm}


We will prove $\mathcal{F}(\mathbb{E})$ is not loosely Bernoulli
by proving that the condition in Lemma \ref{lem:Rothstein} does
not hold. 

In the construction of $\mathcal{W}_{n+1}$ words in the circular system, we have many repetitions
of $\text{\ensuremath{\mathcal{W}_{n}} words.}$ To get lower bounds
on the $\overline{f}$ distance between $\mathcal{W}_{n+1}$ words,
we will make use of Definition \ref{def:cyclic names} and Lemma \ref{lem:repeated_substring}
below.
\begin{defn}
\label{def:cyclic names} If $b_{1}b_{2}\cdots b_{s}$ is any string
of $s$ symbols, let $\mathcal{T}(b_{1}b_{2}\cdots b_{s})$ denote
the collection of all finite consecutive substrings of $(b_{1}b_{2}\cdots b_{s})^{t}$
for any $t\ge1.$ 
\end{defn}

\begin{lem}[Repeated substring matching lemma]\label{lem:repeated_substring}
Suppose $a_{1}\cdots a_{r}$ and $b_{1}\cdots b_{s}$ are
strings of symbols. Then for any $\ell,\tilde{\ell}\ge1$ 
\[
\overline{f}\big((a_{1}\cdots a_{r})^{\ell},(b_{1}\cdots b_{s})^{\tilde{\ell}}\big)\ge\overline{f}\big(a_{1}\cdots a_{r},\mathcal{T}(b_{1}\cdots b_{s})\big),
\]
where the right hand side denotes the infimum of the $\overline{f}$
distance from $a_{1}\cdots a_{r}$ to any element of $\mathcal{T}(b_{1}\cdots b_{s}).$
\end{lem}

\begin{proof}
Apply Property \ref{property:substring_matching} with $x_{1}=x_{2}=\cdots=x_{\ell}=a_{1}a_{2}\cdots a_{r}.$
\end{proof}
The estimate below is proved in \cite{R} by a simple argument using
just the binomial theorem. A similar estimate can be obtained from
Stirling's formula.
\begin{lem}
\label{lem:binomial}If $m$ is a positive integer and $0<\sigma<1,$
then we have the following inequality for the binomial coefficient:
\[
\left(\begin{array}{c}
m\\
\lfloor\sigma m\rfloor
\end{array}\right)<2^{3m\sqrt{\sigma}}.
\]
\end{lem}

We make the following choices of parameters (in addition to those
already described in the previous section). Let $b_{n}=2^{-(n+10)},$
$\ell_{n}>2^{n+10.}$ Recall that $\varepsilon_{n}<2^{-(n+12)}.$
\begin{lem}[Inductive step in Rothstein's
argument to obtain a lower bound on the $\overline{f}$ distance]\label{lem:Rothstein_Inductive_Step}
For $n\in\mathbb{N}$ and $0<\delta_{n}<1$ define
\begin{equation}
\xi_{n}=\xi_{n}(\delta_{n}):=2^{8}(3\delta_{n})^{((b_{n}/2)-\varepsilon_{n})(1-\varepsilon_{n})}2^{3(1-\varepsilon_{n})\sqrt{(b_{n}/2)}}.\label{eq:xi_definition}
\end{equation}
Fix a particular $n\in\mathbb{N}$ and suppose $\delta_{n}$ is sufficiently
small that $\xi_{n}<1.$ Assume that for at least $(1-\delta_{n})$
of the $n$-blocks in $\mathcal{W}_{n}$ the $\overline{f}$ distance
from the $n$-block to any specific $n$-block or substrings of its
extensions (as in Definition \ref{def:cyclic names}) is greater than
$a_{n}$, where $0<a_{n}<1/3.$ Then for $k_{n}$ sufficiently large,
depending only on parameters with subscript $n,$ there exists $\delta_{n+1}$
such that $\xi_{n+1}=\xi_{n+1}(\delta_{n+1})<1,$ and for at least $(1-\delta_{n+1})$ of
the $(n+1)$-blocks in $\mathcal{W}_{n+1}$ the $\overline{f}$ distance
from the $(n+1)$-block to any specific $(n+1)$-block or substrings of
its extensions is greater than $a_{n+1}:=a_{n}(1-b_{n})-15\ell_{n}^{-1}.$
\end{lem}

\begin{proof}
Fix a particular choice 
\[
\prod_{i=0}^{q_{n}-1}\prod_{j=0}^{k_{n}-1}\left(b^{q_{n}-j_{i}}\mathcal{B}_{j}^{\ell_{n}-1}e^{j_{i}}\right)
\]
 of $(n+1)$-block in $\mathcal{W}_{n+1}$, where $\mathcal{B}_{0},\mathcal{B}_{1},\dots,\mathcal{B}_{k_{n-1}}$
are $n$-blocks in $\mathcal{W}_{n}$. Let $\mathcal{T}_0\coloneqq \mathcal{T}\left(\prod_{i=0}^{q_{n}-1}\prod_{j=0}^{k_{n}-1}\left(b^{q_{n}-j_{i}}\mathcal{B}_{j}^{\ell_{n}-1}e^{j_{i}}\right)\right)$. We will prove the inequality
\begin{equation}\label{eq:goal_lower_bound}
\overline{f}\left(\prod_{i=0}^{q_{n}-1}\prod_{j=0}^{k_{n}-1}\left(b^{q_{n}-j_{i}}\mathcal{A}_{j}^{\ell_{n}-1}e^{j_{i}}\right),\mathcal{T}_0\right)
> a_{n}(1-b_{n})-15\ell_{n}^{-1}
\end{equation}
holds for at least $(1-\delta_{n+1})$ of the $(n+1)$-blocks $\prod_{i=0}^{q_{n}-1}\prod_{j=0}^{k_{n}-1}\left(b^{q_{n}-j_{i}}\mathcal{A}_{j}^{\ell_{n}-1}e^{j_{i}}\right)$
in $\mathcal{W}_{n+1},$ where $\delta_{n+1}$ will be specified later
in the proof. The $b$'s and $e$'s that are newly added in constructing
$(n+1$)-blocks from $n$-blocks in the circular system make up a
fraction $\ell_{n}^{-1}$ of the symbols in any $(n+1)$-block.
We may assume that the smallest $\overline{f}$ distance between $\prod_{i=0}^{q_{n}-1}\prod_{j=0}^{k_{n}-1}\left(b^{q_{n}-j_{i}}\mathcal{A}_{j}^{\ell_{n}-1}e^{j_{i}}\right)$
and any element of $\mathcal{T}_0 $
occurs for an element of $\mathcal{T}_0$
of length at least $q^2_{n}k_{n}\ell_{n}/2,$ because otherwise it follows
from (\ref{eq:string_length}) that the $\overline{f}$
distance in (\ref{eq:goal_lower_bound}) is greater than $1/3.$ For
such an element of $\mathcal{T}_0,$
the number of newly added $b$'s and $e$'s is a fraction less than
$2\ell_{n}^{-1}$ of the length of that element. Therefore by Property
\ref{property:omit_symbols} and Lemma \ref{lem:repeated_substring},
to obtain (\ref{eq:goal_lower_bound}) it suffices to show that 
\begin{equation}
\overline{f}\big(\mathcal{A}_{1}^{\ell_{n}-1}\mathcal{A}_{2}^{\ell_{n}-1}\cdots \mathcal{A}_{k_{n}}^{\ell_{n}-1},\mathcal{T}(\mathcal{B}_{1}^{\ell_{n}-1}\mathcal{B}_{2}^{\ell_{n}-1}\cdots \mathcal{B}_{k_{n}}^{\ell_{n}-1})\big)>a_{n}(1-b_{n})-9\ell_{n}^{-1},\label{eq:new_goal_lower}
\end{equation}
without repeating the strings $q_{n}$ times. Suppose to the contrary
of (\ref{eq:new_goal_lower}) that 

\begin{equation}
\begin{split}
&\overline{f}\big(\mathcal{A}_{1}^{\ell_{n}-1}\mathcal{A}_{2}^{\ell_{n}-1}\cdots \mathcal{A}_{k_{n}}^{\ell_{n}-1},\mathcal{T}(\mathcal{B}_{1}^{\ell_{n}-1}\mathcal{B}_{2}^{\ell_{n}-1}\cdots \mathcal{B}_{k_{n}}^{\ell_{n}-1})\big) \\
\le & a_{n}(1-b_{n})-9\ell_{n}^{-1} <(a_{n}-9\ell_{n}^{-1})(1-b_{n})<1/3,\label{eq:contrary_goal}
\end{split}
\end{equation}
for some $k_{n}$ many $n$-blocks $\mathcal{A}_{1},\mathcal{A}_{2},\dots,\mathcal{A}_{k_{n}}.$
We will show that this happens for at most $\delta_{n+1}$ of the
$(n+1)$-blocks in $\mathcal{W}_{n+1}.$ Choose a match between the
$\mathcal{A}$-string and a $\mathcal{B}$-string that realizes the $\overline{f}$
distance in (\ref{eq:contrary_goal}). For each substring $\mathcal{A}_{i}^{\ell_{n}-1}$ of the $\mathcal{A}$-string,
let $f_{i}$ be the $\overline{f}$ distance between $\mathcal{A}_{i}^{\ell_{n}-1}$
and the corresponding part of the $\mathcal{B}$-string. Let $v_{i}$ be the
ratio of the number of symbols in $\mathcal{A}_{i}^{\ell_{n}-1}$ plus the number
of symbols in the corresponding part of the $\mathcal{B}$-string to the total
length of the $\mathcal{A}$- and $\mathcal{B}$-strings. Then by Property \ref{property:substring_matching},
the $\overline{f}$ distance for the entire strings is $\sum_{i=1}^{k_{n}}f_{i}v_{i},$
that is, a weighted average of the $f_{i}$ with weights $v_{i}.$
Since this weighted average is less than $(a_{n}-9\ell_{n}^{-1})(1-b_{n})$,
the weights $v_{i}$ for whose $f_{i}$ with $f_{i}<a_{n}-9\ell_{n}^{-1}<1/3$
must have sum at least $b_{n}.$ For these weights $v_{i},$ Property
\ref{property:string_length} and the assumptions that the $\overline{f}$
distance in (\ref{eq:contrary_goal}) and $f_{i}$ are both less than
$1/3$ imply that $v_{i}<2k_{n}^{-1}.$ Thus there must be at least
$(b_{n}/2)k_{n}$ indices $i$ such that $f_{i}<a_{n}-9\ell_{n}^{-1}.$
Then for at least a fraction $(b_{n}/2)-\varepsilon_{n}$ of the
indices $i\in\{1,2,\dots,k_{n}'\}$ the $\overline{f}$ distance
between $\mathcal{A}_{i}^{\ell_{n}-1}$ and the corresponding part of the $\mathcal{B}$-block
is less than $a_{n}-9\ell_{n}^{-1}.$ For each such $i$ , let $\sigma(i)$
be the first index such that the part of the $\mathcal{B}$-block corresponding
to $\mathcal{A}_{i}^{\ell_{n}-1}$ starts with a substring of a $\mathcal{B}_{\sigma(i)}^{\ell_{n}-1}.$
Then $\mathcal{A}_{i}^{\ell_{n}-1}$ may correspond just to a substring of $\mathcal{B}_{\sigma(i)}^{\ell_{n}-1}$
or to a substring of $\mathcal{B}_{\sigma(i)}^{\ell_{n}-1}\mathcal{B}_{\sigma(i)+1}^{\ell_{n}-1}$
or to a substring of $\mathcal{B}_{\sigma(i)}^{\ell_{n}-1}\mathcal{B}_{\sigma(i)+1}^{\ell_{n}-1}\mathcal{B}_{\sigma(i)+2}^{\ell_{n}-1}.$
Here the addition in the subscripts is modulo $k_{n}.$ Any correspondence
between $\mathcal{A}_{i}^{\ell_{n}-1}$ and strings of four or more
$\mathcal{B}_{j}^{\ell_{n}-1}$'s would lead to the $\overline{f}$
distance between $\mathcal{A}_{i}^{\ell_{n}-1}$ and the corresponding part
of the $\mathcal{B}$-string to be greater than $1/3,$ and therefore we may
disregard this possibility. The number of ways of choosing the $\sigma(i)$'s
and deciding whether to use just $\sigma(i)$ or to continue with
just $\sigma(i)+1$ or to continue with both $\sigma(i)+1$ and $\sigma(i)+2$
is at most $k_{n}\left(\begin{array}{c}
7k_{n}\\
k_{n}
\end{array}\right)$. Here we estimate $7k_{n}>3(2k_{n}+1),$ which is an upper bound
on the number of possible $\sigma(i)$ combinations corresponding
to $\mathcal{A}_{i}^{\ell_{n}-1}$ strings, allowing for the $\mathcal{B}$-string to
be up to twice the length of the $\mathcal{A}$-string (and thus contained in at most $2k_n+1$ 
consecutive $\mathcal{B}_j^{l_n-1}$ strings) and allowing for the
three choices: just $\sigma(i),$ just $\sigma(i)$ and $\sigma(i)+1,$
and all of $\sigma(i)$, $\sigma(i)$+1, and $\sigma(i)+2.$ The additional
factor of $k_{n}$ in front is due to being able to start the $\mathcal{B}$-string
with $\sigma(1)$ being any of $1,2,\dots,k_{n}.$ According to the
estimate on binomial coefficients given in Lemma \ref{lem:binomial},
$\left(\begin{array}{c}
7k_{n}\\
k_{n}
\end{array}\right)\le2^{21k_{n}\sqrt{1/7}}\le2^{8k_{n}}.$ When we match $\mathcal{A}_{i}^{l_{n}-1}$ with a substring of $\mathcal{B}_{\sigma(i)}^{\ell_{n}-1}\mathcal{B}_{\sigma(i)+1}^{\ell_{n}-1}\mathcal{B}_{\sigma(i)+2}^{\ell_{n}-1},$
we divide $\mathcal{A}_{i}^{\ell_{n}-1}$ into as many as three substrings according
to which part corresponds to each of  $\mathcal{B}_{\sigma(i)}^{\ell_{n}-1},$
$B_{\sigma(i)+1}^{\ell_{n}-1},$ or $\mathcal{B}_{\sigma(i)+2}^{\ell_{n}-1}.$
By removing at most a fraction $4(\ell_{n}-1)^{-1}$ of the symbols
in each $\mathcal{A}_{i}^{\ell_{n}-1}$ string, we may assume that full $\mathcal{A}_{i}$
strings correspond to each corresponding substring of $\mathcal{B}_{\sigma(i)}^{\ell_{n}-1},$
$\mathcal{B}_{\sigma(i)+1}^{\ell_{n}-1},$ and $\mathcal{B}_{\sigma(i)+2}^{\ell_{n}-1}.$
By Property \ref{property:omit_symbols} this removal will increase the $\overline{f}$
distance from $\mathcal{A}_{i}^{\ell_{n}-1}$ to the corresponding part of the
$\mathcal{B}$-string by at most $9\ell_{n}^{-1}.$ Thus by Lemma \ref{lem:repeated_substring},
for a fraction of at least $(b_{n}/2)-\varepsilon_{n}$ of the indices in $\{1,2,\dots,k_{n}'\},$
at least one of the three $\overline{f}$ distances from $\mathcal{A}_{i}$
to a string in $\mathcal{T}(\mathcal{B}_{\sigma(i)})$ or from $\mathcal{A}_{i}$ to a
string in $\mathcal{T}(\mathcal{B}_{\sigma(i)+1})$ or from $\mathcal{A}_{i}$ to a string
in $\mathcal{T}(\mathcal{B}_{\sigma(i)+2})$ must be less than $a_{n}.$ By
assumption, the probability that the $n$-block $\mathcal{A}_{i}$ satisfies
at least one of these three conditions is less than $3\delta_{n}.$
Thus if the first $k_{n}'$ $n$-blocks in the odometer $(n+1)$-block 
corresponding to the $\mathcal{A}$-string were selected independently, then the probability
$\delta_{n+1}$ that 
\[
\overline{f}\big(\mathcal{A}_{1}^{\ell_{n}-1}\mathcal{A}_{2}^{\ell_{n}-1}\cdots \mathcal{A}_{k_{n}}^{\ell_{n}-1},\mathcal{T}(\mathcal{B}_{1}^{\ell_{n}-1}\mathcal{B}_{2}^{\ell_{n}-1}\cdots \mathcal{B}_{k_{n}}^{\ell_{n}-1})\big)<a_{n}-9\ell_{n}^{-1}
\]
would be less than $k_{n}2^{8k_{n}}(3\delta_{n})^{((b_{n}/2)-\varepsilon_{n})k_{n}'}\left(\begin{array}{c}
k_{n}'\\
(b_{n}/2)k_{n}'
\end{array}\right).$ Since the selection of the $k_{n}'$ $n$-blocks is not quite independent,
we apply (\ref{eq:tau_n estimate}), and multiply our bound on the
probability by $(1-\tau_{n})^{-1}.$ Therefore from Lemma \ref{lem:binomial},
we obtain $\delta_{n+1}<k_{n}(1-\tau_{n})^{-1}\xi_{n}^{k_{n}}$ .
By assumption $\xi_{n}<1.$ Thus we can choose $k_{n}$ sufficiently
large so that $\delta_{n+1}$ is sufficiently small to imply $\xi_{n+1}<1.$
For at least $1-\delta_{n+1}$ of the $(n+1)$-blocks in $\mathcal{W}_{n+1}$
the $\overline{f}$ distance from the $(n+1)$-block to any specific
$(n+1)$-block or substrings of its extensions is greater than $a_{n+1}:=a_{n}(1-b_{n})-15\ell_{n}^{-1}.$ 
\end{proof}
\emph{Proof of Theorem \ref{non-LB_Rothstein}. }The inductive step
is contained in Lemma \ref{lem:Rothstein_Inductive_Step}. For the
base case, we recall that $0$-blocks are single symbols $1,2,\dots,|\Sigma|.$
Choose $0<\delta_{0}<1$ so that $\xi_{0}(\delta_{0})<1.$ Then require
$|\Sigma|$ to be sufficiently large that $\delta_{0}>|\Sigma|^{-1}.$
We let $a_{1}=1/4.$ According to the recursive formula for $a_{n},$
we have $a_{n}>1/8$ for all $n.$ Thus, if $\varepsilon=1/8$, the
condition in Lemma \ref{lem:Rothstein} is not satisfied. Therefore
$\mathcal{F}(\mathbb{E})$ is not loosely Bernoulli. 

\part{Zero-entropy example}

In this part of the paper we prove the following theorem, which gives
a zero-entropy version of the example constructed in Sections \ref{sec:positive entropy 1}
and \ref{sec:positive entropy 2}. 
\begin{thm}
\label{thm:zero-entropy} There exist circular coefficients $\left(l_{n}\right)$
and a loosely Bernoulli odometer-based system $\mathbb{K}$ of zero
measure-theoretic entropy with uniform and uniquely readable construction
sequence such that $\mathcal{F}\left(\mathbb{K}\right)$ is not loosely
Bernoulli.
\end{thm}

\subsection*{Outline of the proof of Theorem \ref{thm:zero-entropy}}

In Section \ref{sec:Proof-of-Theorem} we give a precise description
of the inductive building process of the uniform and uniquely readable
construction sequence for the odometer-based system $\mathbb{K}$
such that $\mathbb{K}$ will be loosely Bernoulli but $\mathcal{F}\left(\mathbb{K}\right)$
will be not loosely Bernoulli. The creation of this sequence relies
on two mechanisms. On the one hand, we will use what we will call
the \emph{Feldman mechanism }presented in Section \ref{sec:Feldman-Mechanism}.
It will allow us to produce an arbitrarily large number of blocks
that in the circular system remain almost as far apart in $\overline{f}$
as the building blocks. In particular, we can produce sufficiently
many blocks to apply the second mechanism, the so-called \emph{shifting
mechanism} introduced in Section \ref{sec:Shifting-Mechanism}. This
mechanism requires not only sufficiently many $n$-words to start
with but also a sufficiently large number of stages $p$. Then
we can produce $\left(n+p\right)$-words in our construction sequence
in a way that these are close to each other in the $\overline{f}$
metric and that the corresponding blocks in the circular construction
sequence stay apart from each other in the $\overline{f}$ metric.
Both mechanisms will make use of \emph{Feldman patterns} for which
we prove a general statement in Section \ref{sec:Feldman-patterns}.

\section{\label{sec:Feldman-patterns}Feldman patterns for blocks}

In \cite{Fe} Feldman constructed the first example of an ergodic
zero-entropy automorphism that is not loosely Bernoulli. The construction
is based on the observation that no pair of the following strings 
\begin{align*}
abababab\\
aabbaabb\\
aaaabbbb
\end{align*}
can be matched very well. We use his construction of blocks (that
we call \emph{Feldman patterns}) frequently in our two mechanisms
in Sections \ref{sec:Feldman-Mechanism} and \ref{sec:Shifting-Mechanism}. The basic Feldman patterns are displayed in Lemma \ref{lem:symbolic Feldman}. In applying these patterns, we substitute blocks of symbols for the individual symbols
to produce a large number of strings that are almost as far apart
in $\overline{f}$ as their building blocks. 

Our presentation of the Feldman patterns
is similar to the one in \cite{ORW}, but we apply the patterns in the odometer system and then examine the $\overline{f}$ distance between strings in the corresponding circular system. 
We also allow the consideration of different families of strings and a preliminary concatenation of blocks (which will prove useful when dealing with grouped blocks in Section \ref{sec:Shifting-Mechanism}). While we make a statement about substrings of different Feldman patterns
from either the same or different families in Proposition \ref{prop:Feldman},
we focus on the situation of the same Feldman pattern but different
families in Lemma \ref{lem:samePattern}. 

In order to obtain lower bounds on the $\overline{f}$ distance between strings that are built from blocks of symbols (as in Proposition \ref{prop:symbol by block replacement} and Corollary \ref{cor:alpha separated blocks}), it is convenient to introduce a notion of approximate $\overline{f}$ distance that we call $\tilde {f}$. 
\begin{defn}

\label{def:spproximate fbar} If $(i,j)$ and $(i',j')\in\text{\ensuremath{\mathbb{N}\times\mathbb{N}}}$,
then we define $(i,j)\preceq(i',j')$ if $i\le i'$ and $j\le j'$.
If $(i,j)\preceq(i',j')$ and $(i,j)\ne(i',j'),$ then we say $(i,j)\prec(i',j').$
An \emph{approximate match }between two strings of symbols $a_{1}a_{2}\cdots a_{n}$
and $b_{1}b_{2}\cdots b_{m}$ from a given alphabet $\Sigma$ is a
collection $\tilde{I}$ of pairs of indices $(i_{s},j_{s}),s=1,\dots,r,$
such that the following conditions hold:
\begin{itemize}
\item $(1,1)\preceq(i_{1},j_{1})\prec(i_{2},j_{2})\prec\cdots\prec(i_{r},j_{r})\preceq(n,m).$
\item $a_{i_{s}}=b_{j_{s}}$ for $s=1,2,\dots,r.$
\item If $(i,j)\in \tilde{I}$, then there exist $s,t\in \{1,\dots,r-2\}$ such that $\{s':(i,s')\in \tilde{I}\}\subset \{s,s+1,s+2\}$ and  $\{t':(t',j)\in \tilde{I}\}\subset \{t,t+1,t+2\}$. 
\end{itemize}
Then
\[
\begin{array}{ll}
\tilde{f}(a_{1}a_{2}\dots a_{n},b_{1}b_{2}\dots b_{m})=\\ \displaystyle{\max\left(0,1-\frac{2\sup\{|\tilde{I}|:\tilde{I}\text{\ is\ an\ approx.\ match\ between\ }a_{1}a_{2}\cdots a_{n}\text{\ and\ }b_{1}b_{2}\cdots b_{m}\}}{n+m}\right).}
\end{array}
\]
\end{defn}
Clearly $\overline{f}(a_{1}a_{2}\cdots a_{n},b_{1},b_{2}\cdots b_{m})\ge\tilde{f}(a_{1}a_{2}\cdots a_{n},b_{1},b_{2}\cdots b_{m}).$ Moreover, if every three consecutive symbols $a_{s},a_{s+1},a_{s+2}$ are distinct
and every three consecutive symbols $b_{s},b_{s+1},b_{s+2}$ are distinct,
then $\overline{f}(a_{1}a_{2}\cdots a_{n},b_{1},b_{2}\cdots b_{m})=\ \ \ \   $
$\tilde{f}(a_{1}a_{2}\cdots a_{n},b_{1},b_{2}\cdots b_{m}).$
Note that it is possible for $\tilde{f}(a_1a_2\cdots a_n,b_1b_2\cdots b_m)$ to be zero, even for $a_1a_2\cdots a_n\ne b_1b_2\cdots  b_m$; for example, $\tilde{f}(11000,11100)=0.$

\begin{lem}\label{lem:compare ftilde}
Suppose $0\le\varepsilon\le 1.$ If $\overline{f}(a_{1}a_{2}\cdots a_{n},b_{1}b_{2}\cdots b_{m})=1-\varepsilon,$
then $\ \ \ \ \ \ \ \ \ \ \ \ \ \ \ $ $\tilde{f}(a_{1}a_{2}\cdots a_{n},b_{1}b_{2},...b_{m})\ge1-3\varepsilon.$
\end{lem}

\begin{proof}The conclusion is trivially satisfied if $\varepsilon\ge\frac{1}{3}$. So we assume $0\le \varepsilon<\frac{1}{3}$.
Suppose $\tilde{I}=\{(i_1,j_1),\dots,(i_r,j_r)\}$ is an approximate match between $a_{1}a_{2}\cdots a_{n}$
and $b_{1}b_{2}\cdots b_{m}$, as in Definition \ref{def:spproximate fbar}. We construct a match $I$ with $|I|\ge |\tilde{I}|/3.$ Select the first element $(i_1,j_1)$ in $\tilde{I}$ as an element of $I$ and discard those at most two other elements of $\tilde{I}$ that have the same first coordinate
or the same second coordinate as $(i_{1},j_{1}).$ Then select
the next element $(i_{s},j_{s})$ in $\tilde{I}$ that has not already been selected for
$I$ or discarded. We again retain this $(i_{s},j_{s})$ for $I$
and discard those at most two other elements of $\tilde{I}$ that have
not been discarded previously and that have the same first coordinate
or the same second coordinate as $(i_{s},j_{s}).$ Continue in this
way until all elements of $\tilde{I}$ have either been discarded
or retained for $I.$ Then $|\tilde{I}|\le3|I|.$
\end{proof}
\begin{prop}[Symbol by block replacement]
\label{prop:symbol by block replacement}  Suppose $A_{a_{1}},A_{a_{2}},\dots,A_{a_{n}}$
and $B_{b_{1}},B_{b_{2}},\dots,B_{b_{m}}$ are blocks of symbols with
each block of length $L.$ Assume that $\alpha\in(0,\frac{1}{7})$, 
$\beta\in[0,\frac{1}{7})$, $\alpha\ge\beta$, $R>0$, and for all substrings
$C$ and $D$ consisting of consecutive symbols from $A_{a_{i}}$
and $B_{b_{j}}$, respectively, with $|C|,|D|\ge\frac{L}{R}$ we have
\[
\overline{f}(C,D)\ge\alpha\text{\ if\ }a_{i}\ne b_{j},
\]
 and
\[
\overline{f}(C,D)\ge\beta\text{\ if\ }a_{i}=b_{j}.
\]
 Let $\tilde{f}=\tilde{f}(a_{1}a_{2}\cdots a_{n},b_{1}b_{2}\cdots b_{m}).$
Then
\[
\overline{f}(A_{a_{1}}A_{a_{2}}\cdots A_{a_{n}},B_{b_{1}}B_{b_{2}}\cdots B_{b_{m}})>\alpha\tilde{f}+\beta(1-\tilde{f})-\frac{2}{R}\ge\alpha-(1-\tilde{f})+\beta(1-\tilde{f})-\frac{2}{R}.
\]
\end{prop}

\begin{proof}
We may assume that $n\le m.$ We decompose $A_{a_1}A_{a_2}\cdots A_{a_n}$ into the
substrings $A_{a_{1}},A_{a_{2}},\dots,A_{a_{n}}$ and decompose $B_{b_1}B_{b_2}\cdots B_{b_m}$ into corresponding substrings $\tilde{B}_{1},\tilde{B}_{2},\dots,\tilde{B}_{n}$
according to a best possible match between the two strings. Then we
decompose each $A_{a_{i}}$ into at most three further substrings
$A_{a_{i},0},A_{a_{i},1},A_{a_{i},2}$ corresponding to substrings
$\ B_{i,b_{j_{i}}},B_{i,b_{j_{i}+1}},B_{i,b_{j_{i}+2}}\ $ of $\ \tilde{B}_{i}\ $
that lie entirely in $B_{b_{j_{i}}},B_{b_{j_{i}+1}},B_{b_{j_{i}+2}},$
respectively, to obtain a best possible match between $A_{a_i}$ and $\tilde{B}_i$ We will apply Property \ref{property:substring_matching} to this decomposition. We may
ignore any $A_{a_{i}}$ and corresponding $\tilde{B}_{i}$ for which
$\tilde{B}_{i}$ fails to lie in three consecutive blocks $B_{b_{j_{i}}},B_{b_{j_{i}+1}},B_{b_{j_{i}+2}},$
because in this case it follows from equation (\ref{eq:string_length_special_case}) in Property \ref{property:string_length},
that $\overline{f}(A_{a_{i}},\tilde{B}_{i})>\frac{1}{7}>\alpha.$
For the same reason we may also ignore any $B_{b_{j}}$ whose corresponding
substring in $A_{a_{1}}A_{a_{2}}\cdots A_{a_{n}}$ fails to lie in
three consecutive blocks $A_{a_{i}},A_{a_{i+1}},A_{a_{i+2}}.$ We
let $\tilde{I}$ consist of those remaining pairs $(i',j')\in\cup_{i=1}^{n}\{(i,j_{i}),(i,j_{i}+1),(i,j_{i}+2)\}$
such that $a_{i'}=b_{j'}.$ Then $\tilde{I}$ gives an approximate
match between $a_{1}a_{2}\cdots a_{n}$ and $b_{1}b_{2}\cdots b_{m}.$
The number of symbols in all $A_{a_{i'}}\cup B_{b_{j'}}$ for all $(i',j')\in\tilde{I}$
is $2|\tilde{I}|L.$ Thus such symbols form a fraction of at most
$\min\left(1,\frac{2|\tilde{I}|}{n+m}\right)\le1-\tilde{f}$ of the
total number of symbols in $A_{a_{1}}A_{a_{2}}\cdots A_{a_{n}}$ and
$B_{b_{1}}B_{b_{2}}\cdots B_{b_{m}}.$

For any substring $A_{a_{i},s}$ corresponding to a substring $B_{i,b_{j_{i}+s}}$
such that the length of at least one of these two substrings is less
that $\frac{L}{R},$ the other substring has length less than $\frac{4L}{3R}$
(unless $\overline{f}(A_{a_{i},s},B_{i,b_{j_{i}+s}})>\alpha,$ a
case that we again ignore). Thus if we eliminate substrings $A_{a_{i},s}$
and $B_{i,b_{j_{i}+s}}$ such that the length of at least one of
the two substrings is less than $\frac{L}{R},$ we eliminate at most
$\frac{7Ln}{3R}$ symbols from the two strings $A_{a_{1}}A_{a_{2}}\cdots A_{a_{n}}$
and $B_{b_{1}}B_{b_{2}}\cdots B_{b_{m}}$ whose total combined length
is $L(n+m)\ge2Ln.$ Thus the fraction of symbols that are eliminated
due to such short substrings is at most $\frac{7}{6R}<\frac{2}{R}$.
For those pairs $(i,j_{i}+s),$ $s\in\{0,1,2\},$ such that neither
of the strings $A_{a_{i},s}$ and $B_{b_{j_{i}+s}}$ has length
less than $\frac{L}{R},$ it follows from the hypothesis 
that
\[
\overline{f}(A_{a_{i},s},B_{b_{j_{i}+s}})\ge\alpha,\text{\ if\ }(i,j_{i}+s)\notin\tilde{I},
\]
and
\[
\overline{f}(A_{a_{i},s},B_{b_{j_{i}+s}})\ge\beta,\text{\ if\ }(i,j_{i}+s)\in\tilde{I}.
\]
Thus by Properties \ref{property:omit_symbols} and \ref{property:substring_matching}, we obtain $\overline{f}(A_{a_{1}}A_{a_{2}}\cdots A_{a_{n}},B_{b_{1}}B_{b_{2}}\cdots B_{b_{m}})>\alpha\tilde{f}+\beta(1-\tilde{f)}-\frac{2}{R}=\alpha\left(1-(1-\tilde{f})\right)+\beta(1-\tilde{f})-\frac{2}{R}\ge\alpha-(1-\tilde{f})+\beta(1-\tilde{f})-\frac{2}{R}.$
\end{proof}

\begin{cor}\label{cor:alpha separated blocks}
Suppose $A_{a_{1}},A_{a_{2}},\dots,A_{a_{n}}$ and $A_{b_{1}},A_{b_{2}},\dots,A_{b_{m}}$
are blocks of symbols with each block of length $L.$ Assume that
$\alpha\in(0,\frac{1}{7})$, $R>0$, and
\[
\overline{f}(C,D)\ge\alpha
\]
for all substrings $C$ and $D$ consisting of consecutive symbols
from $A_{a_{i}}$ and $A_{b_{j}}$, respectively, where $a_{i}\ne b_{j},$
and $|C|,|D|\ge\frac{L}{R}.$ Then for $\tilde{f}=\tilde{f}(a_{1}a_{2}\cdots a_{n},b_{1}b_{2}\cdots b_{m}),$
we have
\[
\overline{f}(A_{a_{1}}A_{a_{2}}\cdots A_{a_{n}},A_{b_{1}}A_{b_{2}}\cdots A_{b_{m}})>\alpha\tilde{f}-\frac{2}{R}\ge\alpha-(1-\tilde{f})-\frac{2}{R}.
\]
\end{cor}

\begin{proof}
Let $B_{b_{j}}=A_{b_{j}}$ for $j=1,2,\dots,m,$ and $\beta=0$ in
Proposition \ref{prop:symbol by block replacement}.
\end{proof}
\begin{rem*}In fact, Corollary \ref{cor:alpha separated blocks} and the case $\beta=0$ of Proposition \ref{prop:symbol by block replacement} hold with $\tilde{f}$ replaced by $\overline{f}(a_1a_2\cdots a_n,b_1b_2\cdots b_m)$. We omit the proof because these versions are not needed for our results.
\end{rem*}

\begin{rem*} The following lemma is essentially the same as Theorem 4 in \cite{Fe} and Proposition 1.1 in Chapter 10 of \cite{ORW}. The proof is also essentially the same, but the estimates are more suited to our applications. 
\end{rem*}

\begin{lem}\label{lem:symbolic Feldman}
Suppose $a_{1},a_{2},\dots,a_{N}$ are distinct symbols in $\Sigma.$ Let
\begin{align*}
{B}_{1}= & \left(a_{1}^{N^{2}}a_{2}^{N^{2}}\dots a_{N}^{N^{2}}\right)^{N^{2M}},\\
{B}_{2}= & \left(a_{1}^{N^{4}}a_{2}^{N^{4}}\dots a_{N}^{N^{4}}\right)^{N^{2M-2}},\\
\vdots\;\; & \;\;\vdots\\
{B}_{M}= & \left(a_{1}^{N^{2M}}a_{2}^{N^{2M}}\dots a_{N}^{N^{2M}}\right)^{N^{2}}.
\end{align*}
Suppose $B$ and $\overline{B}$ are strings of consecutive symbols in
$B_{j}$ and $B_{k},$ respectively, where $|B|\ge N^{2M+2}$,
$|\overline{B}|\ge N^{2M+2}$, and $j\ne k$. Assume that $N\ge 20$ and $M\ge 2.$
Then
\[
\overline{f}(B,\overline{B})>1-\frac{4}{\sqrt{N}}\text{\ and\  }\tilde{f}(B,\overline{B})>1-\frac{12}{\sqrt{N}}.
\]
\end{lem}

\begin{proof}
We may assume that $j>k.$ By removing fewer than $2N^{2j}$ symbols
from the beginning and end of $B,$ we can decompose the remaining
part of $B$ into strings $C_{1},C_{2},\dots,C_{r}$ each of the form
$a_{i}^{N^{2j}}.$ Since $2N^{2j}\le\frac{|B|+|\overline{B}|}{N^{2}},$
it follows from Property \ref{property:omit_symbols} that removing these symbols increases the
$\overline{f}$ distance between $B$ and $\overline{B}$ by less
than $\frac{2}{N^{2}}.$ Let $\overline{C}_{1},\overline{C}_{2},\dots,\overline{C}_{r}$
be the decomposition of $\overline{B}$ into substrings corresponding
to $C_{1},C_{2},\dots,C_{r}$ under a best possible match between
$C_{1}C_{2}\cdots C_{r}$ and $\overline{B}.$

Let $i\in\{1,2,\dots,r\}.$

\emph{Case 1.} $|\overline{C}_{i}|<\frac{3}{2\sqrt{N}}|C_{i}|$. Then
by Property \ref{property:string_length}, $\overline{f}(C_{i},\overline{C}_{i})>1-\frac{3}{\sqrt{N}}$.

\emph{Case 2.} $|\overline{C}_{i}|\ge\frac{3}{2\sqrt{N}}|C_{i}|=(3/2)N^{2j-(1/2)}$.
The length of a cycle $a_1^{N^{2k}}a_2^{N^{2k}}\cdots a_N^{N^{2k}}$ in $B_{k}$ is at most $N^{2j-1}.$ Therefore
$\overline{C}_{i}$ contains at least $\lfloor\frac{3\sqrt{N}}{2}\rfloor-1>\sqrt{N}$
complete cycles. Thus deleting any partial cycles at the beginning
and end of $\overline{C}_{i}$ would increase the $\overline{f}$ distance between
$C_{i}$ and $\overline{C}_{i}$ by less than $\frac{2}{\sqrt{N}}.$
On the rest of $\overline{C}_{i}$, only $\frac{1}{N}$ of the symbols
in $\overline{C}_{i}$ can match the symbol in $C_{i}.$ Thus $\overline{f}(C_{i},\overline{C}_{i})>1-\frac{2}{\sqrt{N}}-\frac{4}{N}>1-\frac{3}{\sqrt{N}}.$

Therefore, by Property \ref{property:substring_matching}, $\overline{f}(C_{1}C_{2}\cdots C_{r},\overline{C}_{1},\overline{C}_{2}\cdots,\overline{C}_r)> 1-\frac{3}{\sqrt{N}}. $ By Lemma \ref{lem:compare ftilde} the claimed
$\tilde{f}$ inequality holds as well.
\end{proof}

\begin{rem} \label{rem:repetition}
If we replace each symbol $a_i$ by a constant number of repetitions $a_i^l$, then the same conclusion still holds for substrings of length at least $lN^{2M+2}$. 
\end{rem}

\begin{prop}
\label{prop:Feldman} Let $\alpha\in(0,\frac{1}{7})$, $n \in \mathbb{N}$, and $K,R,S,N,M\in\mathbb{N}\setminus\left\{ 0\right\} $
with $N\geq 20$ and $M\geq2$. For $1\leq s\leq S$,
let $\mathtt{A}_{1}^{(s)},\dots,\mathtt{A}_{N}^{(s)}$ be a family
of strings, where each $\mathtt{A}_{j}^{(s)}$ is a concatenation
of $K$ many $n$-blocks. Assume that for all $0\leq i_{1},i_{2}<q_{n}$,
all $1\leq s_{1},s_{2}\leq S$ and all $j_{1},j_{2}\in\left\{ 1,\dots,N\right\} $,
$j_{1}\neq j_{2}$, we have $\overline{f}\left(\mathcal{A},\overline{\mathcal{A}}\right)>\alpha$
for all sequences $\mathcal{A},\overline{\mathcal{A}}$ each consisting
of at least $K l_{n} q_{n}/R$ consecutive symbols
from $\mathcal{C}_{n,i_{1}}\left(\mathtt{A}_{j_{1}}^{(s_{1})}\right)$
and $\mathcal{C}_{n,i_{2}}\left(\mathtt{A}_{j_{2}}^{(s_{2})}\right)$,
respectively. 

Then for $1\leq s\leq S$, we can construct a family of strings $\mathtt{B}_{1}^{(s)},\dots,\mathtt{B}_{M}^{(s)}$
(of equal length $N^{2M+3} K \mathtt{h}_{n}$ and containing
each block $\mathtt{A}_{1}^{(s)},\dots,\mathtt{A}_{N}^{(s)}$ exactly
$N^{2M+2}$ times) such that for all $0\leq i_{1},i_{2}<q_{n}$, all
$1\leq s_{1},s_{2}\leq S$, all $j,k\in\left\{ 1,\dots,M\right\} $,
$j\neq k,$ and all sequences $\mathcal{B}$, $\overline{\mathcal{B}}$
of at least $N^{2M+2} l_{n} K q_{n}$ consecutive symbols
from $\mathcal{C}_{n,i_{1}}\left(\mathtt{B}_{j}^{(s_{1})}\right)$
and $\mathcal{C}_{n,i_{2}}\left(\mathtt{B}_{k}^{(s_{2})}\right)$
we have 
\[
\overline{f}(\mathcal{B},\overline{\mathcal{B}})>\alpha-\frac{13}{\sqrt{N}}-\frac{2}{R}.
\]
 
\end{prop}

\begin{proof}
For every $1\leq s\leq S$ we define 
\begin{align*}
\mathtt{B}_{1}^{(s)}= & \left(\left(\mathtt{A}_{1}^{(s)}\right)^{N^{2}}\left(\mathtt{A}_{2}^{(s)}\right)^{N^{2}}\dots\left(\mathtt{A}_{N}^{(s)}\right)^{N^{2}}\right)^{N^{2M}}\\
\mathtt{B}_{2}^{(s)}= & \left(\left(\mathtt{A}_{1}^{(s)}\right)^{N^{4}}\left(\mathtt{A}_{2}^{(s)}\right)^{N^{4}}\dots\left(\mathtt{A}_{N}^{(s)}\right)^{N^{4}}\right)^{N^{2M-2}}\\
\vdots\;\; & \;\;\vdots\\
\mathtt{B}_{M}^{(s)}= & \left(\left(\mathtt{A}_{1}^{(s)}\right)^{N^{2M}}\left(\mathtt{A}_{2}^{(s)}\right)^{N^{2M}}\dots\left(\mathtt{A}_{N}^{(s)}\right)^{N^{2M}}\right)^{N^{2}}
\end{align*}
Let $\mathcal{B}_{j,i_{1}}^{(s_{1})}=\mathcal{C}_{n,i_{1}}(\mathtt{B}_{j}^{(s_{1})}),$
$\mathcal{B}_{j,i_{2}}^{(s_{2})}=\mathcal{C}_{n,i_{2}}(\mathtt{B}_{j}^{(s_{1})})$,
$\mathcal{A}_{j,i_{1}}^{(s_{1})}=\mathcal{C}_{n,i_{1}}(\mathtt{A}_{j}^{(s_{1})}),$
and $\mathcal{A}_{j,i_{2}}^{(s_{2})}=\mathcal{C}_{n,i_{2}}(\mathtt{A}_{j}^{(s_{2})}).$
Then the formulas for the $\mathcal{B}_{j,i_{1}}^{(s_{1})},j=1,\dots,M,$
in terms of the $\mathcal{A}_{1,i_{1}}^{(s_{1})},\dots$, 
$\mathcal{A}_{N,i_{1}}^{(s_{1})}$
can be obtained from the formulas for the $\mathtt{B}_{j}^{(s_{1})}$
in terms of the $\mathtt{A}_{1}^{(s_{1})},\dots,\mathtt{A}_{N}^{(s_{1})}$
by replacing each typewriter font $\mathtt{A,B}$ by the calligraphic
$\mathcal{A},\mathcal{B}$ with the corresponding sub- and superscripts,
and the analogous statement is true for the $\mathcal{B}_{j,i_{2}}^{(s_{2})},j=1,\dots,M.$

By adding fewer than  $2l_{n}Kq_{n}$ symbols to each of $\mathcal{B}$
and $\overline{\mathcal{B}}$ we can complete any partial $\mathcal{A}_{j,i_{1}}^{(s_{1})}$
at the beginning and end of $\mathcal{B}$ and any partial $\mathcal{A}_{j,i_{2}}^{(s_{2})}$
at the beginning and end of $\overline{\mathcal{B}}.$ Let $\mathcal{B}_{\text{aug}}$ and $\overline{\mathcal{B}}_{\text{aug}}$ be the augmented $\mathcal{B}$ and $\overline{\mathcal{B}}$ strings obtained in this way. By Property \ref{property:omit_symbols}, $\overline{f}(\mathcal{B},\overline{\mathcal{B}})> \overline{f}(\mathcal{B}_{\text{aug}},\overline{\mathcal{B}}_{\text{aug}})- (4l_{n}Kq_{n})/(2N^{2M+2}l_{n}Kq_{n})=\overline{f}(\mathcal{B}_{\text{aug}},\overline{\mathcal{B}}_{\text{aug}})-2/{N^{2M+2}}.$
Then we are comparing two different Feldman patterns of blocks, and by Lemma \ref{lem:symbolic Feldman}
and Corollary \ref{cor:alpha separated blocks} with $\tilde{f}>1-\frac{12}{\sqrt{N}}$, we have
\[
\overline{f}(\mathcal{B}_{\text{aug}},\overline{\mathcal{B}}_{\text{aug}})>\alpha-\frac{12}{\sqrt{N}}-\frac{2}{R}.
\]
 Therefore
\[
\overline{f}(\mathcal{B},\overline{\mathcal{B}})>\alpha-\frac{12}{\sqrt{N}}-\frac{2}{R}-\frac{2}{N^{2M+2}}>\alpha-\frac{13}{\sqrt{N}}-\frac{2}{R}.
\]
\end{proof}
For an application in Section \ref{sec:Shifting-Mechanism} we will need the following result on the $\overline{f}$ distance
between strings that can be built as the same or different Feldman
patterns but with building blocks from different families. 
\begin{lem}
\label{lem:samePattern} Let $\alpha\in(0,\frac{1}{7})$, $n\in \mathbb{N}$, and $K,R,S,N,M\in\mathbb{N}\setminus \{0\}$
with $N\geq 20$, and $M,S$ at least $2$. For $1\leq s\leq S$, let
$\mathtt{A}_{1}^{(s)},\dots,\mathtt{A}_{N}^{(s)}$ be a family of strings, where each $\mathtt{A}_{j}^{(s)}$ is
a concatenation of $K$ many $n$-blocks. Assume that for all $0\leq i_{1},i_{2}<q_{n}$,
all $s_{1}\neq s_{2}$, and all $j_{1},j_{2}\in\left\{ 1,\dots,N\right\} $
we have $\overline{f}\left(\mathcal{A},\overline{\mathcal{A}}\right)>\alpha$
for all strings $\mathcal{A},\overline{\mathcal{A}}$ of at least
$K l_{n} q_{n}/R$ consecutive symbols from
$\mathcal{C}_{n,i_{1}}\left(\mathtt{A}_{j_{1}}^{(s_{1})}\right)$
and $\mathcal{C}_{n,i_{2}}\left(\mathtt{A}_{j_{2}}^{(s_{2})}\right)$
respectively. 

Then for $1\leq s\leq S$ we can construct a family of strings $\mathtt{B}_{1}^{(s)},\dots\mathtt{B}_{M}^{(s)}$,
as in Proposition \ref{prop:Feldman}
and obtain that for all $0\leq i_{1},i_{2}<q_{n}$, all $j_{1},j_{2}\in\left\{ 1,\dots,M\right\} $,
and all sequences $\mathcal{B}$, $\overline{\mathcal{B}}$ of at
least $N^{2M+2} l_{n} K q_{n}$ consecutive symbols
from $\mathcal{C}_{n,i_{1}}\left(\mathtt{B}_{j_{1}}^{(s_{1})}\right)$
and $\mathcal{C}_{n,i_{2}}\left(\mathtt{B}_{j_{2}}^{(s_{2})}\right)$,
$s_{1}\neq s_{2}$, we have 
\begin{equation}\label{eq:same}
\overline{f}(\mathcal{B},\overline{\mathcal{B}})>\alpha-\frac{2}{N^{2M+2}}-\frac{2}{R}.
\end{equation}
 
\end{lem}

\begin{proof}The same argument as in the proof of Proposition \ref{prop:Feldman} applies, except we use Corollary \ref{cor:alpha separated blocks} with $\tilde{f}=1.$

\end{proof}

\section{\label{sec:Feldman-Mechanism}Feldman Mechanism to produce sufficiently
many blocks}

Recall that the images of odometer-$n$-blocks under $\mathcal{C}_{n,i}$
are part of the circular $(n+1)$-block. The following statement allows
us to obtain lower bounds on the $\overline{f}$ distance between substrings of
images of odometer-$n$-blocks under $\mathcal{C}_{n,i}$ given a lower bound on 
the $\overline{f}$ distance between substrings of the circular $n$-blocks.
Since we will apply this result several times in the proofs of Proposition
\ref{prop:FeldmanMechanism} and Proposition \ref{prop:sep}, we state
it as a separate Lemma. 
\begin{lem}
\label{lem:operators} Let $\alpha\in\left(0,\frac{1}{7}\right)$
and $n,N,R\in\mathbb{N}\setminus\{0\}$. Moreover, let $N$ many odometer $n$-blocks $\mathtt{B}_{1}^{(n)},\dots,\mathtt{B}_{N}^{(n)}$
be given such that for all $j,k\in\left\{ 1,\dots,N\right\} $, $j\neq k$,
we have $\overline{f}\left(\mathcal{A},\overline{\mathcal{A}}\right)>\alpha$
for any sequences $\mathcal{A},\overline{\mathcal{A}}$ of at least
$q_{n}/R$ consecutive symbols from the circular $n$-blocks $\mathcal{B}_{j}^{(n)}$
and $\mathcal{B}_{k}^{(n)}$ respectively. Then for all $j,k\in\left\{ 1,\dots,N\right\} $,
$j\neq k$, any $0\leq i_{1},i_{2}<q_{n}$, and any sequences $\mathcal{D},\overline{\mathcal{D}}$
of at least $q_{n}l_{n}/R$ consecutive symbols
from $\mathcal{C}_{n,i_{1}}\left(\mathtt{B}_{j}^{(n)}\right)$ and
$\mathcal{C}_{n,i_{2}}\left(\mathtt{B}_{k}^{(n)}\right)$, respectively, we have
\[
\overline{f}\left(\mathcal{D},\overline{\mathcal{D}}\right)>\alpha-\frac{2}{R}-\frac{4R}{l_{n}}.
\]
\end{lem}

\begin{proof}
Let $\mathcal{D},\overline{\mathcal{D}}$ be arbitrary sequences of
at least $q_{n}l_{n}/R$ many consecutive symbols
from $\mathcal{C}_{n,i_{1}}\left(\mathtt{B}_{j}^{(n)}\right)=b^{q_n-j_{i_1}}\left(\mathcal{B}_j^{(n)}\right)^{l_n-1}e^{j_{i_1}}$ and from 
$\mathcal{C}_{n,i_{2}}\left(\mathtt{B}_{k}^{(n)}\right) = b^{q_n-j_{i_2}}$ \hfill  $\left(\mathcal{B}_j^{(n)}\right)^{l_n-1}e^{j_{i_2}}$, respectively, for any
$j\neq k$ and any $0\leq i_{1},i_{2}<q_{n}$. We modify $\mathcal{D}$ and $\overline{\mathcal{D}}$ by first completing any partial blocks $\mathcal{B}_j^{(n)}$ and $\mathcal{B}_k^{(n)}$, which can be accomplished by adding fewer than $2q_n$ symbols to each of $\mathcal{D}$ and $\overline{\mathcal{D}}$. Then we remove any of the $b$'s preceding the $\mathcal{B}_j^{(n)}$'s and any of the $e$'s following the $\mathcal{B}_j^{(n)}$'s that are included in $\mathcal{D}$'s, and similarly for such $b$'s and $e$'s in 
$\overline{\mathcal{D}}$. At most $q_n$ symbols are removed from each of $\mathcal{D}$ and $\overline{\mathcal{D}}$. Let $\mathcal{D}_{\text{mod}}$ and $\overline{\mathcal{D}}_{\text{mod}}$ be these modified versions of $\mathcal{D}$ and $\overline{\mathcal{D}}.$ Then by Property \ref{property:omit_symbols},
\[
\overline{f}\left(\mathcal{D},\overline{\mathcal{D}}\right)>\overline{f}\left(\mathcal{D}_{\text{mod}},\overline{\mathcal{D}}_{\text{mod}}\right)-\frac{2R}{l_n}-\frac{2R}{l_n}.
\]
We have $\mathcal{D}_{\text{mod}}=\left(\mathcal{B}^{(n)}_j\right)^l$ and $\overline{\mathcal{D}}_{\text{mod}}=\left(\mathcal{B}^{(n)}_k\right)^{\overline{l}}$, for some positive integers $l$ and $\overline{l}$. We are given that $\overline{f}\left(\mathcal{A},\overline{\mathcal{A}}\right)>\alpha$ for any strings of at least $q_n/R$ consecutive symbols from $\mathcal{B}_j^{(n)}$ and $\mathcal{B}_k^{(n)}$ respectively. Thus it follows from Corollary \ref{cor:alpha separated blocks} with $\tilde{f}=1$ that 

\[
\overline{f}\left(\mathcal{D}_{\text{mod}},\overline{\mathcal{D}}_{\text{mod}}\right)>\alpha-\frac{2}{R}.
\]

\end{proof}

In the proofs of both Propositions \ref{prop:FeldmanMechanism} and \ref{prop:sep} we will
also use the sequence $\left(R_{n}\right)^{\infty}_{n=1}$, where $R_1=N(0)$ (with $N(0)+1$ the number of symbols in our alphabet) and $R_{n}=k_{n-2}\cdot q_{n-2}^{2}$
for $n\geq 2$. We note that 
for $n\geq 2$
\begin{equation}
\frac{q_{n}}{R_{n}}=\frac{k_{n-1}\cdot l_{n-1}\cdot\left(k_{n-2}\cdot l_{n-2}\cdot q_{n-2}^{2}\right)\cdot q_{n-1}}{k_{n-2}\cdot q_{n-2}^{2}}=l_{n-2}\cdot k_{n-1}\cdot l_{n-1}\cdot q_{n-1}.\label{eq:sn}
\end{equation}
Hence, for $n\geq 2$ a substring of at least $q_{n}/R_{n}$ consecutive
symbols in a circular $n$-block contains at least $l_{n-2}-1$ complete
$2$-subsections which have length $k_{n-1}l_{n-1}q_{n-1}$ (recall
the notion of a $2$-subsection from the end of Subsection \ref{subsec:Circular-Systems}).
This will allow us to ignore incomplete $2$-subsections at the ends
of the substring.

\begin{prop}
\label{prop:FeldmanMechanism} Let $\alpha\in\left(0,\frac{1}{7}\right)$ and $n,N,M\in\mathbb{N}$ with $N\geq 100$ and $M\geq2$.
Suppose $\mathtt{A}_{0},\dots,\mathtt{A}_{N}$ is the collection of $n$-blocks, which have
equal length $\mathtt{h_{n}}$ and satisfy the unique readability
property. Furthermore, if $n>0$ assume that for all $j_{1},j_{2}\in\left\{ 0,\dots,N\right\} $,
$j_{1}\neq j_{2}$, we have $\overline{f}\left(\mathcal{A},\overline{\mathcal{A}}\right)>\alpha$
for any sequences $\mathcal{A},\overline{\mathcal{A}}$ of at least
$q_{n}/R_{n}$ many consecutive symbols from $\mathcal{A}_{j_{1}}$
and $\mathcal{A}_{j_{2}}$ respectively. Then we can construct
$M$ many $(n+1)$-blocks $\mathtt{B}_{1},\dots,\mathtt{B}_{M}$
of equal length $\mathtt{h}_{n+1}$ (which are uniform in the $n$-blocks
and satisfy the unique readability property) such that for all $j,k\in\left\{ 1,\dots,M\right\} $,
$j\neq k$, and any sequences $\mathcal{B}$, $\overline{\mathcal{B}}$
of at least $q_{n+1}/R_{n+1}$ consecutive symbols
from $\mathcal{B}_{j}$ and $\mathcal{B}_{k}$ we
have 
\[
\overline{f}(\mathcal{B},\overline{\mathcal{B}})> \begin{cases} \alpha-\left(\frac{4}{R_n}+\frac{4R_n}{l_n}+\frac{14}{\sqrt{N}}+\frac{2}{l_{n-1}}\right), & \text{ if } n>0; \\
1-\frac{5}{\sqrt{N}}-\frac{2}{l_0}, & \text{ if } n=0.
\end{cases}
\]
\end{prop}

\begin{proof}
We choose the block $\mathtt{A}_{0}$ as a ``marker'', that is, an $n$-block whose appearances can be used to identify the end of an $(n+1)$-block.  We distribute
the marker blocks over the new words and modify the classical Feldman
patterns on the building blocks $\mathtt{A}_{1},\dots,\mathtt{A}_{N}$
in the following way to define the $(n+1)$-blocks:
\begin{align*}
\mathtt{B}_{1}= & \left(\left(\left(\mathtt{A}_{1}\right)^{N^{2}}\left(\mathtt{A}_{2}\right)^{N^{2}}\dots\left(\mathtt{A}_{N}\right)^{N^{2}}\right)^{N^{2M-1}}\left(\mathtt{A}_{0}\right)^{N^{2M+1}-1}\right)^N\left(\mathtt{A}_{0}\right)^{N}\\
\mathtt{B}_{2}= & \left(\left(\left(\mathtt{A}_{1}\right)^{N^{4}}\left(\mathtt{A}_{2}\right)^{N^{4}}\dots\left(\mathtt{A}_{N}\right)^{N^{4}}\right)^{N^{2M-3}}\left(\mathtt{A}_{0}\right)^{N^{2M+1}-1}\right)^N\left(\mathtt{A}_{0}\right)^{N}\\
\vdots\;\; & \;\;\vdots\\
\mathtt{B}_{M}= & \left(\left(\left(\mathtt{A}_{1}\right)^{N^{2M}}\left(\mathtt{A}_{2}\right)^{N^{2M}}\dots\left(\mathtt{A}_{N}\right)^{N^{2M}}\right)^{N}\left(\mathtt{A}_{0}\right)^{N^{2M+1}-1}\right)^N\left(\mathtt{A}_{0}\right)^{N}
\end{align*}
 We note that every $(n+1)$-block $\mathtt{B}_{k}$ contains each
$n$-block $\mathtt{A}_{l}$ exactly $N^{2M+2}$ many times and has
length $(N+1)\cdot N^{2M+2}\cdot\mathtt{h}_{n}$. Moreover, the new
blocks are uniquely readable because the string $\left(\mathtt{A}_{0}\right)^{N^{2M+1}+N-1}$
only occurs at the end of an $(n+1)$-block. We also observe that $\mathtt{B}_{k}$ is built with $N^{2\cdot(M-k+1)}$ \emph{cycles} 
\[
\mathtt{F}_k \coloneqq \left(\mathtt{A}_{1}\right)^{N^{2k}}\left(\mathtt{A}_{2}\right)^{N^{2k}}\dots\left(\mathtt{A}_{N}\right)^{N^{2k}}.
\]

Let $\mathcal{B}$ and $\overline{\mathcal{B}}$ be sequences of at
least $q_{n+1}/R_{n+1}$ consecutive symbols from
$\mathcal{B}_{j}$ and $\mathcal{B}_{k},$ $j\ne k$.
In case of $n>0$ we note that $\mathcal{B}$ and $\overline{\mathcal{B}}$ have at least the length of $l_{n-1}$ complete 
$2$-subsections by equation (\ref{eq:sn}). By adding fewer than $2l_nk_nq_n$ symbols 
to each of $\mathcal{B}$ and $\overline{\mathcal{B}}$, we can complete any partial $2$-subsections
at the beginning and end of $\mathcal{B}$ and $\overline{\mathcal{B}}$. This change can increase the
$\overline{f}$ distance between $\mathcal{B}$ and $\overline{\mathcal{B}}$, but by less than $2/l_{n-1}$. In addition we remove the marker blocks, possibly increasing $\overline{f}$ by at most $\frac{2}{N+1}$. The modified strings $\mathcal{B}_{\text{mod}}$ and $\overline{\mathcal{B}}_{\text{mod}}$ obtained satisfy $\overline{f}\left(\mathcal{B},\overline{\mathcal{B}}\right)>\overline{f}\left(\mathcal{B}_{\text{mod}},\overline{\mathcal{B}}_{\text{mod}}\right)-\frac{2}{l_{n-1}}-\frac{2}{N+1}.$

By Lemma \ref{lem:operators}, $\overline{f}\left(\mathcal{D},\overline{\mathcal{D}}\right)>\alpha-\frac{2}{R_n}-\frac{4R_n}{l_n}$ for any substrings $\mathcal{D}$,$\overline{\mathcal{D}}$ of at least $q_nl_n/R_n$ consecutive symbols from $\mathcal{C}_{n,i_1}(\mathtt{A}_{j_1})$ and $\mathcal{C}_{n,i_2}(\mathtt{A}_{j_2})$ with $j_1\ne j_2$. If we let $\Phi_{j,i_1}$ and $\Phi_{k,i_2}$ be the $j$th and $k$th Feldman patterns built from $\mathcal{C}_{n,i_1}(\mathtt{A}_1),\mathcal{C}_{n,i_1}(\mathtt{A}_2),\dots,$\hfill $\mathcal{C}_{n,i_1}(\mathtt{A}_{N})$ and $\mathcal{C}_{n,i_2}(\mathtt{A}_1),\mathcal{C}_{n,i_2}(\mathtt{A}_2),\dots,\mathcal{C}_{n,i_2}(\mathtt{A}_{N})$, respectively, then the same argument as in Proposition \ref{prop:Feldman} shows that for any substrings 
$\mathcal{E}$,$\overline{\mathcal{E}}$ consisting of at least $\frac{|\Phi_{j,i_1}|}{N}=\frac{|\Phi_{k,i_1}|}{N}$ consecutive symbols from $\Phi_{j,i_2}$ and $\Phi_{k,i_2}$, we have $\overline{f}(\mathcal{E},\overline{\mathcal{E}})>\alpha-\frac{4}{R_n}-\frac{4R_n}{l_n}-\frac{13}{\sqrt{N}}$. 

Note that $\mathcal{B}_{\text{mod}}$ and $\overline{\mathcal{B}}_{\text{mod}}$ consist, respectively, of a string of $\Phi_{j,i}$'s and a string of $\Phi_{k,i}$'s ($j$ and $k$ fixed, $i$'s varying). Therefore, by Corollary \ref{cor:alpha separated blocks} with $\tilde{f}=1$, $\overline{f}(\mathcal{B}_{\text{mod}},\overline{\mathcal{B}}_{\text{mod}})>\alpha-\frac{4}{R_n}-\frac{4R_n}{l_n}-\frac{13}{\sqrt{N}}-\frac{2}{N}$. Thus in case $n>0$ we obtain $\overline{f}(\mathcal{B},\overline{\mathcal{B}})>\alpha-\frac{4}{R_n}-\frac{4R_n}{l_n}-\frac{13}{\sqrt{N}}-\frac{2}{N}-\frac{2}{l_{n-1}}-\frac{2}{N}>\alpha-\frac{4}{R_n}-\frac{4R_n}{l_n}-\frac{14}{\sqrt{N}}-\frac{2}{l_{n-1}}.$

In case of $n=0$ we complete strings $\mathcal{C}_{0,0}(\mathtt{F}_{j})$ and $\mathcal{C}_{0,0}(\mathtt{F}_{k})$ at the beginning and end of $\mathcal{B}$ and $\overline{\mathcal{B}}$ respectively by adding fewer than $2l_0N^{2M+1}$ symbols to each of $\mathcal{B}$ and $\overline{\mathcal{B}}$. This corresponds to a fraction of at most $2/(N+1)$ of the total length. In the next step we remove the marker blocks and spacers $b$, possibly increasing $\overline{f}$ by at most $\frac{6}{N}+\frac{2}{l_0}$. On the remaining strings we apply Remark \ref{rem:repetition} (note that we have enough symbols by our completion above) to obtain
the claim for $n=0$.
\end{proof}

\begin{rem*}
In the proof above we cannot put all markers at the end of the new $(n+1)$-block
since these markers would cover a fraction $\frac{\mathtt{h}_{n+1}}{N+1}$
of that block due to uniformity. Thus, the conclusion would not hold true in case of $n=0$. We will also need the chosen form of the $(n+1)$-blocks in an analogous statement for the odometer-based system in Proposition \ref{prop:OdomFeld}.
\end{rem*}

\section{\label{sec:Shifting-Mechanism}Mechanism to produce closeness in
Odometer-Based System and separation in corresponding Circular System}

\begin{figure}
    \centering
    \includegraphics[scale=0.8]{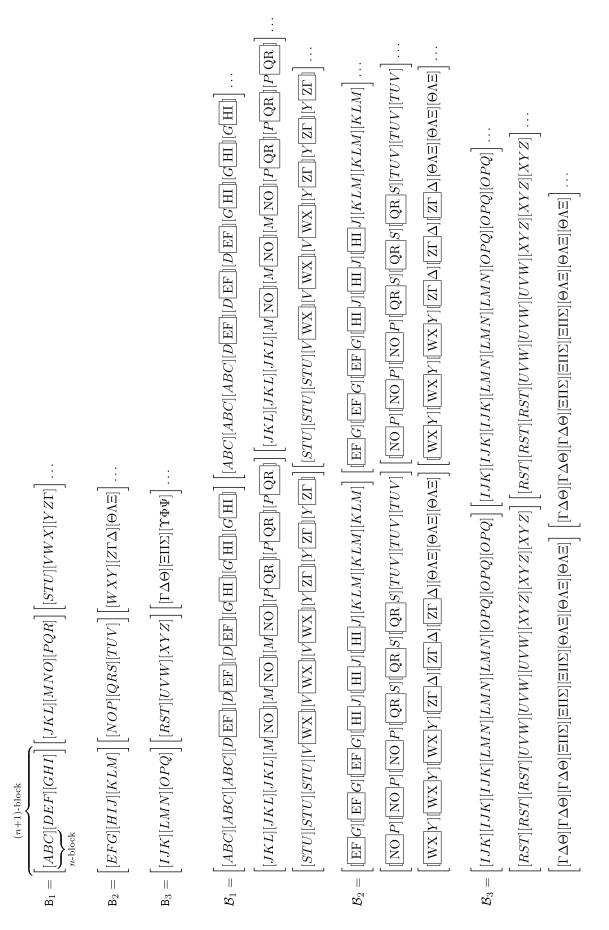}
    \caption{{\bf Heuristic representation of two stages of the shifting mechanism.} Parts of three $(n+2)$-blocks $\mathtt{B}_1, \mathtt{B}_2, \mathtt{B}_3$ in the odometer based system and parts of their images $\mathcal{B}_1, \mathcal{B}_2, \mathcal{B}_3$ under the circular operator are represented. The marked letters indicate a best possible $\overline{f}$ match between $\mathcal{B}_1$ and $\mathcal{B}_2$ with a fit of approximately $\left(1-\frac{1}{3}\right)^2$ (ignoring spacers and boundary effects) while the blocks $\mathtt{B}_1$ and $\mathtt{B}_2$ have a very good fit in the odometer-based system.}
    \label{fig:fig2}
\end{figure}

We impose the following conditions on the circular coefficients $\left(l_n\right)_{n\in \mathbb{N}}$:
\begin{equation} \label{eq:lcond}
    l_{n+1} \geq l^2_n \text{ and } l_{n} \geq 4R_{n+1} \text{ for every } n \in \mathbb{N}.
\end{equation}

\begin{prop}
\label{prop:sep} Let $K\geq2$, $0<\varepsilon<\alpha$, $\delta>0$, and
$\alpha\in\left(0,\frac{1}{7}\right)$. Then there are numbers $N,p\in\mathbb{N}$ such that for $N+1$ many uniquely readable $n$-blocks $\mathtt{B}^{(n)}_{0},\mathtt{B}^{(n)}_{1},\dots,\mathtt{B}^{(n)}_{N}$
with $\overline{f}(\mathcal{A},\overline{\mathcal{A}})>\alpha$ for
all sequences $\mathcal{A},\overline{\mathcal{A}}$ of at least $q_{n}/R_{n}$
consecutive symbols from $\mathcal{B}^{(n)}_{i}$ and $\mathcal{B}^{(n)}_{j}$,
$i>j$, we can build $K$ many $(n+p)$-blocks $\mathtt{B}^{(n+p)}_{1},\dots,\mathtt{B}^{(n+p)}_{K}$
of equal length $\mathtt{h}_{n+p}$ (satisfying the unique readability
property and uniformity in all blocks from stage $n$ through $n+p$)
with the following properties: 
\begin{enumerate}
\item $\overline{f}(\mathtt{B}^{(n+p)}_{i},\mathtt{B}^{(n+p)}_{j})<\delta$ for all $i,j$ 
\item $\overline{f}\left(\mathcal{B},\overline{\mathcal{B}}\right)>\alpha-\varepsilon - \sum^{p-1}_{s=0} \frac{6}{R_{n+s}}$
for all sequences $\mathcal{B}$, $\overline{\mathcal{B}}$ of at
least $q_{n+p}/R_{n+p}$ consecutive symbols
from $\mathcal{B}^{(n+p)}_{i}$ and $\mathcal{B}^{(n+p)}_{j}$, $i>j$. 
\end{enumerate}
\end{prop}

The number of stages $p$ will be the least integer such that
\[
\left(1-\frac{1}{K}\right)^{p} < \frac{\varepsilon}{2}.
\] 
The proof bases upon an inductive construction (called \emph{shifting
mechanism}, see Figure \ref{fig:fig2} for a sketch of its idea) and a final step to align blocks in the odometer-based
system. For this construction process and the given $n\in \mathbb{N}$ let $\left(u_{n+m}\right)_{m\in\mathbb{N}}$
and $\left(e_{n+m}\right)_{m\in\mathbb{N}}$ be increasing sequences
of natural numbers 
such that

\begin{equation}
\sum_{m\in\mathbb{N}}\frac{1}{u_{n+m}^{2}}<\frac{\delta}{4}\label{eq:Delta}
\end{equation}
and

\begin{equation}
\sum_{m\in\mathbb{N}}\left(\frac{8}{u_{n+m}}+\frac{17}{\sqrt{e_{n+m}}}\right)<\frac{\varepsilon}{8}.\label{eq:eps}
\end{equation}
Moreover, we will use the sequence $\left(d_{n+m}\right)_{m\in\mathbb{N}}$,
where 
\begin{equation}
d_{n+m}=u_{n+m}^{2}.\label{eq:d}
\end{equation}
In the following, we will also use the notation 
\[
\lambda_{n+m} = d_{n+m}\cdot e_{n+m} 
\]
and $N(n+m)+1=K\lambda_{n+m}+1$ will be
the number of $(n+m)$-blocks. In particular, we start with $N(n)+1=N+1$
many $n$-blocks and we require $N$ to satisfy 
\begin{equation}
N>\max\left(2/\delta,\left(100/\varepsilon\right)^2\right).\label{eq:Neps}
\end{equation}

\subsection{Initial stage of the Shifting Mechanism: Construction of $(n+1)$-blocks} \label{subsec:initial}

First of all, we choose one $n$-block $\mathtt{B}^{(n)}_{0}$ as a marker. Then we apply Proposition
\ref{prop:Feldman} on the remaining $n$-blocks $\mathtt{B}^{(n)}_{1},\dots,\mathtt{B}^{(n)}_{N}$
to build $\tilde{N}(n+1)\coloneqq2K\lambda_{n+1}$ many
\emph{pre-$(n+1)$-blocks }denoted by \emph{$\mathtt{A}_{i,j}$, $i=1,\dots,K$,
$j=1,\dots,2\lambda_{n+1}$}. In particular, these have length $\tilde{\mathtt{h}}_{n+1}=N^{2\cdot\tilde{N}(n+1)+3}\cdot\mathtt{h}_{n}$
and are uniform in the $n$-blocks $\mathtt{B}^{(n)}_{1},\dots,\mathtt{B}^{(n)}_{N}$
by construction. More precisely, every pre-$(n+1)$-block contains
each $n$-block $\mathtt{B}^{(n)}_{j}$, $1\leq j\leq N$, exactly $N^{2\cdot\tilde{N}(n+1)+2}$
many times and pre-$(n+1)$-blocks $\mathcal{A}_{i,j}$ in the circular
system (i.e. images of $\mathtt{A}_{i,j}$ under the operator $\mathcal{C}_{n,k}$
for some $k\in\left\{ 1,\dots,q_{n}\right\} $, where the value of $k$ does not matter for the following investigation) have length $\tilde{q}_{n+1}=N^{2\cdot\tilde{N}(n+1)+3}\cdot l_{n}\cdot q_{n}$.
Moreover, with the aid of Lemma \ref{lem:operators}, Proposition
\ref{prop:Feldman} also implies that different pre-$(n+1)$-blocks
in the circular system are at least
\begin{equation} 
\alpha-\beta_{n+1}, \ \text{ where } \beta_{n+1}\coloneqq \frac{4}{R_{n}}+\frac{13}{\sqrt{N}}+\frac{4R_{n}}{l_{n}},\label{eq:pren}
\end{equation}
$\overline{f}$ apart on substantial substrings of length at least $N^{2\cdot\tilde{N}(n+1)+2}\cdot l_{n}\cdot q_{n}=\frac{\tilde{q}_{n+1}}{N}$.
Finally, we introduce the abbreviation 
\[
\mathtt{a}_{n}=\left(\mathtt{B}^{(n)}_{0}\right)^{K\cdot N^{2\tilde{N}(n+1)+2}}.
\]

We use these pre-$(n+1)$-blocks to construct \emph{$(n+1)$-blocks}
$\mathtt{B}_{i,j}^{(n+1)}$ of $K$ different types (the index $i$
indicates the type, $j=1,\dots,\lambda_{n+1}=d_{n+1}e_{n+1}$
numbers the $(n+1)$-blocks of that type consecutively): 

\begin{eqnarray*}
\begin{array}{lll}
\mbox{\ensuremath{(n+1)}-blocks of type \ensuremath{1}:} &  & \mathtt{B}_{1,1}^{(n+1)}=\mathtt{A}_{1,1}\mathtt{A}_{2,1}\dots\mathtt{A}_{K,1}\mathtt{a}_{n},\\
 &  & \mathtt{B}_{1,2}^{(n+1)}=\mathtt{A}_{1,2}\mathtt{A}_{2,2}\dots\mathtt{A}_{K,2}\mathtt{a}_{n},\\
 &  & \mathtt{B}_{1,3}^{(n+1)}=\mathtt{A}_{1,3}\mathtt{A}_{2,3}\dots\mathtt{A}_{K,3}\mathtt{a}_{n},\\
 &  & \;\vdots\\
 &  & \mathtt{B}_{1,\lambda_{n+1}}^{(n+1)}=\mathtt{A}_{1,\lambda_{n+1}}\mathtt{A}_{2,\lambda_{n+1}}\dots\mathtt{A}_{K,\lambda_{n+1}}\mathtt{a}_{n}\\
\mbox{\ensuremath{(n+1)}-blocks of type \ensuremath{2}:} &  & \mathtt{B}_{2,1}^{(n+1)}=\mathtt{A}_{2,1}\mathtt{A}_{3,1}\dots\mathtt{A}_{K,1}\mathtt{A}_{1,2}\mathtt{a}_{n},\\
 &  & \mathtt{B}_{2,2}^{(n+1)}=\mathtt{A}_{2,2}\mathtt{A}_{3,2}\dots\mathtt{A}_{K,2}\mathtt{A}_{1,3}\mathtt{a}_{n},\\
 &  & \mathtt{B}_{2,3}^{(n+1)}=\mathtt{A}_{2,3}\mathtt{A}_{3,3}\dots\mathtt{A}_{K,3}\mathtt{A}_{1,4}\mathtt{a}_{n},\\
 &  & \;\vdots\\
 &  & \mathtt{B}_{2,\lambda_{n+1}}^{(n+1)}=\mathtt{A}_{2,\lambda_{n+1}}\mathtt{A}_{3,\lambda_{n+1}}\dots\mathtt{A}_{1,\lambda_{n+1}}\mathtt{a}_{n}\\
\vdots &  & \;\vdots\\
\mbox{(\ensuremath{n}+1)-blocks of type \ensuremath{K}:} &  & \mathtt{B}_{K,1}^{(n+1)}=\mathtt{A}_{K,1}\mathtt{A}_{1,2}\dots\mathtt{A}_{K-1,2}\mathtt{a}_{n},\\
 &  & \mathtt{B}_{K,2}^{(n+1)}=\mathtt{A}_{K,2}\mathtt{A}_{1,3}\dots\mathtt{A}_{K-1,3}\mathtt{a}_{n},\\
 &  & \mathtt{B}_{K,3}^{(n+1)}=\mathtt{A}_{K,3}\mathtt{A}_{1,4}\dots\mathtt{A}_{K-1,4}\mathtt{a}_{n},\\
 &  & \;\vdots\\
 &  & \mathtt{B}_{K,\lambda_{n+1}}^{(n+1)}=\mathtt{A}_{K,\lambda_{n+1}}\mathtt{A}_{1,\lambda_{n+1}+1}\dots\mathtt{A}_{K-1,\lambda_{n+1}+1}\mathtt{a}_{n}
\end{array}
\end{eqnarray*}
Moreover, we define an additional $(n+1)$-block $\mathtt{B}_{0}^{(n+1)}$
which will play the role of a marker: 
\[
\mathtt{B}_{0}^{(n+1)}=\mathtt{A}_{1,\lambda_{n+1}+2}\dots\mathtt{A}_{K,\lambda_{n+1}+2}\mathtt{a}_{n},
\]
where the pre-$(n+1)$-blocks $\mathtt{A}_{i,\lambda_{n+1}+2}$, $i=1,\dots, K$,
are not used in any other $(n+1)$-block. In total, there are $Kd_{n+1}e_{n+1}+1$
many $(n+1)$-blocks. Since each of the pre-$(n+1)$-blocks contains
each $n$-block $\mathtt{B}^{(n)}_{i}$, $1\leq i\leq N$, exactly $N^{2\tilde{N}(n+1)+2}$
many times, and each $(n+1)$-block contains $\mathtt{B}^{(n)}_0$ exactly $KN^{2\tilde{N}(n+1)+2}$ times, every $(n+1)$-block is uniform in the $n$-blocks. 
\begin{lem}[Distance between $(n+1)$-blocks in the odometer-based system] \label{lem:distODn+1}
 For every $i_{1},i_{2}\in\{1,\dots,K\}$ and every $j\in\{1,\dots,d_{n+1} e_{n+1}\}$
we have 
\begin{equation}
\overline{f}\left(\mathtt{B}_{i_{1},j}^{(n+1)},\mathtt{B}_{i_{2},j}^{(n+1)}\right)\leq\left(\frac{N}{N+1}\right)\cdot\frac{\lvert i_{2}-i_{1}\rvert}{K}.\label{eq:od_distance}
\end{equation}
\end{lem}

\begin{proof}
Observe that 
\[
\lvert\mathtt{a}_{n}\rvert=\frac{\mathtt{h}_{n+1}}{N+1},
\]
due to the uniformity of $n$-blocks within the $(n+1)$-blocks. Thus the pre-$(n+1)$-block part of each $\mathtt{B}_{i,j}^{(n+1)}$, that is, the part before $\mathtt{a}_n$, forms a fraction $\frac{N}{N+1}$ of $\mathtt{B}_{i,j}^{(n+1)}$. Without loss of generality,
let $i_{2}>i_{1}$. We note that $\mathtt{B}_{i_1,j}^{(n+1)}$ and $\mathtt{B}_{i_2,j}^{(n+1)}$ have 
\[
\mathtt{A}_{i_2,j}\mathtt{A}_{i_2+1,j}\cdots \mathtt{A}_{K,j}\mathtt{A}_{1,j+1}\cdots \mathtt{A}_{i_1-1,j+1}
\]
as a common substring of their pre-$(n+1)$-block parts. This substring forms a fraction $\frac{K-(i_2-i_1)}{K}$ of the pre-$(n+1)$-block part of each of $\mathtt{B}_{i_1,j}^{(n+1)}$ and $\mathtt{B}_{i_2,j}^{(n+1)}$.
Therefore (\ref{eq:od_distance}) holds.
\end{proof}
\begin{lem}[Distance between $(n+1)$-blocks in the circular system]
\label{lem:DistN} Let $\mathtt{B}_{1}^{(n+1)}$ and $\mathtt{B}_{2}^{(n+1)}$
be $(n+1)$-blocks and $J$ be the number of pre-$(n+1)$-blocks both
have in common. Then for any sequences $\mathcal{B}$ and $\overline{\mathcal{B}}$
of at least $q_{n+1}/R_{n+1}$ consecutive symbols in $\mathcal{B}_{1}^{(n+1)}$
and $\mathcal{B}_{2}^{(n+1)}$ we have 
\[
\overline{f}\left(\mathcal{B},\overline{\mathcal{B}}\right)\geq\left(1-\frac{J}{K}\right)\alpha-E_{n+1}, \ \text{ where } E_{n+1} \coloneqq \beta_{n+1}+ \frac{4}{N}+\frac{4}{l_{n-1}}
\]
with $\beta_{n+1}$ as in equation (\ref{eq:pren}).
\end{lem}

\begin{proof}
In order to get the estimate in the
circular system we recall that under the circular operator $\mathcal{C}_{n}$
the whole $(n+1)$-word consists of $q_{n}$ many $2$-subsections
which differ from each other just in the exponents of the newly introduced spacers
$b$ and $e$. Then we consider $\mathcal{B}$ and $\overline{\mathcal{B}}$ to be a concatenation of complete $2$-subsections
ignoring incomplete ones at the ends which constitute
a fraction of at most $2/l_{n-1}$ of the total length of $\mathcal{B}$ and $\overline{\mathcal{B}}$
by equation (\ref{eq:sn}). In the following consideration we ignore the marker segments which amount to a
fraction of $\frac{1}{N+1}$ of the total length.  Accordingly, we consider $\mathcal{B}$ and $\overline{\mathcal{B}}$
to be a concatenation of complete pre-$(n+1)$-blocks $\mathcal{A}_{h}$. Using the estimate from equation (\ref{eq:pren}) we apply Corollary \ref{cor:alpha separated blocks} with $\tilde{f} \geq 1-\frac{J}{K}$ and obtain
\[
\overline{f}\left(\mathcal{B},\overline{\mathcal{B}}\right)\geq\left(1-\frac{J}{K}\right)\cdot\left(\alpha- \beta_{n+1}\right) - \frac{2}{N}-\frac{4}{l_{n-1}}-\frac{2}{N+1},
\]
which yields the claim.
\end{proof}

\subsection{Induction step: Construction of $(n+m)$-blocks} \label{subsec:IStep}

We will follow the inductive scheme for the construction of $(n+m)$-blocks described in this subsection
for $2\leq m < p$, where $p>2$ is the smallest number such that 
\begin{equation}
\left(1-\frac{1}{K}\right)^{p}<\frac{\varepsilon}{2}.
\end{equation} 
Assume that in our inductive construction we have constructed
$K\lambda_{n+m-1}$ many \emph{$(n+m-1)$-blocks}
$\mathtt{B}_{i,j}^{(n+m-1)}$ of $K$ different types (once again, the index $1\leq i\leq K$ indicates the type, $1\leq j\leq \lambda_{n+m-1}$
numbers the $(n+m-1)$-blocks of that type consecutively), where for $m=2$ the $(n+1)$-blocks are the ones constructed in Subsection \ref{subsec:initial} and for $m>2$ the $(n+m-1)$-blocks are constructed according to the following formula (with $\lambda=\lambda_{n+m-1}$)

\begin{align*}
& \ \ \ \ \ \ \ (n+m-1)\text{-blocks of type }1:   \\
 & \mathtt{B}_{1,1}^{(n+m-1)}=\mathtt{A}_{1,1}^{(n+m-1)}\mathtt{A}_{2,2}^{(n+m-1)}\dots\mathtt{A}_{K,K}^{(n+m-1)}\mathtt{a}_{n+m-2},\\
 & \mathtt{B}_{1,2}^{(n+m-1)}=\mathtt{A}_{1,K+1}^{(n+m-1)}\mathtt{A}_{2,K+2}^{(n+m-1)}\dots\mathtt{A}_{K,2K}^{(n+m-1)}\mathtt{a}_{n+m-2},\\
 & \;\vdots\\
 & \mathtt{B}_{1,\lambda}^{(n+m-1)}=\mathtt{A}_{1,(\lambda-1)K+1}^{(n+m-1)}\mathtt{A}_{2,(\lambda-1)K+2}^{(n+m-1)}\dots\mathtt{A}_{K,\lambda K}^{(n+m-1)}\mathtt{a}_{n+m-2},\\
& \ \ \ \ \ \ \ (n+m-1)\text{-blocks of type }2: \\
& \mathtt{B}_{2,1}^{(n+m-1)}=\mathtt{A}_{1,2}^{(n+m-1)}\mathtt{A}_{2,3}^{(n+m-1)}\dots\mathtt{A}_{K,K+1}^{(n+m-1)}\mathtt{a}_{n+m-2},\\
 & \mathtt{B}_{2,2}^{(n+m-1)}=\mathtt{A}_{1,K+2}^{(n+m-1)}\mathtt{A}_{2,K+3}^{(n+m-1)}\dots\mathtt{A}_{K,2K+1}^{(n+m-1)}\mathtt{a}_{n+m-2},\\
  & \;\vdots\\
 & \mathtt{B}_{2,\lambda}^{(n+m-1)}=\mathtt{A}_{1,(\lambda-1)K+2}^{(n+m-1)}\mathtt{A}_{2,(\lambda-1)K+3}^{(n+m-1)}\dots\mathtt{A}_{K,\lambda K+1}^{(n+m-1)}\mathtt{a}_{n+m-2},\\
 & \;\vdots\\
& \ \ \ \ \ \ \ (n+m-1)\text{-blocks of type }K:  \\
& \mathtt{B}_{K,1}^{(n+m-1)}=\mathtt{A}_{1,K}^{(n+m-1)}\mathtt{A}_{2,K+1}^{(n+m-1)}\dots\mathtt{A}_{K,2K-1}^{(n+m-1)}\mathtt{a}_{n+m-2},\\
  & \mathtt{B}_{K,2}^{(n+m-1)}=\mathtt{A}_{1,2K}^{(n+m-1)}\mathtt{A}_{2,2K+1}^{(n+m-1)}\dots\mathtt{A}_{K,3K-1}^{(n+m-1)}\mathtt{a}_{n+m-2},\\
  & \;\vdots \\
& \mathtt{B}_{K,\lambda}^{(n+m-1)}=\mathtt{A}_{1,\lambda K}^{(n+m-1)}\mathtt{A}_{2,\lambda K+1}^{(n+m-1)}\dots\mathtt{A}_{K,(\lambda +1)K-1}^{(n+m-1)}\mathtt{a}_{n+m-2},
\end{align*}
with marker segment $\mathtt{a}_{n+m-2} = \left(\mathtt{B}_{0}^{(n+m-2)}\right)^{\bar{N}(n+m-1)}$, where $\bar{N}(n+m-1)$ is chosen according to (\ref{eq:N(n+m)}) to guarantee uniformity of $(n+m-2)$-blocks in the $(n+m-1)$-blocks. Moreover, we have an additional marker block 
\[
\mathtt{B}_{0}^{(n+m-1)}=\mathtt{A}_{1,\lambda_{n+m-1}+2}^{(n+m-1)}\mathtt{A}_{2,\lambda_{n+m-1}+2}^{(n+m-1)}\dots\mathtt{A}_{K,\lambda_{n+m-1}+2}^{(n+m-1)}\mathtt{a}_{n+m-2}.
\]

These blocks are defined using pre-$(n+m-1)$-blocks $\mathtt{A}_{i,j}^{(n+m-1)}$ with 
\begin{equation}
\overline{f}\left(\mathtt{A}_{i_{1},j}^{(n+m-1)},\mathtt{A}_{i_{2},j}^{(n+m-1)}\right)\leq\sum_{u=1}^{m-2}\left(\frac{1}{N(n+u-1)+1}+\frac{1}{d_{n+u}}\right).\label{eq:I1}
\end{equation}
in case of $m>2$ (note that the assumption is void in case of $m=2$). Moreover, for any sequences $\mathcal{B}$ and $\overline{\mathcal{B}}$
of at least $\frac{q_{n+m-1}}{R_{n+m-1}}$ consecutive symbols in
$\mathcal{B}_{i_{1},j_{1}}^{(n+m-1)}$ and $\mathcal{B}_{i_{2},j_{2}}^{(n+m-1)}$
for some $i_{1},i_{2}\in\left\{ 0,1,\dots,K\right\} $, $i_1 \leq i_2$, and $j_{1},j_{2}\in\left\{ 1,\dots,\lambda_{n+m-1}\right\} $
we assume
\begin{equation}
\overline{f}\left(\mathcal{B},\overline{\mathcal{B}}\right)\geq\alpha-E_{n+m-1}\;\text{ for \ensuremath{i_{1}=i_{2}} and \ensuremath{j_{1}\neq j_{2}},}\label{eq:I2}
\end{equation}

\begin{equation}
\overline{f}\left(\mathcal{B},\overline{\mathcal{B}}\right)\geq\left(1-\left(1-\frac{1}{K}\right)^{m-1}\right)\alpha-E_{n+m-1}\;\text{ for \ensuremath{i_{1}<i_{2}} and \ensuremath{j_{1}=j_{2}} or \ensuremath{j_{1}=j_{2}+1},}\label{eq:I3}
\end{equation}
 
\begin{equation}
\overline{f}\left(\mathcal{B},\overline{\mathcal{B}}\right)\geq\alpha-E_{n+m-1}\;\text{ for \ensuremath{i_{1}<i_{2}} and all other cases of \ensuremath{j_{1}\neq j_{2}},}\label{eq:I4}
\end{equation}
 where 

\begin{equation}
E_{n+m-1}=E_{n+1}+\sum_{i=1}^{m-2}\left(\frac{6}{R_{n+i}}+\frac{8}{u_{n+i}}+\frac{17}{\sqrt{e_{n+i}}}\right).\label{eq:I5}
\end{equation}
 
We point out that assumptions (\ref{eq:I2})-(\ref{eq:I5}) hold in case of $m=2$ by Lemma \ref{lem:DistN}. \\

In order to continue the inductive construction we
define \emph{grouped $(n+m-1)$-blocks} of type $i$ by concatenating
$d_{n+m-1}$ many $(n+m-1)$-blocks of type $i$: 
\[
\mathtt{G}_{i,s}^{(n+m-1)}=\mathtt{B}_{i,s\cdot d_{n+m-1}+1}^{(n+m-1)}\mathtt{B}_{i,s\cdot d_{n+m-1}+2}^{(n+m-1)}\dots\mathtt{B}_{i,(s+1)\cdot d_{n+m-1}}^{(n+m-1)},\;\;\text{for}\;s=0,\dots,e_{n+m-1}-1.
\]

\begin{lem}[Distance between grouped $(n+m-1)$-blocks in the circular system]
\label{lem:ISGroup} Let $\mathcal{G}$ and $\overline{\mathcal{G}}$
be sequences of at least $u_{n+m-1}l_{n+m-1}q_{n+m-1}$ consecutive
symbols in $\mathcal{G}_{i_{1},s_{1}}^{(n+m-1)}$ and $\mathcal{G}_{i_{2},s_{2}}^{(n+m-1)}$
for some $i_{1},i_{2}\in\left\{ 1,\dots,K\right\} $ and $s_{1},s_{2}\in\left\{ 0,\dots,e_{n+m-1}-1\right\} $.
\begin{enumerate}
\item In case of $i_{1}=i_{2}$ and $s_{1}\neq s_{2}$ we have 
\[
\overline{f}\left(\mathcal{G},\overline{\mathcal{G}}\right)\geq \alpha-E_{n+m-1}-\frac{4}{R_{n+m-1}}-\frac{4R_{n+m-1}}{l_{n+m-1}}-\frac{4}{u_{n+m-1}}.
\]
\item In case of $i_{1}\neq i_{2}$ and $s_{1}=s_{2}$ we have 
\[
\overline{f}\left(\mathcal{G},\overline{\mathcal{G}}\right)\geq\left(1-\left(1-\frac{1}{K}\right)^{m-1}\right)\cdot\alpha-E_{n+m-1}-\frac{4}{R_{n+m-1}}-\frac{4R_{n+m-1}}{l_{n+m-1}}-\frac{4}{u_{n+m-1}}.
\]
\item In case of $i_{1}\neq i_{2}$ and $s_{1}\neq s_{2}$ we have 
\[
\overline{f}\left(\mathcal{G},\overline{\mathcal{G}}\right)\geq\alpha-E_{n+m-1}-\frac{4}{R_{n+m-1}}-\frac{4R_{n+m-1}}{l_{n+m-1}}-\frac{6}{u_{n+m-1}}.
\]
\end{enumerate}
\end{lem}

\begin{proof}
We factor $\mathcal{G}$ and $\overline{\mathcal{G}}$ into 2-subsections $\mathcal{C}_{n,j_1}\left(\mathtt{B}_{i_{1},t}^{(n+m-1)}\right)$ and $\mathcal{C}_{n,j_2}\left(\mathtt{B}_{i_{2},u}^{(n+m-1)}\right)$ for some $j_1,j_2 \in \{0,1,\dots,q_n-1\}$
omitting partial blocks at the ends which constitute a portion of at most $2/u_{n+m-1}$ of the total length of $\mathcal{G}$ and $\overline{\mathcal{G}}$. With the aid of Lemma \ref{lem:operators} we can transfer the estimates in equations (\ref{eq:I2})-(\ref{eq:I4}) to estimates on the $\overline{f}$ distance of substrings of at least $q_{n+m-1}l_{n+m-1}/R_{n+m-1}$ consecutive symbols in these strings of the form $\mathcal{C}_{n,j_1}\left(\mathtt{B}_{i_{1},t}^{(n+m-1)}\right)$ and $\mathcal{C}_{n,j_2}\left(\mathtt{B}_{i_{2},u}^{(n+m-1)}\right)$ respectively.

We now examine the particular situation of each part of the Lemma.
\begin{enumerate}
\item Since the $(n+m-1)$-blocks in $\mathcal{G}_{i_{1},s_{1}}^{(n+m-1)}$
and $\mathcal{G}_{i_{1},s_{2}}^{(n+m-1)}$ are of same type but different
pattern, Corollary \ref{cor:alpha separated blocks} (with $\tilde{f}=1$) and the modified version of equation (\ref{eq:I2}) yield
\[
\overline{f}\left(\mathcal{G},\overline{\mathcal{G}}\right)\geq \alpha-E_{n+m-1}-\frac{2}{R_{n+m-1}}-\frac{4R_{n+m-1}}{l_{n+m-1}}-\frac{2}{R_{n+m-1}}-\frac{4}{u_{n+m-1}}.
\]
\item In case of $m=2$ we note that $\mathtt{B}_{i_{1},t}^{(n+1)}$ and $\mathtt{B}_{i_{2},u}^{(n+1)}$
in $\mathtt{G}_{i_{1},s_{1}}^{(n+1)}$ and $\mathtt{G}_{i_{2},s_{1}}^{(n+1)}$
respectively have at most $J=\max\left(\lvert i_{2}-i_{1}\rvert,\:K-\lvert i_{2}-i_{1}\rvert\right)$
pre-$(n+1)$-blocks in common. Then we apply Lemma \ref{lem:DistN}, Lemma \ref{lem:operators}, and Corollary \ref{cor:alpha separated blocks} (with $\tilde{f}=1$)
to conclude the inequality.

In order to obtain a lower bound on the $\overline{f}$ distance in case of $m>2$ we
note that $\mathtt{B}_{i_{1},t}^{(n+m-1)}$ and $\mathtt{B}_{i_{2},u}^{(n+m-1)}$
in $\mathtt{G}_{i_{1},s_{1}}^{(n+m-1)}$ and $\mathtt{G}_{i_{2},s_{1}}^{(n+m-1)}$ could fall under the situation of the worst possible estimate in equation (\ref{eq:I3}). By another application of Lemma \ref{lem:operators} and Corollary \ref{cor:alpha separated blocks} (with $\tilde{f}=1$) we get the claim.

\item Without loss of generality let $i_2>i_1$. Once again, we start with the proof in case of $m=2$. There is at most one pair of $(n+1)$-blocks in $\mathtt{G}_{i_{1},s_{1}}^{(n+1)}$ and $\mathtt{G}_{i_{2},s_{2}}^{(n+1)}$ respectively that have $i_{2}-i_{1}$
many pre-$(n+1)$-blocks in common while all other pairs have no pair
in common. This corresponds to a proportion of at most $2/u_{n+1}$. Then we use Lemma \ref{lem:DistN}, Lemma \ref{lem:operators}, and Proposition \ref{prop:symbol by block replacement} (with $\tilde{f}\geq 1-\frac{2}{u_{n+1}}$) to obtain 
\begin{align*}
\overline{f}\left(\mathcal{G},\overline{\mathcal{G}}\right) & \geq\left(1-\frac{2}{u_{n+1}}+\frac{2}{u_{n+1}}\left(1-\frac{i_{2}-i_{1}}{K}\right)\right)\alpha-E_{n+1}-\frac{4}{R_{n+1}}-\frac{4R_{n+1}}{l_{n+1}}-\frac{4}{u_{n+1}}\\
 & \geq \alpha-E_{n+1}-\frac{4}{R_{n+1}}-\frac{4R_{n+1}}{l_{n+1}}-\frac{6}{u_{n+1}}.
\end{align*}

We proceed with the case of $m>2$. There is at most one pair of $n$-blocks $\mathtt{B}_{i_{1},j_{1}}^{(n+m-1)}$
and $\mathtt{B}_{i_{2},j_{2}}^{(n+m-1)}$ in $\mathtt{G}_{i_{1},s_{1}}^{(n+m-1)}$
and $\mathtt{G}_{i_{2},s_{2}}^{(n+m-1)}$ respectively with $j_{1}=j_{2}+1$
while for all other pairs $j_{1}\neq j_{2}$ and $j_{1}\neq j_{2}+1$. Since this corresponds to a proportion of at most $2/u_{n+m-1}$, we use the modified versions of equations (\ref{eq:I3}) and (\ref{eq:I4}) and Proposition \ref{prop:symbol by block replacement} (with $\tilde{f}\geq 1-\frac{2}{u_{n+m-1}}$) to obtain the estimate on $\overline{f}\left(\mathcal{G},\overline{\mathcal{G}}\right)$ by the same calculation as above replacing the subscripts $n+1$ by $n+m-1$ and the term $1-\frac{i_{2}-i_{1}}{K}$ by $1-\left(1-\frac{1}{K}\right)^{m-1}$. \\

In all other cases $s_{1}\neq s_{2}$ all pairs of $n$-blocks $\mathtt{B}_{i_{1},j_{1}}^{(n+m-1)}$
and $\mathtt{B}_{i_{2},j_{2}}^{(n+m-1)}$ in $\mathtt{G}_{i_{1},s_{1}}^{(n+m-1)}$
and $\mathtt{G}_{i_{2},s_{2}}^{(n+m-1)}$, respectively, have $j_{1}\neq j_{2}$
and $j_{1}\neq j_{2}+1$. Then we use the same estimate as in part (1).
\end{enumerate}
\end{proof}
Let us introduce the notation 
\[
\gamma_{n+m-1}\coloneqq E_{n+m-1}+\frac{4}{R_{n+m-1}}+\frac{4R_{n+m-1}}{l_{n+m-1}}+\frac{6}{u_{n+m-1}}.
\]

We continue the inductive construction by building $2K\lambda_{n+m}=2Kd_{n+m}e_{n+m}$
many \emph{pre-$(n+m)$-blocks} $\mathtt{A}_{i,j}^{(n+m)}$, where
$j\in \{ 1,\dots,2\lambda_{n+m} \}$ stands for the $j$-th Feldman pattern and $i\in\left\{ 1,\dots,K\right\} $
indicates the type of $(n+m-1)$-blocks used. We let $\mathtt{A}_{i,j}^{(n+m)}$ be the $j$-th Feldman
pattern built of the grouped $(n+m-1)$-blocks of type $i$. We point
out that $\mathtt{A}_{i,j}^{(n+m)}$ contains each $(n+m-1)$-block
of type $i$ 
\begin{equation}\label{eq:N(n+m)}
\bar{N}(n+m)=\left(e_{n+m-1}\right)^{4K\lambda_{n+m}+2}
\end{equation}
many times because it is uniform in the building blocks $\mathtt{G}_{i,s}^{(n+m-1)}$
by construction of the Feldman patterns in Proposition \ref{prop:Feldman}
and each $(n+m-1)$-block of type $i$ is contained in exactly one
grouped $(n+m-1)$-block. Moreover, this number of occurrences is
the same for every chosen Feldman pattern $j$. We will denote the
length of the circular image $\mathcal{A}_{i,j}^{(n+m)}=\mathcal{C}_{n+m-1,k}\left(\mathtt{A}^{(n+m)}_{i,j}\right)$, $k\in \{0,\dots,q_{n+m-1}-1\}$, by $\tilde{q}_{n+m}$.

\begin{lem}[Closeness and Separation of pre-$(n+m)$-blocks of same Feldman pattern]
\label{lem:ISClosSep} For every $j\in\left\{ 1,\dots,2K d_{n+m} e_{n+m}\right\} $
and $i_{1},i_{2}\in\left\{ 1,\dots,K\right\} $ we have
\[
\overline{f}\left(\mathtt{A}_{i_{1},j}^{(n+m)},\mathtt{A}_{i_{2},j}^{(n+m)}\right)\leq\sum_{u=1}^{m-1}\left(\frac{1}{N(n+u-1)+1}+\frac{1}{d_{n+u}}\right)
\]
in the odometer-based system. In the circular system we have for $i_{1}\neq i_{2}$
and for any sequences $\mathcal{A}$ and $\overline{\mathcal{A}}$
of at least $\tilde{q}_{n+m}/e_{n+m-1}$ consecutive symbols
in $\mathcal{A}_{i_{1},j}^{(n+m)}$ and $\mathcal{A}_{i_{2},j}^{(n+m)}$,
respectively, that 
\[
\overline{f}\left(\mathcal{A},\overline{\mathcal{A}}\right)
\geq \left(1-\left(1-\frac{1}{K}\right)^{m-1}\right)\alpha-\gamma_{n+m-1}-\frac{2}{e_{n+m-1}} -\frac{2}{u_{n+m-1}}. 
\]
\end{lem}

\begin{proof} Without loss of generality, let $i_{2}>i_{1}$. 

We start with the proof of the first statement in case of $m=2$: Ignoring the marker
segment $\mathtt{a}_{n}$ we note that for every $s=0,\dots,e_{n+1}-1$
the grouped $n$-block $\mathtt{G}_{i_{2},s}^{(n+1)}$ is obtained
from $\mathtt{G}_{i_{1},s}^{(n+1)}$ by a shift of
$i_{2}-i_{1}$ many pre-$(n+1)$-blocks. Since a grouped $(n+1)$-block
consists of $Kd_{n+1}$ many pre-$(n+1)$-blocks, the grouped
$(n+1)$-blocks $\mathtt{G}_{i_{1},s}^{(n+1)},\mathtt{G}_{i_{2},s}^{(n+1)}$
can be matched on a portion of at least $1-\frac{i_{2}-i_{1}}{K d_{n+1}}$ of the part of pre-$(n+1)$-blocks. Thus, we have 
\[
\overline{f}\left(\mathtt{G}_{i_{1},s}^{(n+1)},\mathtt{G}_{i_{2},s}^{(n+1)}\right)\leq\frac{1}{N+1}+\frac{N}{N+1}\cdot\frac{i_{2}-i_{1}}{K d_{n+1}},
\]
which yields the claim because $\mathtt{A}_{i_{1},j}^{(n+2)}$
and $\mathtt{A}_{i_{2},j}^{(n+2)}$ are constructed as the same Feldman
pattern with these grouped $(n+1)$-blocks of different types as building
blocks.

Proof of the first statement in case of $m>2$: We ignore the marker blocks 
which occupy a fraction of 
\[
\frac{1}{K\lambda_{n+m-2}+1}=\frac{1}{N(n+m-2)+1}
\]
of each $(n+m-1)$-block (since there are $1+K\lambda_{n+m-2}$ many $(n+m-2)$-blocks, it follows from the uniformity of $(n+m-2)$-blocks within each $(n+m-1)$-block that the marker segment occupies a fraction $1/(1+K\lambda_{n+m-2})$ of the $(n+m-1)$-block). Let
$M$ denote the right-hand side from inequality (\ref{eq:I1}), i.e. an upper bound for the $\overline{f}$ distance between pre-$(n+m-1)$-blocks of type
$i_{1}$ and $i_{2}$ and the same Feldman pattern.
Since in the definition of $(n+m-1)$-blocks of types $i_1$ and $i_2$ the used Feldman patterns are shifted by $i_2-i_1$, we obtain for every $s\in\{0,\dots,e_{n+m-1}-1\}$ 
$\mathtt{G}_{i_{1},s}^{(n+m-1)}$ and $\mathtt{G}_{i_{2},s}^{(n+m-1)}$
\begin{align*}
& \overline{f}\left(\mathtt{G}_{i_{1},s}^{(n+m-1)},\mathtt{G}_{i_{2},s}^{(n+m-1)}\right) \\
\leq & \frac{1}{N(n+m-2)+1}+\frac{\left(Kd_{n+m-1}-\lvert i_{1}-i_{2}\rvert\right)}{Kd_{n+m-1}}\cdot M+\frac{\lvert i_{1}-i_{2}\rvert}{Kd_{n+m-1}} \\
\leq & \sum_{u=1}^{m-1}\left(\frac{1}{N(n+u-1)+1}+\frac{1}{d_{n+u}}\right)
\end{align*}
using equation (\ref{eq:I1}). This yields the claim because $\mathtt{A}_{i_{1},j}^{(n+m)}$
and $\mathtt{A}_{i_{2},j}^{(n+m)}$ are constructed as the same Feldman
pattern with these grouped $(n+m-1)$-blocks of different type as
building blocks.

The second statement follows from Lemma \ref{lem:samePattern} and
Lemma \ref{lem:ISGroup}.

\end{proof}

We will need a statement on the $\overline{f}$ distance of different Feldman
patterns in the circular system only, but not in the odometer-based
system.

\begin{lem}[Separation of pre-$(n+m)$-blocks of different Feldman patterns]
\label{lem:ISSepDiff} For any sequences $\mathcal{A}$ and $\overline{\mathcal{A}}$
of at least $\tilde{q}_{n+m}/e_{n+m-1}$ consecutive symbols
in $\mathcal{A}_{i_{1},j_{1}}^{(n+m)}$ and $\mathcal{A}_{i_{2},j_{2}}^{(n+m)}$
for some $i_{1},i_{2}\in\left\{ 1,\dots,K\right\} $ and $j_{1},j_{2}\in\left\{ 1,\dots,2Kd_{n+m}e_{n+m}\right\} $,
$j_{1}\neq j_{2}$, we have 
\[
\overline{f}\left(\mathcal{A},\overline{\mathcal{A}}\right)\geq \alpha-\gamma_{n+m-1}-\frac{13}{\sqrt{e_{n+m-1}}}-\frac{2}{u_{n+m-1}}.
\]
\end{lem}

\begin{proof}
The result follows from Proposition \ref{prop:Feldman} and Lemma
\ref{lem:ISGroup}.
\end{proof}
Then we set 
\[
\beta_{n+m}\coloneqq  \gamma_{n+m-1}+\frac{13}{\sqrt{e_{n+m-1}}}+\frac{2}{u_{n+m-1}}.
\]

We continue our construction process by building \emph{$(n+m)$-blocks}
$\mathtt{B}_{i,j}^{(n+m)}$ of $K$ different types (once again, the
index $1\leq i \leq K$ indicates the type, $1 \leq j \leq \lambda=\lambda_{n+m}$ numbers the $(n+m)$-blocks of
that type consecutively) using the formula from the beginning of Subsection \ref{subsec:IStep} (with $n+m-1$ replaced by $n+m$).  Moreover, we define an additional $(n+m)$-block $\mathtt{B}_{0}^{(n+m)}$
which will play the role of a marker again: 
\[
\mathtt{B}_{0}^{(n+m)}=\mathtt{A}_{1,\lambda_{n+m}+2}^{(n+m)}\mathtt{A}_{2,\lambda_{n+m}+2}^{(n+m)}\dots\mathtt{A}_{K,\lambda_{n+m}+2}^{(n+m)}\mathtt{a}_{n+m-1},
\]
where the pre-$(n+m)$-blocks $\mathtt{A}_{i,\lambda_{n+m}+2}^{(n+m)}$
are not used in any other $(n+m)$-block. Hence, there are $K\lambda_{n+m}+1$
many $(n+m)$-blocks in total. We also note that
each $(n+m)$-block $\mathtt{B}_{i,s}^{(n+m)}$ contains exactly one
pre-$(n+m)$-block of each type. So, it is uniform in the $(n+m-1)$-blocks
by our observation above. 

\begin{lem}[Distance between $(n+m)$-blocks in the circular system]
\label{lem:blockNp} Let $\mathcal{B}$ and $\overline{\mathcal{B}}$
be sequences of at least $q_{n+m}/R_{n+m}$ consecutive symbols
in $\mathcal{B}_{i_{1},j_{1}}^{(n+m)}$ and $\mathcal{B}_{i_{2},j_{2}}^{(n+m)}$
for some $i_{1},i_{2}\in\left\{ 0,1,\dots,K\right\} $ and $j_{1},j_{2}\in\left\{ 1,\dots,\lambda_{n+m}\right\} $.
\begin{enumerate}
\item For blocks of same type: For $i_{1}=i_{2}$ and $j_{1}\neq j_{2}$
we have
\[
\overline{f}\left(\mathcal{B},\overline{\mathcal{B}}\right)\geq \alpha-\beta_{n+m}-\frac{4}{e_{n+m-1}}-\frac{4}{l_{n+m-2}}.
\]
\item For blocks of different type: For $i_{1}<i_{2}$ we have 
\[
\overline{f}\left(\mathcal{B},\overline{\mathcal{B}}\right)\geq\left(1-\left(1-\frac{1}{K}\right)^{m}\right)\cdot \alpha- \beta_{n+m}-\frac{4}{e_{n+m-1}}-\frac{4}{l_{n+m-2}}
\]
in case of $j_{1}=j_{2}$ or $j_{1}=j_{2}+1$. For all other cases of $j_{1}\neq j_{2}$ we have 
\[
\overline{f}\left(\mathcal{B},\overline{\mathcal{B}}\right)\geq \alpha-\beta_{n+m}-\frac{4}{e_{n+m-1}}-\frac{4}{l_{n+m-2}}.
\]
\end{enumerate}
\end{lem}

\begin{proof}
As in the proof of Lemma \ref{lem:DistN} we consider $\mathcal{B}$ and $\overline{\mathcal{B}}$ to be a concatenation of complete $2$-subsections ignoring incomplete ones at the beginnings and ends which constitute a fraction of at most $2/l_{n+m-2}$ of the total length of $\mathcal{B}$ and $\overline{\mathcal{B}}$
by equation (\ref{eq:sn}). In the following consideration we ignore the marker segments $\mathtt{a}_{n+m-1}$ which amount to a
fraction of $\frac{1}{1+K\lambda_{n+m-1}}<\frac{1}{e_{n+m-1}}$ of the total length due to uniformity.  Accordingly, we consider $\mathcal{B}$ and $\overline{\mathcal{B}}$ to be a concatenation of complete pre-$(n+m)$-blocks $\mathcal{A}^{(n+m)}_{i_1,h_1}$ and $\mathcal{A}^{(n+m)}_{i_2,h_2}$ respectively. Finally, we examine the particular situation of each case of this Lemma: 
\begin{enumerate}
\item We note that all Feldman patterns for pre-$(n+m)$-blocks used in $\mathtt{B}_{i_{1},j_{1}}^{(n+m)}$
and $\mathtt{B}_{i_{1},j_{2}}^{(n+m)}$ are different from each other.
Accordingly, we apply Lemma \ref{lem:ISSepDiff} and Corollary \ref{cor:alpha separated blocks} (with $\tilde{f}=1$). 
\item In case of $i_{1}< i_{2}$ the blocks $\mathtt{B}_{i_{1},j_{2}}^{(n+m)}$
and $\mathtt{B}_{i_{2},j_{2}}^{(n+m)}$ have $K-( i_{2}-i_{1})$
many Feldman patterns in common. Then we use the second statement
of Lemma \ref{lem:ISClosSep} to estimate their $\overline{f}$ distance
while we use Lemma \ref{lem:ISSepDiff} for the $ i_{2}-i_{1}$
many differing Feldman patterns. Altogether we conclude with the aid of Proposition \ref{prop:symbol by block replacement} that $\overline{f}\left(\mathcal{B},\overline{\mathcal{B}}\right)$ is at least
\begin{align*}
 & \left(\frac{ i_{2}-i_{1}}{K}+\left( 1-\left(1-\frac{1}{K}\right)^{m-1}\right)\frac{K-( i_{2}-i_{1})}{K}\right)\alpha-\beta_{n+m}-\frac{4}{e_{n+m-1}}-\frac{4}{l_{n+m-2}}\\
\geq & \left(1-\left(1-\frac{1}{K}\right)^{m}\right) \alpha-\beta_{n+m}-\frac{4}{e_{n+m-1}}-\frac{4}{l_{n+m-2}}.
\end{align*}
In our case of $i_{1}<i_{2}$ the blocks $\mathtt{B}_{i_{1},j_{2}+1}^{(n+m)}$
and $\mathtt{B}_{i_{2},j_{2}}^{(n+m)}$ have $ i_{2}-i_{1}$
many Feldman patterns in common. With the aid of the second part of
Lemma \ref{lem:ISClosSep}, Lemma \ref{lem:ISSepDiff}, and Proposition \ref{prop:symbol by block replacement} again we obtain that $\overline{f}\left(\mathcal{B},\overline{\mathcal{B}}\right)$ is at least
\begin{align*}
 & \left(\frac{i_{2}-i_{1}}{K}\left(1-\left(1-\frac{1}{K}\right)^{m-1}\right)+1-\frac{ i_{2}-i_{1}}{K}\right)\alpha-\beta_{n+m}-\frac{4}{e_{n+m-1}}-\frac{4}{l_{n+m-2}}\\
\geq & \left(1-\left(1-\frac{1}{K}\right)^{m}\right)\alpha-\beta_{n+m}-\frac{4}{e_{n+m-1}}-\frac{4}{l_{n+m-2}}.
\end{align*}
In all other cases of $j_{1}\neq j_{2}$ $\mathtt{B}_{i_{1},j_{1}}^{(n+m)}$
and $\mathtt{B}_{i_{2},j_{2}}^{(n+m)}$ do not have any Feldman pattern
in common. Hence, we use the same estimate as for the statement in part (1).
\end{enumerate}
\end{proof}
By the conditions on the sequence $\left(l_n\right)_{n\in \mathbb{N}}$ from equation (\ref{eq:lcond}) we have
\begin{align*}
& \beta_{n+m}+\frac{4}{e_{n+m-1}}+\frac{4}{l_{n+m-2}} \\
= & E_{n+m-1}+\frac{4}{R_{n+m-1}}+\frac{8}{u_{n+m-1}}+\frac{13}{\sqrt{e_{n+m-1}}}+\frac{4}{e_{n+m-1}}+\frac{4}{l_{n+m-2}}+\frac{4R_{n+m-1}}{l_{n+m-1}}\\
\leq & E_{n+m-1}+\frac{6}{R_{n+m-1}}+\frac{8}{u_{n+m-1}}+\frac{17}{\sqrt{e_{n+m-1}}}.
\end{align*}
Accordingly, we set 
\begin{align*}
E_{n+m} & \coloneqq E_{n+m-1}+\frac{6}{R_{n+m-1}}+\frac{8}{u_{n+m-1}}+\frac{17}{\sqrt{e_{n+m-1}}}\\
 & =E_{n+1}+\sum_{i=1}^{m-1}\left(\frac{6}{R_{n+i}}+\frac{8}{u_{n+i}}+\frac{17}{\sqrt{e_{n+i}}}\right).
\end{align*}

\begin{rem}
We note that the equations (\ref{eq:I1}) to (\ref{eq:I5}) are satisfied
on stage $n+m$. Hence, the induction step has been accomplished successfully. 
\end{rem}

\subsection{Final step: Construction of $\left(n+p\right)$-blocks}

Recall that we follow the inductive scheme described in the previous subsection
until 
\begin{equation*}
\left(1-\frac{1}{K}\right)^{p}<\frac{\varepsilon}{2},
\end{equation*}
and we have constructed pre-$(n+p)$-blocks $\mathtt{A}_{i,j}^{(n+p)}$
(of type $i$ and Feldman pattern $j$) with the following properties:
\begin{equation}
\overline{f}\left(\mathtt{A}_{i_{1},j}^{(n+p)},\mathtt{A}_{i_{2},j}^{(n+p)}\right)\leq\sum_{u=1}^{p-1}\left(\frac{1}{N(n+u-1)+1}+\frac{1}{d_{n+u}}\right);\label{eq:F1}
\end{equation}
for any sequences $\mathcal{A}$ and $\overline{\mathcal{A}}$ of
at least $\tilde{q}_{n+p}/e_{n+p-1}$ consecutive symbols
in $\mathcal{A}_{i_{1},j_{1}}^{(n+p)}$ and $\mathcal{A}_{i_{2},j_{2}}^{(n+p)}$, respectively,
for some $i_{1},i_{2}\in\left\{ 1,\dots,K\right\} $ and $j_{1},j_{2}\in\left\{ 1,\dots,2K\lambda_{n+p}\right\} $
we have

\begin{equation}
\overline{f}\left(\mathcal{A},\overline{\mathcal{A}}\right)\geq
\begin{cases}
\left(1-\frac{\varepsilon}{2}\right)\cdot\alpha -\beta_{n+p} & \text{ for all }i_{1}\neq i_{2}  \text{ and } j_1=j_2, \\
\alpha-\beta_{n+p} & \text{ for all }i_{1},i_{2}\text{ and }j_{1}\neq j_{2},
\end{cases}\label{eq:F2}
\end{equation}

where 
\begin{equation}
\beta_{n+p}=E_{n+p-1}+\frac{4}{R_{n+p-1}}+\frac{8}{u_{n+p-1}}+\frac{13}{\sqrt{e_{n+p-1}}}+\frac{4R_{n+p-1}}{l_{n+p-1}},\label{eq:F4a}
\end{equation}
and

\begin{equation}
E_{n+p-1}=E_{n+1}+\sum_{i=1}^{p-2}\left(\frac{6}{R_{n+i}}+\frac{8}{u_{n+i}}+\frac{17}{\sqrt{e_{n+i}}}\right).\label{eq:F4}
\end{equation}

By construction every pre-$(n+p)$-blocks $\mathtt{A}_{i,j}^{(n+p)}$
contains each $(n+p-1)$-block of type $i$ exactly $\bar{N}(n+p)$ many times
and this number of occurrences is the same for every chosen Feldman
pattern $j$. 

Then we construct $K$ many $(n+p)$-blocks 
\begin{align*}
(n+p)\text{-block of type }1: & \;\;\mathtt{B}_{1}^{(n+p)}=\mathtt{A}_{1,1}^{(n+p)}\mathtt{A}_{2,2}^{(n+p)}\dots\mathtt{A}_{K,K}^{(n+p)}\left(\mathtt{B}_{0}^{(n+p-1)}\right)^{\bar{N}(n+p)}\\
(n+p)\text{-block of type }2: & \;\;\mathtt{B}_{2}^{(n+p)}=\mathtt{A}_{2,1}^{(n+p)}\mathtt{A}_{3,2}^{(n+p)}\dots\mathtt{A}_{1,K}^{(n+p)}\left(\mathtt{B}_{0}^{(n+p-1)}\right)^{\bar{N}(n+p)}\\
(n+p)\text{-block of type }3: & \;\;\mathtt{B}_{3}^{(n+p)}=\mathtt{A}_{3,1}^{(n+p)}\mathtt{A}_{4,2}^{(n+p)}\dots\mathtt{A}_{2,K}^{(n+p)}\left(\mathtt{B}_{0}^{(n+p-1)}\right)^{\bar{N}(n+p)}\\
\vdots & \;\;\vdots\\
(n+p)\text{-block of type }\text{K}: & \;\;\mathtt{B}_{K}^{(n+p)}=\mathtt{A}_{K,1}^{(n+p)}\mathtt{A}_{1,2}^{(n+p)}\dots\mathtt{A}_{K-1,K}^{(n+p)}\left(\mathtt{B}_{0}^{(n+p-1)}\right)^{\bar{N}(n+p)}
\end{align*}
 
\begin{rem}
We note that every $(n+p)$-block contains exactly one pre-$(n+p)$-block
of each type. Hence, it is uniform in the $(n+p-1)$-blocks by our
observation above.
\end{rem}

\begin{lem}[Closeness of $(n+p)$-blocks in the Odometer-based System]
\label{lem:Fclos} For every $i_{1},i_{2}\in \{1,\dots,K\}$ we have
\begin{equation}
\overline{f}\left(\mathtt{B}_{i_{1}}^{(n+p)},\mathtt{B}_{i_{2}}^{(n+p)}\right)\leq\sum_{u=1}^{p-1}\left(\frac{1}{N(n+u-1)+1}+\frac{1}{d_{n+u}}\right).
\end{equation}
\end{lem}

\begin{proof}
Let $M$ denote the right-hand side of inequality (\ref{eq:F1}).
We observe that the Feldman patterns of pre-$(n+p)$-blocks are aligned
in all the $(n+p)$-blocks by construction. Moreover, the marker segments
are aligned as well and these occupy a fraction of $\frac{1}{N(n+p-1)+1}$
of each $(n+p)$-block by uniformity. Hence, by equation (\ref{eq:F1}) we have  
\[
\overline{f}\left(\mathtt{B}_{i_{1}}^{(n+p)},\mathtt{B}_{i_{2}}^{(n+p)}\right)\leq\left(1-\frac{1}{N(n+p-1)+1}\right)\cdot M\leq M.
\]
\end{proof}

\begin{lem}[Distance between $(n+p)$-blocks in the Circular System]
\label{lem:Fdist} For every $i_{1},i_{2}\in\left\{ 1,\dots,K\right\} $,
$i_{1}\neq i_{2}$, and any sequences $\mathcal{B}$ and $\overline{\mathcal{B}}$
of at least $q_{n+p}/R_{n+p}$ consecutive symbols in $\mathcal{B}^{(n+p)}_{i_{1}}$
and $\mathcal{B}^{(n+p)}_{i_{2}}$we have 
\begin{equation}\label{eq:np}
\overline{f}\left(\mathcal{B},\overline{\mathcal{B}}\right)\geq\left(1-\frac{\varepsilon}{2}\right)\cdot\alpha-E_{n+p}.
\end{equation}
\end{lem}

\begin{proof}
By the same proof as in Lemma \ref{lem:DistN} (with $J=0$), as well as case (1) of Lemma \ref{lem:blockNp},
and the observation that all pre-$(n+p)$-blocks used in the construction of the $(n+p)$-blocks are distinct,
(\ref{eq:np}) follows from equations (\ref{eq:F2})-(\ref{eq:F4}).
\end{proof}
\begin{proof}[Proof of Proposition \ref{prop:sep}]
 By Lemma \ref{lem:Fclos} and equations (\ref{eq:Delta}), (\ref{eq:d}), and (\ref{eq:Neps}),
we have 
\begin{align*}
\overline{f}\left(\mathtt{B}_{i}^{(n+p)},\mathtt{B}_{j}^{(n+p)}\right) & \leq\sum_{u=1}^{p-1}\left(\frac{1}{N(n+u-1)+1}+\frac{1}{d_{n+u}}\right)\\
 & <\frac{1}{N+1}+\sum_{u=1}^{p-1}\left(\frac{1}{Kd_{n+u}e_{n+u}+1}+\frac{1}{d_{n+u}}\right)<\delta,
\end{align*}
which is the first statement of Proposition \ref{prop:sep}. In order
to prove the second statement we note that 
\begin{align*}
E_{n+p} & =E_{n+1}+\sum_{i=1}^{p-1}\left(\frac{6}{R_{n+i}}+\frac{8}{u_{n+i}}+\frac{17}{\sqrt{e_{n+i}}}\right) \\
& \leq \frac{14}{\sqrt{N}}+\sum_{i=1}^{p-1}\left(\frac{8}{u_{n+i}}+\frac{17}{\sqrt{e_{n+i}}}\right)+\sum^{p-1}_{i=0}\frac{6}{R_{n+i}}\\
 & <\frac{\varepsilon}{2}+\sum^{p-1}_{i=0}\frac{6}{R_{n+i}}
\end{align*}
by equations (\ref{eq:eps}), (\ref{eq:Neps}), and our assumption on the circular coefficients $\left(l_{n}\right)_{n\in\mathbb{N}}$ in (\ref{eq:lcond}). Hereby, we conclude for any sequences $\mathcal{B}$ and $\overline{\mathcal{B}}$
of at least $q_{n+p}/R_{n+p}$ consecutive symbols in $\mathcal{B}_{i}^{(n+p)}$
and $\mathcal{B}_{j}^{(n+p)}$, respectively, that
\[
\overline{f}\left(\mathcal{B},\overline{\mathcal{B}}\right)\geq\left(1-\frac{\varepsilon}{2}\right)\cdot\alpha-E_{n+p}\geq\alpha-\frac{\varepsilon}{2}-E_{n+p}>\alpha-\varepsilon  -\sum^{p -1}_{i=0}\frac{6}{R_{n+i}}
\]
 with the aid of Lemma \ref{lem:Fdist}. 
\end{proof}

\section{\label{sec:Proof-of-Theorem}Proof of Theorem \ref{thm:zero-entropy}}

We define the construction sequence for the odometer-based system
inductively. We begin by choosing an integer $R_1 \geq 400$, an increasing sequence $\left(K_{s}\right)_{s\in\mathbb{N}}$
of positive integers such that 
\begin{equation}
\sum_{s\in\mathbb{N}}\frac{14}{\sqrt{K_{s}}}<\frac{1}{32},\label{eq:K}
\end{equation}
and two decreasing sequences $\left(\varepsilon_{s}\right)_{s\in\mathbb{N}}$
and $\left(\delta_{s}\right)_{s\in\mathbb{N}}$ of positive real numbers
such that $\delta_{s}\searrow0$ and 
\begin{equation}
\sum_{s\in\mathbb{N}}\varepsilon_{s}<\frac{1}{64}.\label{eq:epsFinal}
\end{equation}
In addition to (\ref{eq:lcond}) we also impose the condition 
\[
\frac{6}{R_1} + \sum_{n\in \mathbb{N}} \frac{6}{k_n} < \frac{1}{32}
\]
on the circular coefficients $\left(k_n,l_n\right)_{n\in \mathbb{N}}$. In particular, this yields
\begin{equation} \label{eq:Rcond}
    \sum^{\infty}_{n =1} \frac{6}{R_n} < \frac{1}{32}
\end{equation}
by $R_n = k_{n-2}q^2_{n-2}$ for $n\geq 2$. We start with $R_1+1$ symbols and let $\alpha_{0}=\frac{1}{8}$.

The first application of Proposition \ref{prop:sep} will be on $1$-blocks
and we will apply it for $\varepsilon=\varepsilon_{1}$, $\delta=\delta_{1}$,
and $K=K_{1}+1$. In order to apply the Proposition we need sufficiently many $1$-blocks. Moreover, we want the $1$-blocks in the circular system to be at
least $\alpha_{0}-\varepsilon_{1}$ apart in the $\overline{f}$ metric
on substantial substrings of at least $\frac{q_{1}}{R_{1}}$ consecutive
symbols. To produce such a family of $1$-blocks we apply Proposition
\ref{prop:FeldmanMechanism}.

After the application of Proposition \ref{prop:sep} on the $1$-blocks
we have $K_{1}+1$ many $n_{1}$-blocks (where $n_{1}=1+p_1$ with $p_1$ from Proposition \ref{prop:sep}, i.e. the least integer such that $\left(1-\frac{1}{K_1+1}\right)^{p_1}<\frac{\varepsilon_1}{2}$) which
are $\delta_{1}$-close in the odometer-based system and at least
$\alpha_{1}\coloneqq \alpha_{0}-2\varepsilon_{1}-\sum^{n_1-1}_{s=1}\frac{6}{R_{s}}$ apart
in the $\overline{f}$ metric on substantial substrings of at least
$q_{n_{1}}/R_{n_{1}}$ consecutive symbols. The next application
of Proposition \ref{prop:sep} will be on $(n_1+1)$-blocks and
we will apply it for $\varepsilon=\varepsilon_{2}$, $\delta=\delta_{2}$,
and $K=K_{2}+1$. Once again this will require sufficiently many $(n_1+1)$-blocks, and we apply Proposition \ref{prop:FeldmanMechanism}
to produce such a family of $(n_1+1)$-blocks that are at
least $\alpha_{1}-\frac{14}{\sqrt{K_{1}}}-\frac{4}{R_{n_1}}-\varepsilon_{2}$ apart on
substantial substrings. This imposes the condition $\frac{2}{l_{n_1-1}}+\frac{4R_{n_1}}{l_{n_1}}<\varepsilon_2$ which by condition (\ref{eq:lcond}) is fulfilled if $\frac{3}{l_{n_1-1}}<\varepsilon_2$ and we choose $l_{n_1-1}$ sufficiently large to satisfy this requirement.

Continuing like this we use Proposition \ref{prop:FeldmanMechanism} and \ref{prop:sep} alternately to produce
$K_{s}+1$ many $n_{s}$-blocks which are at least
\begin{equation}
\alpha_{s}=\alpha_{0}-\sum_{i=1}^{s-1}\frac{14}{\sqrt{K_{i}}}-\sum_{i=1}^{s}2\varepsilon_{i} - \sum^{n_s-1}_{i=1}\frac{6}{R_i}
\end{equation}
apart on substantial substrings of length at least $\frac{h_{n_s}}{K_s+1}$ in the circular system and $\delta_{s}$-close
in the odometer-based system. In the next step, we want to apply Proposition
\ref{prop:sep} on $(n_{s}+1)$-blocks with
$\varepsilon=\varepsilon_{s+1}$, $\delta=\delta_{s+1}$, and $K=K_{s+1}+1$.
In order to have sufficiently many $(n_{s}+1)$-blocks for this application
we make use of Proposition \ref{prop:FeldmanMechanism} (imposing the condition on $l_{n_s-1}$ as described above) and
produce the required number of $(n_{s}+1)$-blocks which are at least
$\alpha_{s}-\frac{14}{\sqrt{K_{s}}}-\frac{4}{R_{n_s}}-\varepsilon_{s+1}$
apart on substantial substrings in the circular system. After the
intended application of Proposition \ref{prop:sep} we have $K_{s+1}+1$
many $n_{s+1}$-blocks (where $n_{s+1}=n_s +1 +p_{s+1}$ with $p_{s+1}$ the least integer such that $\left(1-\frac{1}{K_{s+1}+1}\right)^{p_{s+1}}<\frac{\varepsilon_{s+1}}{2}$) which are at least
\begin{equation}
\alpha_{s+1}=\alpha_{s}-\frac{14}{\sqrt{K_{s}}}-2\varepsilon_{s+1}- \sum^{n_{s+1}-1}_{i=n_s}\frac{6}{R_i}
\end{equation}
apart on substantial substrings in the circular system and $\delta_{s+1}$-close
in the odometer-based system. This completes the inductive step.

By the requirements (\ref{eq:K}), (\ref{eq:epsFinal}), and (\ref{eq:Rcond}) this shows
that any two distinct $n_s$-blocks in the circular system are at least 
\begin{equation}
\alpha_{0}-\sum^{\infty}_{s=1}\left(\frac{14}{\sqrt{K_{s}}}+2\varepsilon_{s}+ \frac{6}{R_s}\right)>\frac{1}{8}-\frac{1}{32}-\frac{1}{32} - \frac{1}{32}=\frac{1}{32}
\end{equation}
apart on substantial substrings. Hence, the
circular system cannot be loosely Bernoulli by Lemma \ref{lem:Rothstein}.

On the other hand, the $\overline{f}$ distance between $n_s$-blocks in
the odometer-based system goes to zero because $\delta_{s}\searrow0$.
Thus, the odometer-based system is loosely Bernoulli by Lemma \ref{lem:Rothstein}. Since the blocks
constructed by Propositions \ref{prop:FeldmanMechanism} and \ref{prop:sep}
satisfy the properties of uniformity and unique readability, our construction
sequence satisfies those as well.

\part{Non-loosely Bernoulli Odometer-Based System whose corresponding Circular
System is Loosely Bernoulli}

In the first two parts of this paper we showed that $\mathcal{F}$ does not preserve the loosely Bernoulli property. In the following four sections we will show that $\mathcal{F}^{-1}$ also does not preserve the loosely Bernoulli property.
\begin{thm}
\label{thm:converse} There exist circular coefficients $\left(l_{n}\right)_{n\in\mathbb{N}}$
and a non-loosely Bernoulli odometer-based system $\mathbb{M}$ of
zero measure-theoretic entropy with uniform and uniquely readable
construction sequence such that $\mathcal{F}\left(\mathbb{M}\right)$
is loosely Bernoulli.
\end{thm}

\section{Feldman Patterns Revisited} \label{sec:OdomFeld}
\begin{prop}
\label{prop:OdomFeld}Let $\alpha\in(0,\frac{1}{7})$ and $n\in \mathbb{N}$, $K,R,S,N,M\in\mathbb{N}\setminus \{0\}$
with $N\geq 20$, and $M\geq2$. For $1\leq s\leq S$, let
$\mathtt{A}_{1}^{(s)},\dots,\mathtt{A}_{N}^{(s)}$ be a family of strings, where each $\mathtt{A}_{j}^{(s)}$ is
a concatenation of $K$ many $n$-blocks. Assume that for all $1\leq s_{1},s_{2}\leq S$
and all $j_{1},j_{2}\in\left\{ 1,\dots,N\right\} $ with $j_{1}\neq j_{2}$,
we have $\overline{f}\left(\mathtt{A},\overline{\mathtt{A}}\right)>\alpha$
for all strings $\mathtt{A},\overline{\mathtt{A}}$ of at least $K\mathtt{h}_n/R$
consecutive symbols from $\mathtt{A}_{j_{1}}^{(s_{1})}$ and
$\mathtt{A}_{j_{2}}^{(s_{2})},$ respectively. 

Then for $1\leq s\leq S$, we can construct a family of strings $\mathtt{B}_{1}^{(s)},\dots,\mathtt{B}_{M}^{(s)}$ (of equal length $N^{2M+3}\cdot K\cdot\mathtt{h}_{n}$
and containing each block $\mathtt{A}_{1}^{(s)},\dots,\mathtt{A}_{N}^{(s)}$
exactly $N^{2M+2}$ many times) such that for all $1\leq s_{1},s_{2}\leq S$,
all $j,k\in\left\{ 1,\dots,M\right\} $ with $j\neq k,$ and all strings
$\mathtt{B}$, $\overline{\mathtt{B}}$ of at least $N^{2M+2}\cdot K\cdot \mathtt{h}_{n}$
consecutive symbols from $\mathtt{B}_{j}^{(s_{1})}$ and $\mathtt{B}_{k}^{(s_{2})},$ respectively,
we have 
\[
\overline{f}(\mathtt{B},\overline{\mathtt{B}})>\alpha-\frac{13}{\sqrt{N}}-\frac{2}{R}.
\]
\end{prop}

\begin{proof}
The proof follows along the lines of the statement on Feldman patterns for the circular system in Proposition \ref{prop:Feldman}. (Here $\mathcal{C}_{n,i_1}$ and $\mathcal{C}_{n,i_2}$ are not applied since we remain in the odometer-based system.) 
\end{proof}
We will also need a statement on the $\overline{f}$ distance between the identical
Feldman pattern but with building blocks from different families (compare with Lemma \ref{lem:samePattern}).
\begin{lem} \label{lem:same}
Let $\alpha\in(0,\frac{1}{7})$ and $n\in\mathbb{N}$, $K,R,S,N,M\in\mathbb{N}\setminus \{0\}$
with $N\geq 20$, and $M,S$ at least $2$. For $1\leq s\leq S$, let $\mathtt{A}_{1}^{(s)},\dots,\mathtt{A}_{N}^{(s)}$ be a family of strings, where each $\mathtt{A}_{j}^{(s)}$ is
a concatenation of $K$ many $n$-blocks. Assume that for all $j_{1},j_{2}\in\left\{ 1,\dots,N\right\} $
and all $s_{1},s_{2}\in\left\{ 1,\dots,S\right\} $ with $s_{1}\neq s_{2}$,
we have $\overline{f}\left(\mathtt{A},\overline{\mathtt{A}}\right)>\alpha$
for all sequences $\mathtt{A},\overline{\mathtt{A}}$ of at least
$K\mathtt{h}_{n}/R$ consecutive symbols from $\mathtt{A}_{j_{1}}^{(s_{1})}$
and $\mathtt{A}_{j_{2}}^{(s_{2})},$ respectively. Then for $1\leq s\leq S,$ we can construct
a family of strings $\mathtt{B}_{1}^{(s)},\dots,\mathtt{B}_{M}^{(s)}$ as in Proposition \ref{prop:OdomFeld} such that
for all $j,k\in\left\{ 1,\dots,M\right\} $, all $s_{1},s_{2}\in\left\{ 1,\dots,S\right\} $ with
$s_{1}\neq s_{2}$, and all strings $\mathtt{B}$, $\overline{\mathtt{B}}$
of at least $N^{2M+2}K\mathtt{h}_{n}$ consecutive symbols from
$\mathtt{B}_{j}^{(s_{1})}$ and $\mathtt{B}_{k}^{(s_{2})}$, respectively, we have 
\[
\overline{f}(\mathtt{B},\overline{\mathtt{B}})>\alpha-\frac{2}{N^{2M+2}}-\frac{2}{R}.
\]
\end{lem}

\section{Feldman Mechanism in the Odometer-based System} \label{sec:OdomFeldMechanism}

Again we use the \emph{Feldman mechanism} to produce an arbitrarily large number of new
blocks whose substantial substrings are almost as far apart in $\overline{f}$
as the building blocks. 
In contrast
to Proposition \ref{prop:OdomFeld}, the constructed words also satisfy
the unique readability property.
\begin{prop}
\label{prop:OdomFeldMech}Let $\alpha\in\left(0,\frac{1}{7}\right)$
and $n ,N,M,R\in\mathbb{N}$ with $R>0$, $N\geq 100$, and $M\geq2$.
Suppose there are $N+1$ many $n$-blocks $\mathtt{A}_{0},\dots,\mathtt{A}_{N}$, which have equal length $\mathtt{h_{n}}$ and satisfy the unique readability
property. Furthermore, if $n>0$ assume that for all $j\neq k,$ we
have $\overline{f}\left(\mathtt{A},\overline{\mathtt{A}}\right)>\alpha$
for all strings $\mathtt{A},\overline{\mathtt{A}}$ of at least $\mathtt{h}_{n}/R$
consecutive symbols from $\mathtt{A}_{j}$ and $\mathtt{A}_{k}$,
respectively. Then we can construct $M$ many $(n+1)$-blocks $\mathtt{\mathtt{B}_{1}},\dots,\mathtt{B}_{M}$
of equal length $\mathtt{h}_{n+1}$ (which are uniform in the $n$-blocks
and satisfy the unique readability property) such that for all $j \ne k$ and all strings
$\mathtt{B}$, $\overline{\mathtt{B}}$ of at least $\mathtt{h}_{n+1}/N$
consecutive symbols from $\mathtt{B}_{j}$ and $\mathtt{B}_{k}$, respectively,
we have 
\[
\overline{f}(\mathtt{B},\overline{\mathtt{B}})>\alpha-\frac{12}{\sqrt{ N}}-\frac{2}{R}-\frac{8}{N}\geq \alpha-\frac{13}{\sqrt{ N}}-\frac{2}{R}.
\]
\end{prop}

\begin{proof}
We define the $(n+1)$-blocks as in Proposition \ref{prop:FeldmanMechanism}. As in its proof in the case $n=0$ we complete cycles at the beginning and end of each $\mathtt{B}$ and $\overline{\mathtt{B}}$. Once again, we remove the marker blocks and apply Corollary \ref{cor:alpha separated blocks} and $\tilde{f}>1-\frac{12}{\sqrt{N}}$ from Lemma \ref{lem:symbolic Feldman}.
\end{proof}

\section{Cycling Mechanism} \label{sec:cycling}

\begin{figure}
    \centering
    \includegraphics[scale=0.75]{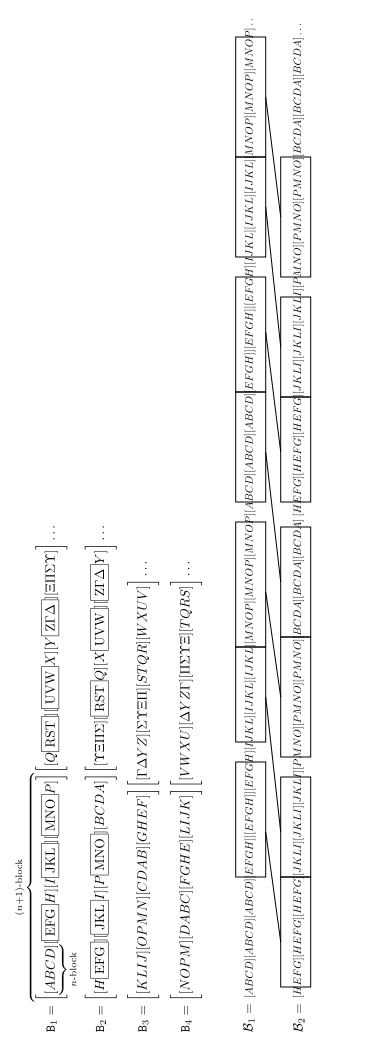}
    \caption{{\bf Heuristic representation of two stages of the cycling mechanism.} Parts of four $(n+2)$-blocks $\mathtt{B}_i$, $i=1,\dots,4$, in the odometer based system and parts of the images $\mathcal{B}_1$ and $\mathcal{B}_2$ under the circular operator (omitting the spacers) are represented. The marked letters indicate a best possible $\overline{f}$ match between $\mathtt{B}_1$ and $\mathtt{B}_2$ with a fit of approximately $\left(1-\frac{1}{4}\right)^2$ (ignoring boundary effects) while the blocks $\mathcal{B}_1$ and $\mathcal{B}_2$ have a very good fit in the circular system (the lines indicating a best possible $\overline{f}$ match).}
    \label{fig:fig3}
\end{figure}

\begin{prop}
\label{prop:cycling} If $n ,K,T\in \mathbb{N}$, $K\geq 2$, $T>0$, $\alpha\in\left(4/T,1/7\right)$, $\delta>0$,
and $\varepsilon \in \left(0,\alpha/(4K)\right)$, then there exist 
$N,m_n\in\mathbb{N}$ and circular coefficients $\left(k_{n+m},\ell_{n+m}\right)_{m=0}^{m_n-1}$ satisfying the following condition. 
If we are given $N+1$ uniquely readable $n$-blocks $\mathtt{B}_{0}^{(n)},\mathtt{B}_{1}^{(n)},\dots,\mathtt{B}_{N}^{(n)}$
in the odometer-based system such that for all $i\ne j$ and all sequences $\mathtt{A}$ and $\overline{\mathtt{A}}$ of at least $\mathtt{h}_n/T$
consecutive symbols from $\mathtt{B}_{i}^{(n)}$ and $\mathtt{B}_{j}^{(n)}$, respectively, we have
$\overline{f}\left(\mathtt{A},\overline{\mathtt{A}}\right)\geq\alpha,$
then we can build $K$ many $(n+m_n)$-blocks $\mathtt{B}^{(n+m_n)}_{1},\dots,\mathtt{B}^{(n+m_n)}_{K}$
of equal length $\mathtt{h}_{n+m_n}$ (with the unique readability
property and uniformity in all blocks from stage $n$ through $n+m_n$) satisfying
the following properties: 
\begin{enumerate}
\item  For all $i\ne j$ and all sequences $\mathtt{B}$ and $\overline{\mathtt{B}}$ of at least $\mathtt{h}_{n+m_n}/K$
consecutive symbols from $\mathtt{B}_{i}^{(n+m_n)}$ and $\mathtt{B}_{j}^{(n+m_n)}$, respectively, we have 
\[
\overline{f}\left(\mathtt{B},\overline{\mathtt{B}}\right)\geq\alpha-\frac{2}{T}-\varepsilon.
\]
\item If $\mathcal{B}_1^{(n+m_n)},\dots,\mathcal{B}_K^{(n+m_n)}$ are the corresponding circular $(n+m_n)$-blocks, then for all $1\leq i,j\leq K$ we have 
\[
\overline{f}\left(\mathcal{B}_{i}^{(n+m_n)},\mathcal{B}_{j}^{(n+m_n)}\right)\leq\delta.
\]
\end{enumerate}
\end{prop}

\begin{rem*}
Thus we obtain a mechanism to produce an arbitrarily large number of $(n+m_n)$-blocks
that are still apart from each other in the odometer-based system
but arbitrarily close to each other in the circular system. The proof is based on an inductive construction that we call the \emph{cycling mechanism} (indicated in Figure \ref{fig:fig3}) and a final step to guarantee closeness of blocks in the circular system. In each step of the cycling mechanism the blocks of different types are constructed by cycling the used $K$ pre-blocks. On the one hand, this will yield closeness in the circular system under the repetitions in the circular operator. On the other hand, in a matching in the odometer-based system at least one pair of pre-blocks will have a $\overline{f}$ distance close to $\alpha$ and at most $K-1$ many pairs will have a smaller $\overline{f}$ distance. Over the course of the construction this distance will increase towards $\alpha$ (see equation \ref{eq:BetaN+M}).
\end{rem*}

For this construction, let $\left(u_{n+m}\right)_{m\in\mathbb{N}}$
and $\left(e_{n+m}\right)_{m\in\mathbb{N}}$ be increasing sequences
of positive integers such that 
\begin{equation}
\sum_{m\in\mathbb{N}}\frac{2}{Ku_{n+m}^{2}e_{n+m}}<\frac{\delta}{2},\label{eq:Delta-1}
\end{equation}
and

\begin{equation}
\sum_{m\in\mathbb{N}}\left(\frac{4}{u_{n+m}}+\frac{14}{\sqrt{e_{n+m}}}\right)<\frac{\varepsilon}{8}.\label{eq:eps-1}
\end{equation}
Additionally, we define the sequences $(\lambda_{n+m})_{m\in\mathbb{N}}$ and $\left(d_{n+m}\right)_{m\in\mathbb{N}}$ by

\begin{equation}
d_{n+m}=u_{n+m}^{2}\label{eq:d-1}
\end{equation}
and
\[
\lambda_{n+m} = 2d_{n+m}e_{n+m}.
\]
Moreover, we choose $N$ sufficiently large such that 
\begin{equation}
\frac{14}{\sqrt{N}}<\frac{\varepsilon}{8}\label{eq:N}
\end{equation}
and the positive integers $\left(l_{n+m}\right)_{m\in\mathbb{N}}$ sufficiently large
such that 
\begin{equation}
\sum_{m\in\mathbb{N}}\frac{4}{l_{n+m}}<\frac{\delta}{2}.\label{eq:l-1}
\end{equation}
The terms $l_n,l_{n+1},\dots,l_{n+m_n-1}$ are the coefficients in the circular $(n+1)$-, $(n+2)$-, \dots,$(n+m_n)$-blocks, where  $m_n$ is defined below. 
The proof of Proposition \ref{prop:cycling} utilizes the following parameters $\left(\alpha_{n+m}\right)^{m=\infty}_{m=2}$ and $\left(\beta_{n+m}\right)^{m=\infty}_{m=2}$ defined inductively via
\begin{equation}
\alpha_{n+2}=\alpha-\frac{14}{\sqrt{N}}-\frac{2}{T}-\frac{4}{u_{n+1}}-\frac{13}{\sqrt{e_{n+1}}},\label{eq:alphaN+2}
\end{equation}
\begin{equation}
\beta_{n+2}=\frac{1}{K}\alpha -\frac{14}{\sqrt{N}}-\frac{2}{KT}-\frac{4}{u_{n+1}}-\frac{2}{e_{n+1}},\label{eq:BetaN+2}
\end{equation} 
\begin{equation}
\alpha_{n+m+1}=\alpha_{n+2}-\sum_{i=2}^{m}\left(\frac{2}{\lambda_{n+i-1}K+1}+\frac{2}{e_{n+i-1}}+\frac{4}{u_{n+i}}+\frac{13}{\sqrt{e_{n+i}}}\right),\label{eq:AlphaN+M}
\end{equation}
and 
\begin{equation}\label{eq:BetaN+M}
     \beta_{n+m+1}=
     \beta_{n+m}+\frac{1}{K}\left(\alpha_{n+m}-\beta_{n+m}\right)-\frac{2}{e_{n+m-1}}-\frac{4}{u_{n+m}}-\frac{2}{\lambda_{n+m-1}K+1}-\frac{2}{e_{n+m}}.
\end{equation}

Since $\frac{1}{K}\left(\alpha-\frac{2}{T}\right)>\frac{\alpha}{2K}>2\varepsilon$, assumptions (\ref{eq:eps-1}) and (\ref{eq:N}) imply that $\beta_{n+2}>\varepsilon$. We also note that for every $m\geq2$
\begin{equation*}
\alpha_{n+m} \geq\alpha-\frac{14}{\sqrt{N}}-\frac{2}{T}-\sum_{i=1}^{\infty}\left(\frac{14}{\sqrt{e_{n+i}}}+\frac{4}{u_{n+i}}\right) >\alpha-\frac{2}{T}-\frac{\varepsilon}{8}-\frac{\varepsilon}{8},
\end{equation*}
by equation (\ref{eq:AlphaN+M}) and our assumptions (\ref{eq:eps-1})
and (\ref{eq:N}). Similarly, assumption (\ref{eq:eps-1}) implies that in our $\beta_{n+m}$ equation (\ref{eq:BetaN+M}), the terms we subtract from $\beta_2$ over the whole course of the construction are bounded by $\varepsilon/8$.
Due to these bounds, we can choose $m_n$ as the least integer such that
\begin{equation} \label{eq:betacond}
\beta_{n+m_{n}}>\alpha-\frac{2}{T}-\frac{\varepsilon}{2},
\end{equation}
and we will apply our inductive construction until stage $n+m_n$.

\subsection{Initial step: Construction of $(n+1)$-blocks} \label{subsec:3Initial}

First of all, we
choose one $n$-block $\mathtt{B}_{0}^{(n)}$ as a marker. Then we
apply Proposition \ref{prop:OdomFeld} on the remaining $n$-blocks
$\mathtt{B}_{1}^{(n)},\dots,\mathtt{B}_{N}^{(n)}$ to build $\tilde{N}(n+1)\coloneqq K\cdot\left(\lambda_{n+1}+1\right)$
many Feldman patterns denoted by $\mathtt{A}_{i,j}$, $i=1,\dots,K$,
$j=1,\dots,\lambda_{n+1}+1$. We will call them \emph{pre-$(n+1)$-blocks.} In particular, these
have length $\tilde{\mathtt{h}}_{n+1}=N^{2\cdot\tilde{N}(n+1)+3}\cdot\mathtt{h}_{n}$
and are uniform in the $n$-blocks $\mathtt{B}_{1}^{(n)},\dots,\mathtt{B}_{N}^{(n)}$
by construction. More precisely, every pre-$(n+1)$-block contains
each $n$-block $\mathtt{B}_{i}^{(n)}$, $1\leq i\leq N$, exactly
\begin{equation}
\bar{N}(n)\coloneqq N^{2\cdot\tilde{N}(n+1)+2}\label{eq:3n}
\end{equation}
 many times and pre-$(n+1)$-blocks in the circular system have length
$\tilde{q}_{n+1}=N^{2\cdot\tilde{N}(n+1)+3}\cdot l_{n}\cdot q_{n}$.
In order to obtain uniformity, we will define the marker segment by 
\[
\mathtt{a}_{n}=\left(\mathtt{B}_{0}^{(n)}\right)^{K\bar{N}(n)}.
\]
Moreover, we have 
\begin{equation}
\overline{f}\left(\mathtt{A},\overline{\mathtt{A}}\right)\geq\alpha-\frac{13}{\sqrt{N}}-\frac{2}{T}\label{eq:pren+1}
\end{equation}
for any strings $\mathtt{A}$ and $\overline{\mathtt{A}}$ of at least
$\tilde{\mathtt{h}}_{n+1}/N=N^{2\cdot\tilde{N}(n+1)+2}\cdot\mathtt{h}_{n}$
consecutive symbols in different pre-$(n+1)$-blocks by Proposition
\ref{prop:OdomFeld}. 

Finally, we define the $(n+1)$-blocks: 
\begin{align*}
(n+1)\text{-blocks of type \ensuremath{1}: \;} & \mathtt{B}_{1,1}^{(n+1)}=\mathtt{A}_{1,1}\mathtt{A}_{2,1}\dots\mathtt{A}_{K-1,1}\mathtt{A}_{K,1}\mathtt{a_{n}},\\
 & \mathtt{B}_{1,2}^{(n+1)}=\mathtt{A}_{1,2}\mathtt{A}_{2,2}\dots\mathtt{A}_{K-1,2}\mathtt{A}_{K,2}\mathtt{a_{n}},\\
 & \mathtt{B}_{1,3}^{(n+1)}=\mathtt{A}_{1,3}\mathtt{A}_{2,3}\dots\mathtt{A}_{K-1,3}\mathtt{A}_{K,3}\mathtt{a_{n}},\\
 & \mathtt{B}_{1,4}^{(n+1)}=\mathtt{A}_{1,4}\mathtt{A}_{2,4}\dots\mathtt{A}_{K-1,4}\mathtt{A}_{K,4}\mathtt{a_{n}},\\
 & \;\vdots\\
 & \mathtt{B}_{1,\lambda_{n+1}}^{(n+1)}=\mathtt{A}_{1,\lambda_{n+1}}\mathtt{A}_{2,\lambda_{n+1}}\dots\mathtt{A}_{K-1,\lambda_{n+1}}\mathtt{A}_{K,\lambda_{n+1}}\mathtt{a_{n}},\\
(n+1)\text{-blocks of type \ensuremath{2}: \;} & \mathtt{B}_{2,1}^{(n+1)}=\mathtt{A}_{2,1}\mathtt{A}_{3,1}\dots\mathtt{A}_{K,1}\mathtt{A}_{1,1}\mathtt{a_{n}},\\
 & \mathtt{B}_{2,2}^{(n+1)}=\mathtt{A}_{K,2}\mathtt{A}_{1,2}\dots\mathtt{A}_{K-1,2}\mathtt{a_{n}},\\
 & \mathtt{B}_{2,3}^{(n+1)}=\mathtt{A}_{2,3}\mathtt{A}_{3,3}\dots\mathtt{A}_{K,3}\mathtt{A}_{1,3}\mathtt{a_{n}},\\
 & \mathtt{B}_{2,4}^{(n+1)}=\mathtt{A}_{K,4}\mathtt{A}_{1,4}\dots\mathtt{A}_{K-1,4}\mathtt{a_{n}},\\
 & \;\vdots\\
 & \mathtt{B}_{2,\lambda_{n+1}}^{(n+1)}=\mathtt{A}_{K,\lambda_{n+1}}\mathtt{A}_{1,\lambda_{n+1}}\dots\mathtt{A}_{K-1,\lambda_{n+1}}\mathtt{a_{n}},\\
\vdots & \;\vdots\\
(n+1)\text{-blocks of type \ensuremath{K}: \;} & \mathtt{B}_{K,1}^{(n+1)}=\mathtt{A}_{K,1}\mathtt{A}_{1,1}\dots\mathtt{A}_{K-1,1}\mathtt{a_{n}},\\
 & \mathtt{B}_{K,2}^{(n+1)}=\mathtt{A}_{2,2}\mathtt{A}_{3,2}\dots\mathtt{A}_{K,2}\mathtt{A}_{1,2}\mathtt{a_{n}},\\
 & \mathtt{B}_{K,3}^{(n+1)}=\mathtt{A}_{K,3}\mathtt{A}_{1,3}\dots\mathtt{A}_{K-1,3}\mathtt{a_{n}},\\
 & \mathtt{B}_{K,4}^{(n+1)}=\mathtt{A}_{2,4}\mathtt{A}_{3,4}\dots\mathtt{A}_{K,4}\mathtt{A}_{1,4}\mathtt{a_{n}},\\
 & \;\vdots \\
 & \mathtt{B}_{K,\lambda_{n+1}}^{(n+1)}=\mathtt{A}_{2,\lambda_{n+1}}\mathtt{A}_{3,\lambda_{n+1}}\dots\mathtt{A}_{K,\lambda_{n+1}}\mathtt{A}_{1,\lambda_{n+1}}\mathtt{a_{n}}.
\end{align*}
Here, the index $i\in\left\{ 1,\dots,K\right\} $ in $\mathtt{B}_{i,j}^{(n+1)}$
indicates the type and $j=1,\dots,\lambda_{n+1}$ numbers the $(n+1)$-blocks
of that type consecutively. We note that for $j$ odd the block $\mathtt{B}_{i+1,j}^{(n+1)}$
is obtained from $\mathtt{B}_{i,j}^{(n+1)}$ by cycling the pre-$(n+1)$-blocks
to the left. On the other hand, for $j$ even the block $\mathtt{B}_{i+1,j}^{(n+1)}$
is obtained from $\mathtt{B}_{i,j}^{(n+1)}$ by cycling the pre-$(n+1)$-blocks
to the right. Additionally, we define the next marker block 
\[
\mathtt{B}_{0}^{(n+1)}=\mathtt{A}_{1,\lambda_{n+1}+1}^{(n+1)}\mathtt{A}_{2,\lambda_{n+1}+1}^{(n+1)}\dots\mathtt{A}_{K,\lambda_{n+1}+1}^{(n+1)}\mathtt{a}_{n},
\]
where $\mathtt{A}_{i,\lambda_{n+1}+1}^{(n+1)}$, $i=1,\dots,K$, have
not been used in any of the other $(n+1)$-blocks. Hence, there are
are $N(n+1)+1=\lambda_{n+1}K+1$ many $(n+1)$-blocks in total. We also
note that every $(n+1)$-block is uniform in the $n$-blocks by equation
(\ref{eq:3n}). 

\subsection{Inductive step: Construction of $(n+m)$-blocks} \label{subsec:Step}

In an inductive process we construct $(n+m)$-blocks for $m\geq 2$. Assume that in our inductive construction we have constructed
$K\lambda_{n+m-1}$ many $(n+m-1)$-blocks
$\mathtt{B}_{i,j}^{(n+m-1)}$ of $K$ different types, where for $m=2$ the $(n+1)$-blocks are the ones constructed in Subsection \ref{subsec:3Initial} and for $m\geq 3$ the $(n+m-1)$-blocks are constructed according to the following formula (with $\lambda=\lambda_{n+m-1}$)
\begin{align*}
& \ \ \ \ \ \ \ (n+m-1)\text{-blocks of type \ensuremath{1}: \;} \\ 
& \mathtt{B}_{1,1}^{(n+m-1)}=\mathtt{A}_{1,1}^{(n+m-1)}\mathtt{A}_{2,2}^{(n+m-1)}\dots\mathtt{A}_{K,K}^{(n+m-1)}\mathtt{a}_{n+m-2},\\
 & \mathtt{B}_{1,2}^{(n+m-1)}=\mathtt{A}_{1,K+1}^{(n+m-1)}\mathtt{A}_{2,K+2}^{(n+m-1)}\dots\mathtt{A}_{K,2K}^{(n+m-1)}\mathtt{a}_{n+m-2},\\
 & \mathtt{B}_{1,3}^{(n+m-1)}=\mathtt{A}_{1,2K+1}^{(n+m-1)}\mathtt{A}_{2,2K+2}^{(n+m-1)}\dots\mathtt{A}_{K,3K}^{(n+m-1)}\mathtt{a}_{n+m-2},\\
 & \mathtt{B}_{1,4}^{(n+m-1)}=\mathtt{A}_{1,3K+1}^{(n+m-1)}\mathtt{A}_{2,3K+2}^{(n+m-1)}\dots\mathtt{A}_{K,4K}^{(n+m-1)}\mathtt{a}_{n+m-2},\\
 & \;\vdots\\
 & \mathtt{B}_{1,\lambda}^{(n+m-1)}=\mathtt{A}_{1,(\lambda-1)K+1}^{(n+m-1)}\mathtt{A}_{2,(\lambda-1)K+2}^{(n+m-1)}\dots\mathtt{A}_{K,\lambda K}^{(n+m-1)}\mathtt{a}_{n+m-2},\\
& \ \ \ \ \ \ \ (n+m-1)\text{-blocks of type \ensuremath{2}: \;} \\
& \mathtt{B}_{2,1}^{(n+m-1)}=\mathtt{A}_{1,2}^{(n+m-1)}\mathtt{A}_{2,3}^{(n+m-1)}\dots\mathtt{A}_{K-1,K}^{(n+m-1)}\mathtt{A}_{K,1}^{(n+m-1)}\mathtt{a}_{n+m-2},\\
 & \mathtt{B}_{2,2}^{(n+m-1)}=\mathtt{A}_{1,2K}^{(n+m-1)}\mathtt{A}_{2,K+1}^{(n+m-1)}\dots\mathtt{A}_{K,K-1}^{(n+m-1)}\mathtt{a}_{n+m-2},\\
 & \mathtt{B}_{2,3}^{(n+m-1)}=\mathtt{A}_{1,2K+2}^{(n+m-1)}\mathtt{A}_{2,2K+3}^{(n+m-1)}\dots\mathtt{A}_{K-1,3K}^{(n+m-1)}\mathtt{A}_{K,2K+1}^{(n+m-1)}\mathtt{a}_{n+m-2},\\
 & \mathtt{B}_{2,4}^{(n+m-1)}=\mathtt{A}_{1,4K}^{(n+m-1)}\mathtt{A}_{2,3K+1}^{(n+m-1)}\dots\mathtt{A}_{K,4K-1}^{(n+m-1)}\mathtt{a}_{n+m-2},\\
 & \;\vdots\\
 & \mathtt{B}_{2,\lambda}^{(n+m-1)}=\mathtt{A}_{1,\lambda K}^{(n+m-1)}\mathtt{A}_{2,(\lambda-1)K+1}^{(n+m-1)}\dots\mathtt{A}_{K,\lambda K-1}^{(n+m-1)}\mathtt{a}_{n+m-2}, \\
 & \;\vdots\\
& \ \ \ \ \ \ \ (n+m-1)\text{-blocks of type \ensuremath{K}: \;} \\
& \mathtt{B}_{K,1}^{(n+m-1)}=\mathtt{A}_{1,K}^{(n+m-1)}\mathtt{A}_{2,1}^{(n+m-1)}\dots\mathtt{A}_{K,K-1}^{(n+m-1)}\mathtt{a}_{n+m-2},\\
 & \mathtt{B}_{K,2}^{(n+m-1)}=\mathtt{A}_{1,K+2}^{(n+m-1)}\mathtt{A}_{2,K+3}^{(n+m-1)}\dots\mathtt{A}_{K-1,2K}^{(n+m-1)}\mathtt{A}_{K,K+1}^{(n+m-1)}\mathtt{a}_{n+m-2},\\
 & \mathtt{B}_{K,3}^{(n+m-1)}=\mathtt{A}_{1,3K}^{(n+m-1)}\mathtt{A}_{2,2K+1}^{(n+m-1)}\dots\mathtt{A}_{K,3K-1}^{(n+m-1)}\mathtt{a}_{n+m-2},\\
 & \mathtt{B}_{K,4}^{(n+m-1)}=\mathtt{A}_{1,3K+2}^{(n+m-1)}\mathtt{A}_{2,3K+3}^{(n+m-1)}\dots\mathtt{A}_{K-1,4K}^{(n+m-1)}\mathtt{A}_{K,3K+1}^{(n+m-1)}\mathtt{a}_{n+m-2},\\
 & \;\vdots \\
 & \mathtt{B}_{K,\lambda}^{(n+m-1)}=\mathtt{A}_{1,(\lambda-1)K+2}^{(n+m-1)}\dots\mathtt{A}_{K-1,\lambda K}^{(n+m-1)}\mathtt{A}_{K,(\lambda -1)K+1}^{(n+m-1)}\mathtt{a}_{n+m-2},
\end{align*}
using pre-$(n+m-1)$-blocks $\mathtt{A}^{(n+m-1)}_{i,j}$ of length $\tilde{h}_{n+m-1}$ and a marker segment $\mathtt{a}_{n+m-2}=\left(\mathtt{B}_{0}^{(n+m-2)}\right)^{\bar{N}(n+m-2)}$ with $\bar{N}(n+m-2)$ chosen according to equation (\ref{eq:3N(n+m)}) such that the $(n+m-1)$-blocks are uniform in $(n+m-2)$-blocks. We note that for $j$ odd the block $\mathtt{B}_{i+1,j}^{(n+m-1)}$
is obtained from $\mathtt{B}_{i,j}^{(n+m-1)}$ by cycling the second
index to the left. On the other hand,
for $j$ even the block $\mathtt{B}_{i+1,j}^{(n+m-1)}$ is obtained
from $\mathtt{B}_{i,j}^{(n+m-1)}$ by cycling the second index to the right. Additionally, we have a marker block
\[
\mathtt{B}_{0}^{(n+m-1)}=\mathtt{A}_{1,\lambda_{n+m-1}+1}^{(n+m-1)}\mathtt{A}_{2,\lambda_{n+m-1}+1}^{(n+m-1)}\dots\mathtt{A}_{K,\lambda_{n+m-1}+1}^{(n+m-1)}\mathtt{a}_{n+m-2},
\]
where the pre-$(n+m-1)$-blocks $\mathtt{A}_{i,\lambda_{n+m-1}+1}^{(n+m-1)}$
have not been used in any other $(n+m-1)$-block. Hence, there are
$N(n+m-1)+1=\lambda_{n+m-1}K+1$ many $(n+m-1)$-blocks in total. 

In our inductive construction process for $m\geq 3$, for any strings $\mathtt{A},\overline{\mathtt{A}}$ of at least
$\tilde{h}_{n+m-1}/e_{n+m-2}$ consecutive symbols in $\mathtt{A}_{i_{1},j_{1}}^{(n+m-1)}$
and $\mathtt{A}_{i_{2},j_{2}}^{(n+m-1)},$ we assume
\begin{equation}
\overline{f}\left(\mathtt{A},\overline{\mathtt{A}}\right)\geq\beta_{n+m-1}\;\text{in case of \ensuremath{i_{1}\neq i_{2},\;}\ensuremath{j_{1}=j_{2}}, }\label{eq:n+mdifftype}
\end{equation}
and 
\begin{equation}
\overline{f}\left(\mathtt{A},\overline{\mathtt{A}}\right)\geq\alpha_{n+m-1}\;\text{ in case of \ensuremath{j_{1}\neq j_{2}} for all $i_1,i_2$,}\label{eq:n+mdiffPatt}
\end{equation}
with the numbers $\alpha_{n+m-1}$ and $\beta_{n+m-1}$ from equations (\ref{eq:alphaN+2}) to (\ref{eq:BetaN+M}).
In the corresponding circular system we have 
\begin{equation}
\overline{f}\left(\mathcal{A}_{i_{1},j}^{(n+m-1)},\mathcal{A}_{i_{2},j}^{(n+m-1)}\right)\leq\sum_{i=1}^{m-2}\left(\frac{4}{l_{n+i}}+\frac{2}{N(n+i-1)+1}\right).\label{eq:N+Mcirc}
\end{equation}
Note that this assumption is void in case of $m=2$. In the odometer-based system we will use equation (\ref{eq:pren+1}) for the first inductive step.

In the inductive step starting with $m \geq 2$, we use $(n+m-1)$-blocks to define grouped $(n+m-1)$-blocks
$\mathtt{G}_{i,j}^{(n+m-1)}$ (where $i=1,\dots,K$ indicates the type
of used $(n+m-1)$-blocks and $j=0,\dots,e_{n+m-1}-1$ enumerates the
grouped $(n+m-1)$-blocks of that type) as follows: 
\[
\mathtt{G}_{i,j}^{(n+m-1)}=\mathtt{B}_{i,j\cdot2d_{n+m-1}+1}^{(n+m-1)}\mathtt{B}_{i,j\cdot2d_{n+m-1}+2}^{(n+m-1)}\dots\mathtt{B}_{i,(j+1)\cdot2d_{n+m-1}}^{(n+m-1)},
\]
i.e. it is a concatenation of $2d_{n+m-1}$ many $(n+m-1)$-blocks of
the same type. In the following Lemmas we see that different grouped
blocks are still apart from each other in the odometer-based system but grouped blocks with coinciding index $j$ can be made arbitrarily
close to each other in the circular system.
\begin{lem}[Distance between grouped $(n+m-1)$-blocks in the odometer-based system]
\label{lem:Distn+m} Let $i_{1},i_{2}\in\left\{ 1,\dots,K\right\} $,
$j_{1},j_{2}\in\left\{ 0,\dots,e_{n+m-1}-1\right\} $ and $\mathtt{G},\overline{\mathtt{G}}$
be strings of at least $u_{n+m-1}\mathtt{h}_{n+m-1}$ consecutive
symbols from grouped $(n+m-1)$-blocks $\mathtt{G}_{i_{1},j_{1}}^{(n+m-1)}$
and $\mathtt{G}_{i_{2},j_{2}}^{(n+m-1)}$ respectively.
\begin{enumerate}
\item For $i_{1}\neq i_{2}$ and $j_{1}=j_{2}$ we have 
\[
\begin{split}
    & \overline{f}\left(\mathtt{G},\overline{\mathtt{G}}\right)\geq \\
& \begin{cases}
\frac{1}{K}\alpha -\frac{14}{\sqrt{N}}-\frac{2}{KT}-\frac{2}{u_{n+1}}, & \text{for } m=2; \\
\beta_{n+m-1}+\frac{1}{K}\left(\alpha_{n+m-1}-\beta_{n+m-1}\right)-\frac{2}{e_{n+m-2}}-\frac{2}{u_{n+m-1}}-\frac{2}{N(n+m-2)+1}, & \text{for } m\geq 3.
\end{cases}
\end{split}
\]
\item For $j_{1}\neq j_{2}$ and all $i_{1},i_{2}$ we have 
\[
\overline{f}\left(\mathtt{G},\overline{\mathtt{G}}\right)\geq
\begin{cases}
\alpha-\frac{14}{\sqrt{N}}-\frac{2}{T}-\frac{2}{u_{n+1}}, &  \text{ for } m=2; \\
\alpha_{n+m-1}-\frac{2}{e_{n+m-2}}-\frac{2}{u_{n+m-1}}-\frac{2}{N(n+m-2)+1}, & \text{ for } m\geq 3.
\end{cases}
\]
\end{enumerate}
\end{lem}

\begin{proof}
In the first case we have $j_{1}=j_{2}$. We treat $\mathtt{G}$ and $\overline{\mathtt{G}}$ as strings of complete $(n+m-1)$-blocks
by adding fewer than $2h_{n+m-1}$ symbols to complete partial blocks at the beginning and end of $\mathtt{G}$ and $\overline{\mathtt{G}}$. These constitute
a fraction of at most $2/u_{n+m-1}$ of the total length.
Additionally, we ignore the marker segments $\mathtt{a}_{n+m-2}$
which form a fraction of $1/(N(n+m-2)+1)$ of the length of each
$(n+m-1)$-block due to uniformity and so of $\mathtt{G}$ as well as $\overline{\mathtt{G}}$. On the remaining strings in case of $m=2$ we apply Corollary \ref{cor:alpha separated blocks} with $\tilde{f}\geq \frac{1}{K}$ and equation (\ref{eq:pren+1}) which yields
\[
\overline{f}\left(\mathtt{G},\overline{\mathtt{G}}\right)\geq\frac{1}{K} \left(\alpha-\frac{13}{\sqrt{N}}-\frac{2}{T}\right)-\frac{2}{N}-\frac{2}{N+1}-\frac{2}{u_{n+1}} \geq\frac{1}{K}\alpha -\frac{14}{\sqrt{N}}-\frac{2}{KT}-\frac{2}{u_{n+1}}.
\]
On the remaining strings $\mathtt{G}_{\text{mod}}$ and $\overline{\mathtt{G}}_{\text{mod}}$ in case of $m\geq 3$ we use Proposition \ref{prop:symbol by block replacement} with $\tilde{f}\geq \frac{1}{K}$ and equations (\ref{eq:n+mdifftype}) and (\ref{eq:n+mdiffPatt}) to obtain
\[
\overline{f}\left(\mathtt{G}_{\text{mod}},\overline{\mathtt{G}}_{\text{mod}}\right)\geq \frac{1}{K} \alpha_{n+m-1} + \left(1-\frac{1}{K}\right)\beta_{n+m-1}-\frac{2}{e_{n+m-2}},
\]
which implies the claim.

In the second case we observe that $\mathtt{G}$ and $\overline{\mathtt{G}}$
do not have any Feldman pattern of pre-$(n+m-1)$-blocks in common due to $j_{1}\neq j_{2}$. As before we complete partial blocks at the beginning and end of $\mathtt{G}$ and $\overline{\mathtt{G}}$ and remove the marker segments $\mathtt{a}_{n+m-2}$. On the remaining strings in case of $m=2$ we apply Corollary \ref{cor:alpha separated blocks} with $\tilde{f}=1$ and equation (\ref{eq:pren+1}), while for $m\geq 3$ we use Corollary \ref{cor:alpha separated blocks} with $\tilde{f}=1$ and equation (\ref{eq:n+mdiffPatt}).
\end{proof}

\begin{lem}[Distance between grouped $(n+m-1)$-blocks in the circular system]
\label{lem:N+mCircular} For all $i_{1},i_{2}\in\left\{ 1,\dots,K\right\} $
and $j\in\left\{ 0,\dots,e_{n+m-1}-1\right\},$ we have 
\[
\overline{f}\left(\mathcal{G}_{i_{1},j}^{(n+m-1)},\mathcal{G}_{i_{2},j}^{(n+m-1)}\right)\leq\sum_{i=1}^{m-1}\left(\frac{4}{l_{n+i}}+\frac{2}{N(n+i-1)+1}\right).
\]
\end{lem}

\begin{proof}
We recall that the marker segment $\mathtt{a}_{n+m-2}$ occupies a fraction
of $\frac{1}{N(n+m-2)+1}$ of the total length of each $(n+m-1)$-block. Moreover, for every $j=1,\dots,\lambda_{n+m-1}$
the $(n+m-1)$-block $\mathtt{B}_{i_{2},j}^{(n+m-1)}$ is obtained from
$\mathtt{B}_{i_{1},j}^{(n+m-1)}$ by a cycling permutation of the Feldman patterns used for the pre-$(n+m-1)$-blocks.
Under the cyclic operator $\mathcal{C}_{n+m-1}$ each $(n+m-1)$-block
is repeated $l_{n+m-1}-1$ many times. Hence, the $\overline{f}$ distance
between $\mathcal{G}_{i_{1},j}^{(n+m-1)}$ and $\mathcal{G}_{i_{2},j}^{(n+m-1)}$
is at most 
\[
M+\frac{4}{l_{n+m-1}}+\frac{2}{N(n+m-2)+1},
\]
where $M$ is the $\overline{f}$ distance of pre-$(n+m-1)$-blocks of the same pattern in the
circular system. In case of $m=2$ this distance $M=0$, while for $m\geq 3$ we obtain the claim with the aid of equation
(\ref{eq:N+Mcirc}).
\end{proof}

For each type $i\in\left\{ 1,\dots,K\right\} $ we use the $e_{n+m-1}$
many grouped $(n+m-1)$-blocks of type $i$ as building blocks for the
Feldman patterns $\mathtt{A}_{i,j}^{(n+m)}$, $j=1,\dots,\left(\lambda_{n+m}+1\right)K$ which are the pre-$(n+m)$-blocks with length $\tilde{h}_{n+m}$. For every $i\in\left\{ 1,\dots,K\right\} $
each pattern $\mathtt{A}_{i,j}^{(n+m)}$, $j=1,\dots,\left(\lambda_{n+m}+1\right)K$,
contains each $(n+m-1)$-block of type $i$ exactly 
\begin{equation} \label{eq:3N(n+m)}
\bar{N}(n+m-1)=\left(e_{n+m-1}\right)^{2\left(\lambda_{n+m}+1\right)K+3}
\end{equation}
many times by the construction in Proposition \ref{prop:OdomFeld}. 
\begin{lem}[Separation and closeness of pre-$(n+m)$-blocks of the same Feldman
pattern]
 Let $j\in\left\{ 1,\dots,\left(\lambda_{n+m}+1\right)K\right\} $
and $i_{1},i_{2}\in\left\{ 1,\dots,K\right\} $, $i_{1}\neq i_{2}$.
For all strings $\mathtt{A}$ and $\overline{\mathtt{A}}$ of at least
$\tilde{h}_{n+m}/e_{n+m-1}$ consecutive symbols in $\mathtt{A}_{i_{1},j}^{(n+m)}$
and $\mathtt{A}_{i_{2},j}^{(n+m)}$, respectively, we have
\[
\overline{f}\left(\mathtt{A},\overline{\mathtt{A}}\right)\geq\frac{1}{K}\alpha -\frac{14}{\sqrt{N}}-\frac{2}{TK}-\frac{4}{u_{n+1}}-\frac{2}{e_{n+1}},
\]
in case of $m=2$; while for $m\geq3$, we have $\overline{f}\left(\mathtt{A},\overline{\mathtt{A}}\right)$ greater than or equal to
\[
\beta_{n+m-1}+\frac{1}{K}\left(\alpha_{n+m-1}-\beta_{n+m-1}\right)-\frac{2}{e_{n+m-2}}-\frac{4}{u_{n+m-1}}-\frac{2}{N(n+m-2)+1}-\frac{2}{e_{n+m-1}}.
\]
 For the corresponding strings in the
circular system we have 
\[
\overline{f}\left(\mathcal{A}_{i_{1},j}^{(n+m)},\mathcal{A}_{i_{2},j}^{(n+m)}\right)\leq\sum_{i=1}^{m-1}\left(\frac{4}{l_{n+i}}+\frac{2}{N(n+i-1)+1}\right).
\]
\end{lem}

\begin{proof}
The statement in the odometer based system follows from the first part of Lemma \ref{lem:Distn+m}
and Lemma \ref{lem:same} (with $R=u_{n+m-1}$) because $\mathtt{A}_{i_{1},j}^{(n+m)}$
and $\mathtt{A}_{i_{2},j}^{(n+m)}$ are constructed as the same Feldman
pattern with the grouped $(n+m-1)$-blocks of different type but same pattern as building
blocks. This also yields the statement in the circular system as a direct consequence
of Lemma \ref{lem:N+mCircular}.
\end{proof}

We will also need a statement on the $\overline{f}$ distance between different
Feldman patterns in the odometer based system.

\begin{lem}[Separation of pre-$(n+m)$-blocks of different Feldman patterns]
 Let $j_{1},j_{2}\in\left\{ 1,\dots,\left(\lambda_{n+m}+1\right)K\right\} $,
$j_{1}\neq j_{2}$, and $i_{1},i_{2}\in\left\{ 1,\dots,K\right\} $.
For all strings $\mathtt{A}$ and $\overline{\mathtt{A}}$ of at least
$\tilde{h}_{n+m}/{e_{n+m-1}}$ consecutive symbols in $\mathtt{A}_{i_{1},j_1}^{(n+m)}$
and $\mathtt{A}_{i_{2},j_2}^{(n+m)}$, respectively, we have 
\[
\overline{f}\left(\mathtt{A},\overline{\mathtt{A}}\right)\geq \begin{cases}
\alpha-\frac{14}{\sqrt{N}}-\frac{2}{T}-\frac{4}{u_{n+1}}-\frac{13}{\sqrt{e_{n+1}}}, & \text{for } m=2; \\
\alpha_{n+m-1}-\frac{2}{e_{n+m-2}}-\frac{4}{u_{n+m-1}}-\frac{2}{N(n+m-2)+1}-\frac{13}{\sqrt{e_{n+m-1}}}, & \text{for } m\geq 3.
\end{cases}
\]
\end{lem}

\begin{proof}
Since we consider different Feldman pattern with the grouped $(n+m-1)$-blocks
as building blocks, we obtain from the second part of Lemma  \ref{lem:Distn+m}
and Proposition \ref{prop:OdomFeld} (with $R=u_{n+m-1}$).
\end{proof}

In the next step, we use these Feldman patterns $\mathtt{A}_{i,j}^{(n+m)}$
to define $(n+m)$-blocks for $i=1,\dots,K$ and $j=1,\dots,\lambda_{n+m}$
as in the formula at the beginning of Subsection \ref{subsec:Step} with $n+m-1$ replaced by $n+m$,
and an additional marker block 
$\mathtt{B}_{0}^{(n+m)}=\mathtt{A}_{1,\lambda_{n+m}+1}^{(n+m)}\mathtt{A}_{2,\lambda_{n+m}+1}^{(n+m)}\dots\mathtt{A}_{K,\lambda_{n+m}+1}^{(n+m)}\mathtt{a}_{n+m-1}$
with the marker segment $\mathtt{a}_{n+m-1}=\left(\mathtt{B}_{0}^{(n+m-1)}\right)^{\bar{N}(n+m-1)}$. We also
note that every $(n+m)$-block contains exactly one pattern $\mathtt{A}_{i,j}^{(n+m)}$
of each type $i\in\left\{ 1,\dots,K\right\} $ and thus it is uniform
in the $(n+m-1)$-blocks. Thus, the inductive step has been accomplished.

\subsection{Final step: Construction of $(n+m_n)$-blocks} \label{subsec:3Final}

As foreshadowed in equation (\ref{eq:betacond}) we follow the inductive construction scheme
until $\beta_{n+m_{n}}>\alpha-\frac{2}{T}-\frac{\varepsilon}{2}$ and we have constructed Feldman patterns $\mathtt{A}_{i,j}^{(n+m_{n})}$, $j=1,\dots,\left(\lambda_{n+m_{n}}+1\right)K$, $i=1,\dots,K$, 
of length $\tilde{h}_{n+m_{n}}$. 
In particular, we have 
\begin{equation}
\overline{f}\left(\mathtt{A},\overline{\mathtt{A}}\right)\geq\beta_{n+m_{n}}>\alpha-\frac{2}{T}-\frac{\varepsilon}{2}\label{eq:3distFinal}
\end{equation}
for all strings $\mathtt{A},\overline{\mathtt{A}}$ of at least $\tilde{h}_{n+m_{n}}/e_{n+m_{n}-1}$
consecutive symbols in $\mathtt{A}_{i_{1},j_1}^{(n+m_{n})}$ and $\mathtt{A}_{i_{2},j_2}^{(n+m_{n})}$,
respectively, for $i_{1}\neq i_{2}$ or $j_1\neq j_2$. On the other hand, we have 
\begin{equation}
\overline{f}\left(\mathcal{A}_{i_{1},j}^{(n+m_{n})},\mathcal{A}_{i_{2},j}^{(n+m_{n})}\right)\leq\sum_{i=1}^{m_{n}-1}\left(\frac{4}{l_{n+i}}+\frac{2}{N(n+i-1)+1}\right)\label{eq:FinalCircular}
\end{equation}
in the circular system. Moreover, we recall that for every $i\in\left\{ 1,\dots,K\right\} $
each pattern $\mathtt{A}_{i,j}^{(n+m_{n})}$, $j=1,\dots,\left(\lambda_{n+m_{n}}+1\right)K$,
contains each $(n+m_{n}-1)$-block of type $i$ exactly 
\[
\bar{N}(n+m_{n}-1)=\left(e_{n+m_{n}-1}\right)^{2\cdot\left(\lambda_{n+m_{n}}+1\right)\cdot K+3}
\]
many times. In the final step, we define $K$ many $(n+m_{n})$-blocks
as follows (with $\lambda=\lambda_{n+m_n}$): 
\begin{align*}
\mathtt{B}_{1}^{(n+m_{n})}= & \mathtt{A}_{1,1}^{(n+m_{n})}\mathtt{A}_{2,2}^{(n+m_{n})}\dots\mathtt{A}_{K,K}^{(n+m_{n})}\mathtt{A}_{1,K+1}^{(n+m_{n})}\dots \mathtt{A}_{K,2K}^{(n+m_{n})}\dots \mathtt{A}_{K,\lambda K}^{(n+m_{n})}\mathtt{a}_{n+m_n-1}\\
\mathtt{B}_{2}^{(n+m_{n})}= & \mathtt{A}_{2,1}^{(n+m_{n})}\mathtt{A}_{3,2}^{(n+m_{n})}\dots\mathtt{A}_{1,K}^{(n+m_{n})}\mathtt{A}_{2,K+1}^{(n+m_{n})}\dots \mathtt{A}_{1,2K}^{(n+m_{n})}\dots \mathtt{A}_{1,\lambda K}^{(n+m_{n})}\mathtt{a}_{n+m_n-1}\\
\vdots\; & \;\vdots\\
\mathtt{B}_{K}^{(n+m_{n})}= & \mathtt{A}_{K,1}^{(n+m_{n})}\mathtt{A}_{1,2}^{(n+m_{n})}\dots\mathtt{A}_{K-1,K}^{(n+m_{n})}\mathtt{A}_{K,K+1}^{(n+m_{n})}\dots \mathtt{A}_{K-1,2K}^{(n+m_{n})}\dots \mathtt{A}_{K-1,\lambda K}^{(n+m_{n})}\mathtt{a}_{n+m_n-1}
\end{align*}
with
\[
\mathtt{a}_{n+m_n-1}=\left(\mathtt{B}_{0}^{(n+m_{n}-1)}\right)^{\lambda_{n+m}\cdot \bar{N}(n+m_{n}-1)}.
\]

We note that each $(n+m_{n})$-block contains exactly $\lambda_{n+m_n}$ patterns
of each type. Hence, it is uniform in the $(n+m_{n}-1)$-blocks. We prove the statement in Proposition \ref{prop:cycling} on the $\overline{f}$ distance in the
odometer-based system.

\begin{proof}[Proof of part (1) in Proposition \ref{prop:cycling}]
By adding fewer than $2\tilde{h}_{n+m_n}$ symbols to each $\mathtt{B}$
and $\overline{\mathtt{B}}$ we can complete any partial pre-$(n+m_n)$-blocks at the beginning and end of $\mathtt{B}$
and $\overline{\mathtt{B}}$. This change increases the $\overline{f}$ distance between $\mathtt{B}$
and $\overline{\mathtt{B}}$ by at most 
\[
\frac{4\tilde{h}_{n+m_n}}{2\mathtt{h}_{n+m_{n}}/K}<\frac{2\tilde{h}_{n+m_n}}{\lambda_{n+m_n}\tilde{h}_{n+m_n}}=\frac{2}{\lambda_{n+m_n}}.
\]
In the next step, we ignore the marker segment $\mathtt{a}_{n+m_{n}-1}$ which occupies
a fraction of $1/(N(n+m_{n}-1)+1)$ in each
$(n+m_{n})$-block due to uniformity and so a fraction of at most
$K(N(n+m_{n}-1)+1)^{-1}<e_{n+m_{n}-1}^{-1}$ of the total length of $\mathtt{B}$
and $\overline{\mathtt{B}}$. On the remaining strings all pre-$(n+m_n)$-blocks are different from each other. Hence, they are at least $\alpha-\frac{2}{T}-\frac{\varepsilon}{2}$ apart in $\overline{f}$ on substantial substrings of at least $\tilde{h}_{n+m_{n}}/e_{n+m_{n}-1}$ consecutive symbols by equation (\ref{eq:3distFinal}). We apply Corollary \ref{cor:alpha separated blocks} with $\tilde{f}=1$ to obtain
\[
\overline{f}\left(\mathtt{B},\overline{\mathtt{B}}\right)\geq\alpha-\frac{2}{T}-\frac{\varepsilon}{2} - \frac{4}{e_{n+m_{n}-1}}-\frac{2}{\lambda_{n+m_n}},
\]
which yields the claim.
\end{proof}

By choosing the circular coefficients $\left(l_{n+m}\right)_{m\in\mathbb{N}}$
to grow sufficiently fast as in equation (\ref{eq:l-1}) we also obtain
the second statement in Proposition \ref{prop:cycling}.

\begin{proof}[Proof of part (2) in Proposition \ref{prop:cycling}]
Since the marker segments and the used Feldman
patterns of the pre-$(n+m_{n})$-blocks are aligned and only the used
type of blocks differs, we use equation (\ref{eq:FinalCircular}) to
obtain 
\[
\overline{f}\left(\mathcal{B}_{i}^{(n+m_{n})},\mathcal{B}_{j}^{(n+m_{n})}\right) \leq\sum_{i=1}^{m_{n}-1}\frac{4}{l_{n+i}}+\sum_{i=1}^{m_{n}-1}\frac{2}{Ku_{n+i-1}^{2}e_{n+i-1}+1} <\delta
 \]
with the aid of assumptions (\ref{eq:Delta-1}) and (\ref{eq:l-1})
in the last step.
\end{proof}
Hence the proof of Proposition \ref{prop:cycling} has been accomplished.

\section{Proof of Theorem \ref{thm:converse}} \label{sec:ConverseProof}

To define the construction sequence for the odometer-based system
inductively we choose an increasing
sequence $\left(K_s\right)_{s\in\mathbb{N}}$ of positive integers
with 
\begin{equation}
\sum_{k\in\mathbb{N}}\frac{15}{\sqrt{K_s}}<\frac{1}{32}\label{eq:3M}
\end{equation}
and two decreasing sequences $\left(\varepsilon_{s}\right)_{s\in\mathbb{N}}$
and $\left(\delta_{s}\right)_{s\in\mathbb{N}}$ of positive real numbers
such that $\delta_{s}\searrow0$, $\varepsilon_s <1/(64K_s)$, and
\begin{equation}
\sum_{s\in\mathbb{N}}\varepsilon_{s}<\frac{1}{32}.\label{eq:3eps}
\end{equation}
We start by applying Proposition \ref{prop:OdomFeldMech} on $K_0+1$
symbols to obtain $N(1)+1$ many uniform and uniquely readable $1$-blocks
that are $\alpha_{1}\coloneqq(1/8)-13K_0^{-1/2}$ apart in $\overline{f}$ on
substantial substrings of length at least $\mathtt{h}_1/K_0$. Here, the number $N(1)+1$ is chosen such that
it allows the application of the cycling mechanism from Proposition
\ref{prop:cycling} to obtain $K_{1}+1$ many $n_{1}$-blocks (where
$n_{1}\coloneqq 1+m_{1}$ with $m_1$ from Proposition \ref{prop:cycling}) that are $\alpha_{2}=\alpha_{1}-2K_0^{-1}-\varepsilon_{0}$
apart on substantial subshifts of length at least $\mathtt{h}_{n_{1}}/(K_{1}+1)$
in the odometer-based system and are $\delta_{0}$-close in the corresponding
circular system. 

We continue by applying Proposition \ref{prop:OdomFeldMech} on those
blocks to obtain sufficiently many $\left(n_{1}+1\right)$-blocks
(which are $\alpha_{3}=\alpha_{2}-13K_{1}^{-1/2}$ apart on substantial
substrings of length at least $\mathtt{h}_{n_1+1}/K_1$ in the odometer-based system) such that we can apply Proposition
\ref{prop:cycling} again to get $K_{2}+1$ many $n_{2}$-blocks (with
$n_{2}\coloneqq n_{1}+1+m_{n_{1}+1}$) that are $\alpha_{4}=\alpha_{3}-2K_{1}^{-1}-\varepsilon_{1}$
apart on substantial subshifts of length at least $\mathtt{h}_{n_{2}}/(K_{2}+1)$
in the odometer-based system and $\delta_{1}$-close in the circular
system.

Continuing like this we produce $n$-blocks that are at least 
\[
\frac{1}{8}-\sum_{s\in\mathbb{N}}\left(\frac{15}{\sqrt{K_{s}}}+\varepsilon_{s}\right)>\frac{1}{16}
\]
apart from each other by the requirements (\ref{eq:3M}) and (\ref{eq:3eps}).
Hence, the odometer-based system cannot be loosely Bernoulli by Lemma \ref{lem:Rothstein}.

On the other hand, the $\overline{f}$ distance between $n$-blocks in the circular
system goes to zero because $\delta_{s}\searrow0$. Thus, the odometer-based
system is loosely Bernoulli by Lemma \ref{lem:Rothstein}. Since the blocks constructed by Propositions
\ref{prop:OdomFeldMech} and \ref{prop:cycling} satisfy the properties
of uniformity and unique readability, our construction sequence satisfies
those as well. 

\subsection*{Acknowledgements}

We thank M. Foreman and J.-P. Thouvenot for helpful conversations.
In particular, we thank Foreman for his patient explanations of his
work with B. Weiss during his visit to Indiana University. We are
also grateful for Thouvenot's suggestion to use Rothstein's argument
in our positive entropy example. 


\begin{thebibliography}{ORW82}

\bibitem[AK70]{AK} D.\ Anosov and A.\ Katok: New examples in smooth ergodic theory.
Ergodic diffeomorphisms. Trudy Moskov. Mat. Obsc., 23: 3-36, 1970.

\bibitem[BF96]{BF} F.\ Beleznay and M.\ Foreman: The complexity of the collection
of measure-distal transformations. Ergodic Theory Dynam. Systems, 16(5): 929 - 962, 1996.

\bibitem[Fe76]{Fe} J.\ Feldman: New K-automorphisms and a problem
of Kakutani. Israel Journal of Mathematics, 24 (1): 16 - 38, 1976.

\bibitem[FRW11]{FRW} M.\ Foreman, D.\ Rudolph, and B.\ Weiss: The conjugacy
problem in ergodic theory. Ann. of Math. (2), 173(3): 1529 - 1586, 2011.

\bibitem[FW19]{FW1} M.\ Foreman and B.\ Weiss: A symbolic representation
of Anosov-Katok systems. Journal d'Analyse Mathematique, 137(2): 603 - 661, 2019.

\bibitem[FW2]{FW2} M.\ Foreman and B.\ Weiss: From Odometers to Circular
Systems: A Global Structure Theorem. Journal of Modern Dynamics, 15: 345 - 423, 2019.

\bibitem[FW3]{FW3} M.\ Foreman and B.\ Weiss: Measure preserving Diffeomorphisms
of the Torus are unclassifiable. Preprint, arXiv:1705.04414.

\bibitem[HN42]{HN} P.\ Halmos and J.\ von Neumann: Operator methods in classical
mechanics. II. Ann. of Math. (2), 43: 332 - 350, 1942.

\bibitem[Hj01]{H} G.\ Hjorth: On invariants for measure preserving transformations. Fund.\ Math., 169(1): 51 - 84, 2001.

\bibitem[Ka75]{K75} A.\ Katok: Time change, monotone equivalence, and standard dynamical systems. Dokl. Akad. Nauk SSSR 223: 784-792, 1975.

\bibitem[Ka77]{K77} A.\ Katok: Monotone equivalence in ergodic theory. Mathematics of the USSR-Izvestiya 11 (1): 99 - 146, 1977.

\bibitem[KH95]{KH95} A.\ Katok and B.\ Hasselblatt: Introduction to the modern theory of dynamical systems. Encyclopedia of Mathematics and its Applications, 54. Cambridge University Press, Cambridge, 1995. 

\bibitem[KR]{KR} A.\ Kanigowski and T.\ De La Rue: Product of two staircase rank one transformations that is not loosely Bernoulli. Preprint, arXiv:1812.08027.

\bibitem[Ku65]{Ku} A.\ Kushnirenko: An upperbound for the entropy of a classical dynamical system. (Russian)
Dokl. Akad. Nauk SSSR 161: 37 - 38, 1965. 

\bibitem[KW19]{KW} A.\ Kanigowski and D.\ Wei: Product of two Kochergin flows with different exponents is not standard. Studia Mathematica 244: 265 - 283, 2019.

\bibitem[Ne32]{Ne} J.\ von Neumann: Zur Operatorenmethode in der klassischen Mechanik.
Ann. of Math. (2), 33(3): 587 - 642, 1932.

\bibitem[Or70]{Or} D. Ornstein: Bernoulli shifts with the same entropy are isomorphic. Advances in Math. 4: 337 - 352, 1970.

\bibitem[Or]{Orbook} D.\ Ornstein: Ergodic Theory, Randomness, and Dynamical Systems. Yale University Press, New Haven, 1970.

\bibitem[ORW82]{ORW} D.\ Ornstein, D.\ Rudolph and B.\ Weiss: Equivalence
of measure preserving transformations. Mem. American Mathematical
Soc. 37, Vol. 262, 1982.

\bibitem[Ra78]{R1} M.\ Ratner: Horocycle flows are loosely Bernoulli. Israel J. Math., 31 (2): 122 - 132, 1978.

\bibitem[Ra79]{R2} M.\ Ratner: The Cartesian square of the horocycle flow is not loosely Bernoulli. Israel J. Math., 34 (1): 72 - 96, 1979.

\bibitem[Ro80]{R} A.\ Rothstein: Vershik processes: First steps. Israel J. Math. 36 (3-4): 205 - 224, 1980.

\bibitem[Si62]{S} Ya.\ Sinai: A weak isomorphism of transformations with invariant measure. Dokl. Akad. Nauk SSSR 147: 797 - 800, 1962.

\bibitem[Sm71]{Sm} M.\ Smorodinsky: Ergodic theory, entropy. Lecture Notes in Mathematics, 214, Springer, 1971.

\end{thebibliography}
\end{document}